\documentclass[11pt,twoside,english]{article}
\usepackage[T1]{fontenc}
\usepackage[utf8]{inputenc}
\usepackage{geometry}
\usepackage{color}
\usepackage{babel}
\usepackage{array}
\usepackage{units}
\usepackage{multirow}
\usepackage{amsmath}
\usepackage{amsthm}
\usepackage{amssymb}
\usepackage{stackrel}
\usepackage{graphicx}
\usepackage{esint}
\usepackage[all]{xy}
\usepackage[unicode=true,pdfusetitle,
 bookmarks=true,bookmarksnumbered=false,bookmarksopen=false,
 breaklinks=false,pdfborder={0 0 1},backref=page,colorlinks=true]
 {hyperref}
\hypersetup{
 linkcolor=blue,citecolor=blue}

\makeatletter

\providecommand{\tabularnewline}{\\}

\numberwithin{equation}{section}
\numberwithin{figure}{section}
\theoremstyle{plain}
\newtheorem{thm}{\protect\theoremname}[section]
\theoremstyle{definition}
\newtheorem{defn}[thm]{\protect\definitionname}
\theoremstyle{plain}
\newtheorem{conjecture}[thm]{\protect\conjecturename}
\theoremstyle{plain}
\newtheorem{cor}[thm]{\protect\corollaryname}
\theoremstyle{plain}
\newtheorem{prop}[thm]{\protect\propositionname}
\theoremstyle{remark}
\newtheorem{rem}[thm]{\protect\remarkname}
\theoremstyle{plain}
\newtheorem{lem}[thm]{\protect\lemmaname}

\usepackage{babel}

\usepackage{ytableau}

\usepackage{multirow}

\usepackage{appendix}

\makeatother

\providecommand{\conjecturename}{Conjecture}
\providecommand{\corollaryname}{Corollary}
\providecommand{\definitionname}{Definition}
\providecommand{\lemmaname}{Lemma}
\providecommand{\propositionname}{Proposition}
\providecommand{\remarkname}{Remark}
\providecommand{\theoremname}{Theorem}

\begin{document}
\global\long\def\F{\mathbb{\mathbb{\mathbf{F}}}}%
 
\global\long\def\rk{\mathbb{\mathrm{rk}}}%
 
\global\long\def\crit{\mathbb{\mathrm{Crit}}}%
 
\global\long\def\Hom{\mathrm{Hom}}%
 
\global\long\def\defi{\stackrel{\mathrm{def}}{=}}%
 
\global\long\def\tr{{\cal T}r }%
 
\global\long\def\id{\mathrm{id}}%
 
\global\long\def\Aut{\mathrm{Aut}}%
 
\global\long\def\wl{w_{1},\ldots,w_{\ell}}%
 
\global\long\def\alg{\le_{\mathrm{alg}}}%
 
\global\long\def\ff{\stackrel{*}{\le}}%
 
\global\long\def\A{{\cal A}}%
 
\global\long\def\mobius{M\dacute{o}bius}%
 
\global\long\def\chimax{\chi^{\mathrm{max}}}%
 
\global\long\def\uexp{\mathbb{E}_{\mathrm{unif}}}%
 
\global\long\def\symirr{\widehat{S_{\infty}}}%
 
\global\long\def\supp{\mathrm{supp}}%
  
\global\long\def\suppi{\widetilde{\mathrm{supp}}}%
 
\global\long\def\fq{\mathbb{F}_{q}}%
 
\global\long\def\IT{I|_{T}}%
 
\global\long\def\gl{\mathrm{GL}}%
 
\global\long\def\glm{\mathrm{GL}_{m}\left(\k\right)}%
 
\global\long\def\gln{\mathrm{GL}_{N}\left(\k\right)}%
 
\global\long\def\T{\mathbb{T}}%
 
\global\long\def\k{K}%
 
\global\long\def\fix{\mathrm{fix}}%
 
\global\long\def\pa{\pi_{\A}}%
 
\global\long\def\R{{\cal R}}%
 
\global\long\def\indep{\mathfrak{indep}}%
 
\global\long\def\B{{\cal B}}%
 
\global\long\def\btil{\tilde{{\cal B}}}%
 
\global\long\def\eqb{\mathrm{EQ}_{\B,w}}%
 
\global\long\def\free{\mathfrak{Free}}%
 
\global\long\def\C{{\cal C}}%
 
\global\long\def\lm{\mathfrak{lm}}%
 
\global\long\def\spn{\mathrm{span}}%
 
\global\long\def\sw{{\cal S}_{w}}%
 
\global\long\def\fr{\mathrm{full-rank}}%

\title{Word Measures on $\mathrm{GL_{N}\left(q\right)}$ and Free Group Algebras}
\author{Danielle Ernst-West~~~~~~Doron Puder~~~~~~Matan Seidel}
\maketitle
\begin{abstract}
Fix a finite field $\k$ of order $q$ and a word $w$ in a free group
$\F$ on $r$ generators. A $w$-random element in $\gln$ is obtained
by sampling $r$ independent uniformly random elements $g_{1},\ldots,g_{r}\in\gln$
and evaluating $w\left(g_{1},\ldots,g_{r}\right)$. Consider $\mathbb{E}_{w}\left[\fix\right]$,
the average number of vectors in $\k^{N}$ fixed by a $w$-random
element. We show that $\mathbb{E}_{w}\left[\fix\right]$ is a rational
function in $q^{N}$. Moreover, if $w=u^{d}$ with $u$ a non-power,
then the limit $\lim_{N\to\infty}\mathbb{E}_{w}\left[\fix\right]$
depends only on $d$ and not on $u$. These two phenomena generalize
to all stable characters of the groups $\left\{ \gln\right\} _{N}$. 

A main feature of this work is the connection we establish between
word measures on $\gln$ and the free group algebra $\k\left[\F\right]$.
A classical result of Cohn and Lewin \cite{cohn1964free,lewin1969free}
is that every one-sided ideal of $\k\left[\F\right]$ is a free $\k\left[\F\right]$-module
with a well-defined rank. We show that for $w$ a non-power, $\mathbb{E}_{w}\left[\fix\right]=2+\frac{C}{q^{N}}+O\left(\frac{1}{q^{2N}}\right)$,
where $C$ is the number of rank-2 right ideals $I\le\k\left[\F\right]$
which contain $w-1$ but not as a basis element. We describe a full
conjectural picture generalizing this result, featuring a new invariant
we call the $q$-primitivity rank of $w$.

In the process, we prove several new results about free group algebras.
For example, we show that if $T$ is any finite subtree of the Cayley
graph of $\F$, and $I\le\k\left[\F\right]$ is a right ideal with
a generating set supported on $T$, then $I$ admits a basis supported
on $T$. We also prove an analogue of Kaplansky's unit conjecture
for certain $\k\left[\F\right]$-modules. 
\end{abstract}
\tableofcontents{}

\section{Introduction\label{sec:Introduction}}

Fix $r\in\mathbb{Z}_{\ge1}$. We let $\F$ denote the free group on
$r$\marginpar{$r$} generators. A word $w\in\F$ induces a map on
any finite group, $w:G^{r}\to G$, by substituting the letters of
$w$ with elements of $G$. This map defines a distribution on the
group $G$: the pushforward of the uniform distribution on $G^{r}$.
Equivalently, this distribution is the normalized number of times
each element in $G$ is obtained by a substitution in $w$. We call
such a distribution a \emph{word measure} on $G$, and if $w$ is
given, \emph{the $w$-measure on $G$}. For example, if $w=abab^{-2}$,
a $w$-random element in $G$ is $ghgh^{-2}$ where $g,h$ are independent,
uniformly random elements of $G$.

The study of word measures on various families of groups revealed
structural depth with surprising connections to objects in combinatorial
and geometric group theory (see, e.g.~\cite{Puder2014,PP15,MP-Un,MP-On,hanany2020word,MP-surface-words}).
It has proven useful for many questions regarding free groups and
their automorphism groups (see, e.g., \cite{PP15,hanany2020some}),
as well as for questions about random Schreier graphs and their expansion
(see, e.g., \cite{Puder2015,hanany2020word}). Previous works in the
subject study word measures on the groups $\mathrm{Sym}\left(N\right)$,
$U\left(N\right)$, $O\left(N\right)$, $\mathrm{Sp\left(N\right)}$
and generalized symmetric groups. Section \ref{subsec:Related-works}
explains how some of the results in the current paper relate to the
established structure in other families of groups.

In this paper we focus on word measures on $\gln$, the general linear
group over a fixed finite field $\k$\marginpar{$K,q$} of order $q$.
As seen in other families of groups, word measures on this family
demonstrate structural depth. Most interestingly, we show that the
analysis of word measures on $\gln$ is intertwined with the theory
of free group algebras.

\subsection{The average number of fixed vectors}

We consider various families of real- or complex-valued functions
defined on $\gln$, and study their expected value under word measures.
Our core example is the function \marginpar{$\protect\fix$}$\fix\colon\gln\to\mathbb{Z}_{\ge0}$
counting elements in the vector space $V=\k^{N}$ which are fixed
by a given matrix in $\gln$. Not only does this special case illustrate
our more general results, but is also a case in which our understanding
goes deeper. Note that the function $\fix$ is, in fact, a family
of functions, one for every value of $N\in\mathbb{Z}_{\ge1}$. We
let $\mathbb{E}_{w}\left[\fix\right]$ denote the expected value of
$\fix$ under the $w$-measure on $\gln$, so $\mathbb{E}_{w}\left[\fix\right]$
is also a sequence of numbers, one for every value of $N\in\mathbb{Z}_{\ge1}$.
Our first result is the following.
\begin{thm}
\label{thm:rational expression fix}For every $w\in\F$ and every
large enough $N$, $\mathbb{E}_{w}\left[\fix\right]$ is given by
a rational function in $q^{N}$ with rational coefficients.
\end{thm}

For example, if $w=\left[a,b\right]=aba^{-1}b^{-1}$  is the commutator
of two basis elements, then 
\[
\mathbb{E}_{w}\left[\fix\right]=2+\frac{\left(q-1\right)^{2}q^{N}-\left(q-1\right)^{3}}{\left(q^{N}-1\right)\left(q^{N}-q\right)}
\]
for every $N\ge2$ (recall that $q=\left|\k\right|$ is fixed throughout,
so this expression is indeed a rational function in $q^{N}$ with
coefficients in $\mathbb{Q}$). Consult Table \ref{tab:The-rational-expression}
for further examples. For general words, the rational expression is
valid for every $N\ge\left|w\right|$. See Section \ref{sec:Rational-expressions}
for a tighter lower bound on $N$. Theorem \ref{thm:rational expression fix}
is a special case of Theorem \ref{thm:rational general}.

\begin{table}
\begin{centering}
\begin{tabular}{|c|c|c|c|}
\hline 
$w$ & $q$ & $\mathbb{E}_{w}\left[\fix\right]$ & valid for\tabularnewline
\hline 
\hline 
$a$ & every $q$ & $2$ & $N\ge1$\tabularnewline
\hline 
\hline 
\multirow{2}{*}{$a^{2}$} & $q$~even & $3$ & \multirow{2}{*}{$N\ge2$}\tabularnewline
\cline{2-3} \cline{3-3} 
 & $q$~odd & $4$ & \tabularnewline
\hline 
\hline 
\multirow{2}{*}{$a^{3}$} & $q\equiv0,2\pmod3$ & $4$ & \multirow{2}{*}{$N\ge3$}\tabularnewline
\cline{2-3} \cline{3-3} 
 & $q\equiv1\pmod3$ & $8$ & \tabularnewline
\hline 
\hline 
$\left[a,b\right]$ & every $q$ & $2+\frac{\left(q-1\right)^{2}q^{N}-\left(q-1\right)^{3}}{\left(q^{N}-1\right)\left(q^{N}-q\right)}$ & $N\ge2$\tabularnewline
\hline 
\hline 
\multirow{2}{*}{$a^{2}b^{3}$} & $q=2$ & $2+\frac{2}{2^{N}-2}$ & \multirow{2}{*}{$N\ge3$}\tabularnewline
\cline{2-3} \cline{3-3} 
 & $q=3$ & $2+\frac{4}{3^{N}-3}$ & \tabularnewline
\hline 
\hline 
$\left[a,b\right]^{2}$ & $q=2$ & $3+\frac{2\left(2^{2N}-9\cdot2^{N}+26\right)}{\left(2^{N}-1\right)\left(2^{N}-2\right)\left(2^{N}-8\right)}$ & $N\ge4$\tabularnewline
\hline 
\hline 
\multirow{2}{*}{$a^{2}b^{2}c^{2}$} & $q=2$ & $2+\frac{1}{\left(2^{N}-2\right)^{2}}$ & \multirow{2}{*}{$N\ge2$}\tabularnewline
\cline{2-3} \cline{3-3} 
 & $q=3$ & $2+\frac{8\left(3^{2N}-4\cdot3^{N}+5\right)}{\left(3^{N}-1\right)^{2}\left(3^{N}-3\right)^{2}}$ & \tabularnewline
\hline 
\end{tabular}
\par\end{centering}
\caption{The rational expressions giving $\mathbb{E}_{w}\left[\protect\fix\right]$
for various words $w\in\protect\F\left(a,b,c\right)$ and various
values of $q=\left|\protect\k\right|$. For the first four words,
rational expressions are given for all values of $q$. For the remaining
three words, rational expressions are given only for particular values
of $q$.\label{tab:The-rational-expression}}
\end{table}

Our second result alludes to a result of Nica \cite{nica1994number}.
Let $1\ne w=u^{d}$ where $d\in\mathbb{N}_{\ge1}$ and $u$ a non-power.
Nica proved, inter alia, that the distribution of the number of fixed
points in a $w$-random permutation in $\mathrm{Sym}\left(N\right)$
has a limit distribution as $N\to\infty$ which depends solely on
$d$ and not on $u$. A similar phenomenon was later shown to hold
in various other families of groups. We add the groups $\left\{ \gln\right\} _{N}$
as such a family. In our illustrative special case, this is captured
by the following result, which first appeared in \cite{West19}. It
also appeared independently in \cite[Sec.~8]{eberhard2021babai}.
\begin{thm}
\label{thm:limit powers fix}Let $1\ne w=u^{d}$ with $d\ge1$ and
$u$ a non-power. Then 
\begin{equation}
\lim_{N\to\infty}\mathbb{E}_{w}\left[\fix\right]=\#\left\{ p\in K\left[x\right]\,\middle|\,p\mid x^{d}-1~\mathrm{and}~p~\mathrm{monic}\right\} .\label{eq:limit of fix}
\end{equation}
In particular, the limit does not depend on $u$.
\end{thm}

Combined with Theorem \ref{thm:rational expression fix}, if $c_{d}$
is the number of monic divisors of $x^{d}-1\in\k\left[x\right]$,
we get that $\mathbb{E}_{w}\left[\fix\right]=c_{d}+O\left(\frac{1}{q^{N}}\right)$.
In particular, for non-powers, \textbf{$\mathbb{E}_{w}\left[\fix\right]=2+O\left(\frac{1}{q^{N}}\right)$},
and for proper powers\textbf{ $c_{d}\ge3$} (if $d\ge2$, then $x^{d}-1$
admits at least three distinct monic divisors: $1,x-1$ and $x^{d}-1$).
Theorem \ref{thm:limit powers fix} is analogous to the result in
the symmetric group $\mathrm{Sym}\left(N\right)$, where this limit
is equal to the number of positive divisors of $d$ in $\mathbb{Z}$
\cite{nica1994number}. In fact, it is sufficient to prove that the
limit in \eqref{eq:limit of fix} depends only on $d$ and not on
$u$, and then the left-hand side of \eqref{eq:limit of fix} is equal
to $\lim_{N\to\infty}\mathbb{E}_{a^{d}}\left[\fix\right]$. This number
can then be extracted from the analysis of uniformly random elements
in $\gln$ -- see, for example, \cite{fulman2016distribution} and
the references therein. Theorem \ref{thm:limit powers fix} is a special
case of the more general Theorem \ref{thm:limit powers general} below. 

\subsection{The $q$-primitivity rank}

The analysis of $\mathbb{E}_{w}\left[\fix\right]$, yielding Theorems
\ref{thm:rational expression fix} and \ref{thm:limit powers fix},
can be performed using elementary linear algebraic arguments. In fact,
this is how they were first derived in \cite{West19}. However, it
turns out to be extremely useful to analyze these quantities using
the theory of free group algebras. 

Denote by $\A\defi\k\left[\F\right]$\marginpar{$\protect\A$} the
free group algebra over $\k$: its elements are finite linear combinations
of elements of the free group $\F$ with coefficients from the finite
field $\k$. It is a classical result of Cohn \cite{cohn1964free}
and Lewin\footnote{\label{fn:Cohn Vs. Lewin}Some claim that the first correct proof
of this result (stated formally below as Theorem \ref{thm:Cohn-Lewin})
is due to Lewin in \cite{lewin1969free} -- see \cite[Footnote 5]{hog1990short}.} \cite{lewin1969free} that right ideals of $\A$ are free right $\A$-modules
with a well-defined rank.\footnote{For example, it can be shown that the augmentation ideal $I_{\F}=\left\{ \sum\alpha_{w}w\,\middle|\,\sum\alpha_{w}=0\right\} \subseteq\A$
is of rank $r=\rk\F$. For instance, when $\F=\F\left(a,b,c\right)$,
$I_{\F}=\left(a-1\right)\A\oplus\left(b-1\right)\A\oplus\left(c-1\right)\A$.} An analogous result holds for left ideals, but here we use right
ideals only -- in fact, from now on, we write ``ideals'' to mean
``right ideals''. In Section \ref{sec:Rational-expressions} below
we derive a formula for $\mathbb{E}_{w}\left[\fix\right]$ as a sum
over a finite set of finitely generated ideals of $\A$, and Section
\ref{sec:The-free-group-algebra} shows that the contribution of every
such ideal is of order determined by its rank. 

In particular, this algebraic perspective allows a further understanding
of the deviation of $\mathbb{E}_{w}\left[\fix\right]$ from $\mathbb{E}_{a}\left[\fix\right]$,
the analogous expectation under the uniform measure. Namely, as the
action $\gln\curvearrowright\k^{N}$ admits two orbits (the zero vector
and all non-zero vectors), the expected number of vectors in $K^{N}$
fixed by a uniformly random element of $\gln$ is $\mathbb{E}_{a}\left[\fix\right]=2$,
and we consider the difference $\mathbb{E}_{w}\left[\fix\right]-2$.
Theorems \ref{thm:rational expression fix} and \ref{thm:limit powers fix}
imply that if $w$ is a proper power, then $\mathbb{E}_{w}\left[\fix\right]-2$
is of order $\Theta\left(1\right)$, and otherwise, it is of order
$O\left(\frac{1}{q^{N}}\right)$. Next, we provide a more refined
and accurate description of this difference in the non-power case.
To state our result and conjecture, we first define the notion of
primitivity of elements in ideals. Recall that by Cohn and Lewin's
result, every ideal $I\le\A$ is a free $\A$-module and so admits
a basis. Moreover, all bases of $I$ have the same cardinality, called
the rank of $I$ and denoted $\rk I$\marginpar{$\protect\rk I$}.
\begin{defn}
\label{def:primitive}Let $I\le\A$ be an ideal and let $f\in I$.
We say that $f$ is a \textbf{primitive} element of $I$ if it is
contained in some basis of $I$ (considering $I$ as a free right
$\A$-module). Otherwise, $f$ is \textbf{imprimitive} in $I$.
\end{defn}

This is analogous to the notion of a primitive element in a free group:
an element belonging to some basis of this group. Our next central
result captures the $\frac{1}{q^{N}}$-term of the Laurent expansion
of $\mathbb{E}_{w}\left[\fix\right]$.
\begin{thm}
\label{thm:fixed vectors in pi=00003D2}Let $1\ne w\in\F$ be a non-power.
Then the expected number of vectors in $\k^{N}$ fixed by a $w$-random
element of $\gln$ is

\[
\mathbb{E}_{w}\left[\fix\right]=2+\frac{\left|\crit_{q}^{2}\left(w\right)\right|}{q^{N}}+O\left(\frac{1}{q^{2N}}\right),
\]
where $\crit_{q}^{2}\left(w\right)$ is the set of ideals $I\le\A$
of rank two which contain the element $w-1$ as an imprimitive element. 
\end{thm}

As implied by the theorem, the set $\crit_{q}^{2}\left(w\right)$
is indeed finite for every non-power $w$. We prove this fact directly
in Corollary \ref{cor:finitely many critical extensions}. To illustrate,
consider the commutator word $w=\left[a,b\right]$. As mentioned above,
\[
\mathbb{E}_{\left[a,b\right]}\left[\fix\right]=2+\frac{\left(q-1\right)^{2}q^{N}-\left(q-1\right)^{3}}{\left(q^{N}-1\right)\left(q^{N}-q\right)}=2+\frac{\left(q-1\right)^{2}}{q^{N}}+O\left(\frac{1}{q^{2N}}\right).
\]
In this case there are exactly $\left(q-1\right)^{2}$ distinct ideals
of rank two containing $\left[a,b\right]-1$ as an imprimitive element:
these are $\left(\delta a-1,\varepsilon b-1\right)$ with $\delta,\varepsilon\in K^{*}$.
We conjecture a more general phenomenon, for which we make the following
definition. 
\begin{defn}
\label{def:q-primitivity-rank}The \textbf{$q$-primitivity rank}
of $w\in\F$, denoted $\pi_{q}\left(w\right)$, is the smallest rank
of a \emph{proper} ideal of $\A$ containing $w-1$ as an imprimitive
element. Namely,
\[
\pi_{q}\left(w\right)\defi\min\left\{ \rk I\,\middle|\,\begin{gathered}I\lvertneqq\A,I\ni w-1,\mathrm{and}\\
w-1~\mathrm{is~imprimitive~in}~I
\end{gathered}
\right\} .
\]
If this set is empty, we set $\pi_{q}\left(w\right)=\infty$. A \emph{critical}
ideal for $w$ is a proper ideal of rank $\pi_{q}\left(w\right)$
containing $w-1$ as an imprimitive element. We denote by $\crit_{q}\left(w\right)$
the set of critical ideals for $w$. 
\end{defn}

Corollary \ref{cor:values of pi_q} shows that $\pi_{q}\left(w\right)$
takes values only in $\left\{ 0,1,\ldots,r\right\} \cup\left\{ \infty\right\} $,
where $r$ is the rank of $\F$. Note that $\pi_{q}\left(w\right)=0$
if and only if $w=1$: the only rank-$0$ ideal is $\left(0\right)$,
whose only basis is the empty set. In Section \ref{subsec:powers fix}
below, we prove that $\pi_{q}\left(w\right)=1$ if and only if $w\in\F$
is a proper power (Corollary \ref{cor:pa=00003D1 iff power}), and
that in this case, if one writes $w=u^{d}$ with $d\ge2$ and $u$
a non-power, the set of critical ideals of $w$ is
\[
\crit_{q}\left(u^{d}\right)=\left\{ \left(p\left(u\right)\right)\,\middle|\,p\mid x^{d}-1\in K\left[x\right],~p~\mathrm{monic~and}~p\ne1,x^{d}-1\right\} .
\]
For example, if $\left|\k\right|=q=3$ and $w=u^{4}$, the critical
ideals of $w$ are in one-to-one correspondence with the six non-trivial
monic divisors the polynomial $x^{4}-1\in\k\left[x\right]$. These
rank-$1$ ideals are $\left(u-1\right)$, $\left(u+1\right)$, $\left(u^{2}-1\right)$,
$\left(u^{2}+1\right)$, $\left(u^{3}-u^{2}+u-1\right)$ and $\left(u^{3}+u^{2}+u+1\right)$.
Note that the trivial monic divisors of $x^{4}-1$ correspond to the
ideal $\left(1\right)=\A$ which is not proper, and to the ideal $\left(u^{4}-1\right)$
in which $w-1$ is primitive. By Proposition \ref{prop:prim iff pi_q=00003Dinfty},
$\pi_{q}\left(w\right)=\infty$ if and only if $w$ is a primitive
element of $\F$.

The following conjecture thus generalizes Theorems \ref{thm:limit powers fix}
and \ref{thm:fixed vectors in pi=00003D2}.
\begin{conjecture}
\label{conj:general pi and fixed vectors}Let $w\in\F$ and denote
$\pi=\pi_{q}\left(w\right)$. Then the expected number of vectors
in $\k^{N}$ fixed by a $w$-random element of $\gln$ is
\begin{equation}
\mathbb{E}_{w}\left[\fix\right]=2+\frac{\left|\crit_{q}\left(w\right)\right|}{q^{N\cdot\left(\pi-1\right)}}+O\left(\frac{1}{q^{N\cdot\pi}}\right).\label{eq:conj general pi and fixed vectors}
\end{equation}
\end{conjecture}

Corollary \ref{cor:finitely many critical extensions} yields that
$\crit_{q}\left(w\right)$ is indeed finite. Note that if $\pi:=\pi_{q}\left(w\right)=0$
(namely, if $w=1$), then $\crit_{q}\left(w\right)=\left\{ \left(0\right)\right\} $
and \eqref{eq:conj general pi and fixed vectors} is obvious. Theorem
\ref{thm:limit powers fix} proves \eqref{eq:conj general pi and fixed vectors}
when $\pi=1$, and Theorem \ref{thm:fixed vectors in pi=00003D2}
proves it when $\pi=2$. As mentioned above, $\pi_{q}\left(w\right)=\infty$
if and only if $w$ is primitive in $\F$, and in this case a $w$-random
element of $\gln$ distributes uniformly \cite[Obs.~1.2]{PP15}, and
so \eqref{eq:conj general pi and fixed vectors} holds. In particular,
Conjecture \ref{conj:general pi and fixed vectors} holds for the
free group of rank $2$ as the possible values of $\pi_{q}\left(w\right)$
are $\left\{ 0,1,2,\infty\right\} $ (Corollary \ref{cor:values of pi_q}).
We conclude the following analogue of a result about $S_{N}$ \cite[Thm.~1.5]{Puder2014}.
\begin{cor}
\label{cor:F2}Let $w\in\F_{2}$. Then $w$ induces the uniform measure
on $\gln$ for all $N$ if and only if $w$ is primitive.
\end{cor}

Another important background for Conjecture \ref{conj:general pi and fixed vectors}
is an analogous result in the case of the symmetric group $S_{N}$.
The primitivity rank of a word $w\in\F$, denoted $\pi\left(w\right)$
and introduced in \cite{Puder2014}, is the smallest rank of a subgroup
of $\F$ containing $w$ as an imprimitive element. Let $\crit_{\F}\left(w\right)$
denote the set of subgroups of $\F$ of rank $\pi\left(w\right)$
which contain $w$ as an imprimitive element. Then the $S_{N}$-analogue
of Conjecture \ref{conj:general pi and fixed vectors} is \cite[Thm.~1.8]{PP15}:
the expected number of fixed points in a $w$-random permutation in
$S_{N}$ is
\[
1+\frac{\left|\crit_{\F}\left(w\right)\right|}{N^{\pi\left(w\right)-1}}+O\left(\frac{1}{N^{\pi\left(w\right)}}\right).
\]

Alongside its role in word measures on $S_{N}$, the original primitivity
rank $\pi\left(w\right)$ seems to play a universal role in word measures
on groups (see \cite[Conj.~1.13]{hanany2020word}), it has connections
with stable commutator length (see Section 1.6 in the same article)
and was recently found relevant to the study of one-relator groups
(see, for example, \cite{louder2022negative}). Definition \ref{def:q-primitivity-rank}
seemingly introduces a family of related invariants of words -- one
for every prime power $q$. In fact, the same definition can be applied
to arbitrary fields -- see Section \ref{sec:Open-Questions}. However,
it is possible that all these invariants coincide for a given word.
We are able to show one inequality and conjecture a full equality.
\begin{prop}
\label{prop:pi_q =00005Cle pi}For every word $w\in\F$ and every
prime power $q$, $\pi_{q}\left(w\right)\le\pi\left(w\right)$.
\end{prop}

\begin{conjecture}
\label{conj:pi and pi_q}For every word $w\in\F$ and every prime
power $q$, $\pi_{q}\left(w\right)=\pi\left(w\right)$.
\end{conjecture}

Conjecture \ref{conj:pi and pi_q}, along with Conjecture \ref{conj:general pi and fixed vectors},
are in line with a universal conjecture -- \cite[Conj.~1.13]{hanany2020word}
-- about the role of the primitivity rank $\pi\left(w\right)$ in
word measures on groups. For more background, see Section 1.6 in the
same article. \medskip{}

As part of our study of word measures in $\gln$ employing the free
group algebras, we also prove some results about these algebras which
may be of independent interest. For example, suppose that $T$ is
a subtree of the Cayley graph of $\F$ with respect to some basis.
If $I\le\A$ is a finitely generated ideal which is supported on $T$,
then $I$ admits a basis with a generating set supported on $T$ (Theorem
\ref{thm:basis from coincidences}). We also analyze the $\A$-module
$\A/\left(w-1\right)$ obtained as the quotient of the right $\A$-module
$\A$ by its submodule $\left(w-1\right)$. Theorem \ref{thm:cyclic generators of A_w}
proves an analogue of Kaplansky's unit conjecture for these modules
and shows that if $w$ is a non-power, then the only cyclic generators
of $\A/\left(w-1\right)$ are the trivial ones. See Section \ref{sec:Open-Questions}
for a further discussion of this line of research.

\subsection{General stable class functions and characters\label{subsec:General-stable-class}}

As mentioned above, some of the results concerning the function $\fix$
and its expectation under word measures are only an illustrative special
case of more general results. The variety of functions we consider
are those relating to \emph{stable} representations of the family
$\gl_{\bullet}\left(\k\right)$ (see \cite{putman2017representation,gan2018representation}).
Below we present the generalizations of Theorems \ref{thm:rational expression fix}
and \ref{thm:limit powers fix} and of Conjecture \ref{conj:general pi and fixed vectors}. 

First, we must remark on the unconventional definition we make in
this paper. Formal words in group theory are usually read from left
to right: this is why one usually considers \emph{right} Cayley graphs.
As a consequence, we consider here the slightly non-standard \emph{right}
action of $\gln$ on $V_{N}\defi K^{N}$, namely, we consider $V_{N}$
as row vectors, and the action of $g\in\gln$ on $v\in V_{N}$ is
given by $\left(v,g\right)\mapsto vg$. Thus, the action of $w\left(g_{1},\ldots,g_{r}\right)$
on a vector $v\in V_{N}$ can be thought of as the composition of
the action, letter by letter, from left to right -- the natural direction
in which the word is read.

Rather than considering only the number of vectors fixed by $g$,
we consider more generally the number of subspaces of $V$ of a fixed
dimension which are invariant under $g$ and on which $g$ acts in
a prescribed way. This is formalized as follows:
\begin{defn}
\label{def:B_w}Let $m\in\mathbb{Z}_{\ge1}$ and $\B\in\glm$. We
define a map \marginpar{$\protect\btil$}$\btil\colon\gln\to\mathbb{Z}_{\ge0}$
(valid for arbitrary $N$) as follows. For $g\in\gln$ we let $\btil\left(g\right)$
be the number of $m$-tuples of vectors $v_{1},\ldots,v_{m}\in V_{N}=\k^{N}$
on which the (right) action of $g$ can be described by a multiplication
from the left by the matrix $\B$. Namely, 
\[
\btil\left(g\right)=\#\left\{ M\in M_{m\times N}\left(\k\right)\,\middle|\,Mg=\B M\right\} .
\]

For example, if $\B=\left(1\right)\in\gl_{1}\left(\k\right)$, then
$\btil=\fix$. For $\B=\left(\lambda\right)\in\gl_{1}\left(\k\right)$,
the function $\btil$ gives the size of the eigenspace $V_{\lambda}\le V_{N}$
of an element. If $\B=I_{m}\in\gl_{m}\left(\k\right)$, then $\btil\left(g\right)=\fix\left(g\right)^{m}$,
and if 
\[
{\cal B}=\left(\begin{array}{ccccc}
 & 1\\
 &  & 1\\
 &  &  & \ddots\\
 &  &  &  & 1\\
1
\end{array}\right)\in\gl_{m}\left(\k\right),
\]
then $\btil\left(g\right)=\fix\left(g^{m}\right)$. The following
two theorems are the generalization of Theorems \ref{thm:rational expression fix}
and \ref{thm:limit powers fix}:

\end{defn}

\begin{thm}
\label{thm:rational general}Suppose that $w\in\F$, $m\in\mathbb{Z}_{\ge1}$
and $\B\in\glm$. Then for every large enough $N$, the expectation
$\mathbb{E}_{w}\left[\btil\right]$ is given by a rational function
in $q^{N}$.
\end{thm}

\begin{thm}
\label{thm:limit powers general}Let $1\ne w=u^{d}$ with $d\ge1$
and $u$ a non-power. For every $m\in\mathbb{Z}_{\ge1}$ and $\B\in\glm$,
the limit $\lim_{N\to\infty}\mathbb{E}_{w}\left[\btil\right]$ exists
and depends only on $d$ and not on $u$.
\end{thm}

In the special case of $\btil=I_{m}\in\gl_{m}\left(\k\right)$, Theorem
\ref{thm:rational general} appeared in \cite{West19}. The same special
case of Theorem \ref{thm:limit powers general} first appeared in
the same thesis, and then, independently, in \cite[Sec.~8]{eberhard2021babai}. 

In particular, Theorem \ref{thm:limit powers general} captures all
moments of the number of fixed vectors under the $w$-measure. So
if $w=u^{d}$, all these moments converge, as $N\to\infty$, to the
same limits as for $w=a^{d}$, namely as for a $d$-th power of a
uniformly random element of $\gln$. Denote the number of fixed vectors
in $\k^{N}$ of a $w$-random element of $\gln$ by $\fix_{w,N}$\marginpar{$\protect\fix_{w,N}$}.
When $w=a$, a limit distribution as $N\to\infty$ is known to exist
\cite[Thm.~2.1]{fulman2016distribution}\footnote{This was originally proved by Rudvalis and Shinoda -- see \cite{fulman2016distribution}.}.
Although this limit distribution is not determined by its moments,
we do prove the following in Appendix \ref{sec:Matan}:
\begin{thm}
\label{thm:limit distr}Let $1\ne w\in\F$ be a non-power. Then the
random variables $\fix_{w,N}$ have a limit distribution, and this
limit distribution is identical to the one of $\fix_{a,N}$ described
in \cite[Thm.~2.1]{fulman2016distribution}.
\end{thm}

Theorem \ref{thm:limit distr} is analogous to the $L=d=1$ case of
Nica's main Theorem 1.1 from \cite{nica1994number}, which revolves
around the limit distribution of the number of fixed point in $w$-random
permutations. We suspect that Theorem \ref{thm:limit distr} can be
generalized to a full analogue of Nica's result (and see Section \ref{sec:Open-Questions}).
\begin{rem}
One can further generalize Theorem \ref{thm:rational general} to
more than one word. For example, for any tuple of words $\wl\in\F$,
consider an $\ell$-tuple of random elements 
\[
\overline{w_{1}}=w_{1}\left(g_{1},\ldots,g_{r}\right),\ldots,\overline{w_{\ell}}=w_{\ell}\left(g_{1},\ldots,g_{r}\right)\in\gln,
\]
where $g_{1},\ldots,g_{r}$ are independent, uniformly random elements
of $\gln$, and consider expressions like $\mathbb{E}\left[\fix\left(\overline{w_{1}}\right)\cdot\fix\left(\overline{w_{2}}\right)\cdots\fix\left(\overline{w_{\ell}}\right)\right]$.
The same argument given in the proof of Theorem \ref{thm:rational general}
shows that this expectation is given by a rational expression in $q^{N}$.
Also, Corollary 1.3 in \cite{West19} shows that the difference $\mathbb{E}\left[\fix\left(\overline{w_{1}}\right)\cdots\fix\left(\overline{w_{\ell}}\right)\right]-\mathbb{E}_{w_{1}}\left[\fix\right]\cdots\mathbb{E}_{w_{\ell}}\left[\fix\right]=O\left(\frac{1}{q^{N}}\right)$
if and only if no pair of words is conjugated into the same cyclic
subgroup of $\F$.
\end{rem}

We further introduce a generalization of Conjecture \ref{conj:general pi and fixed vectors}.
Consider \marginpar{$\protect\R$}\label{page:R ring of stable class functions}
\[
\R\defi\mathbb{C}\left[\left\{ \btil\,\middle|\,\B\in\glm,m\in\mathbb{Z}_{\ge0}\right\} \right],
\]
the $\mathbb{C}$-algebra generated by all functions $\btil$ from
Definition \ref{def:B_w}. Note that every element of $\R$ is a (class)
function defined on $\gln$ for every $N$. Rather than formal polynomials
in the $\btil$'s, the elements of $\R$ are functions on $\gl_{\bullet}\left(\k\right)$,
so two elements giving the same function on $\gln$ for every $N$
are identified. For example, every conjugate of $\B$ gives rise to
the same function as $\B$. In fact, this is the only case where two
elements give the same function: $\btil_{1}=\btil_{2}$ if and only
if $\B_{1}$ and $\B_{2}$ belong to $\gl_{m}\left(\k\right)$ for
the same $m$ and are conjugates -- see \cite[Cor.~3.1]{EW-P-Sh2024}.
If we also include the constant function $1$, thought of as $\btil$
where $\B=e\in\gl_{0}\left(\k\right)\defi\left\{ e\right\} $, then
$\R$ is the $\mathbb{C}$-span of the $\btil$'s: indeed, if $\B_{1}\in\gl_{m_{1}}\left(\k\right)$
and $\B_{2}\in\gl_{m_{2}}\left(\k\right)$, then $\btil_{1}\cdot\btil_{2}=\widetilde{\B_{1}\oplus\B_{2}}$
where $\B_{1}\oplus\B_{2}\in\gl_{m_{1}+m_{2}}\left(\k\right)$ is
the suitable block-diagonal matrix. In the same article, it is shown
that $\R$ is, in fact, a graded algebra and admits a linear basis
consisting of $\left\{ \btil\right\} $, where $\B$ goes over exactly
one representative from every conjugacy class in all $\glm$ $\left(m\ge0\right)$.

Some of the functions in $\R$ coincide, for large enough $N$, with
irreducible characters of $\gln$. For example, for $N\ge2$, the
action of $\gln$ on the projective space $\mathbb{P}^{N-1}\left(\k\right)$
decomposes to the trivial representation and an irreducible representation
whose character we denote $\chi^{\mathbb{P}}$. Then for every $N\ge2$,
the character $\chi^{\mathbb{P}}$ is equal to an element in $\R$:
\[
\chi^{\mathbb{P}}=\frac{1}{q-1}\sum_{\lambda\in\k^{*}}\left(\tilde{\lambda}-1\right)-1
\]
(here $\tilde{\lambda}$ is the function corresponding to $\lambda\in\gl_{1}\left(\k\right)$).
In \cite{EW-P-Sh2024} it is shown that the set of families of irreducible
characters $\left\{ \chi_{N}\in\gln\right\} _{N\ge N_{0}}$ which
coincide with elements of $\R$ is precisely the set of \emph{stable}
irreducible representations of $\gl_{\bullet}\left(\k\right)$ as
in \cite{gan2018representation}. Our generalization of Conjecture
\ref{conj:general pi and fixed vectors} deals with these families
of irreducible characters.
\begin{conjecture}
\label{conj:irreducible chars}Let $\chi$ be a stable character of
$\gl_{\bullet}\left(\k\right)$, namely, an element of $\R$ which
coincides, for every large enough $N$, with some irreducible character
of $\gln$. Then
\[
\mathbb{E}_{w}\left[\chi\right]=O\left(\left(\dim\chi\right)^{1-\pi_{q}\left(w\right)}\right).
\]
\end{conjecture}

By Theorem \ref{thm:rational general} (with $w=1$), $\dim\chi=\chi\left(1\right)$
is a rational function in $q^{N}$. Conjecture \ref{conj:irreducible chars},
together with a positive answer to Question \ref{conj:pi and pi_q},
constitute a special case of the more general, albeit not as precise,
\cite[Conj.~1.13]{hanany2020word}. See also \cite[Conj.~A.4]{stable2023}
for a conjecture slightly more ambitious than Conjecture 1.15.

Note that for $N\ge2$, the decomposition of the function $\fix$
to irreducible characters is
\[
\fix=2\cdot\mathrm{triv}+\chi^{\mathbb{P}}+\xi_{1}+\ldots+\xi_{q-2},
\]
where $\xi_{1},\ldots,\xi_{q-2}$ are distinct irreducible characters,
each of dimension $\frac{q^{N}-1}{q-1}$, all belonging to $\R$.
Thus, they all fall into the framework of Conjecture \ref{conj:irreducible chars},
and we get that this conjecture implies, in particular, that $\mathbb{E}_{w}\left[\fix\right]=2+O\left(\left(q^{N}\right)^{1-\pi_{q}\left(w\right)}\right)$.
In particular, Conjecture \ref{conj:irreducible chars} generalizes
(a slightly weaker version of) Conjecture \ref{conj:general pi and fixed vectors}.
Some background for Conjecture \ref{conj:irreducible chars} can be
found in \cite[Sec.~1]{hanany2020word}.

\subsection{Reader's guide}

\subsubsection*{Notation}

The free group $\F$ has rank $r$ and a fixed basis $B=\left\{ b_{1},\ldots,b_{r}\right\} $.
Recall that all ideals in this paper are one-sided right ideals unless
stated otherwise, and we write $I\le\A$ to mean that $I$ is an ideal
of the free group algebra $\A=\k\left[\F\right]$. More generally,
we write $M\le\A^{m}$ if $M$ is a submodule of the free right $\A$-module
$\A^{m}$. For any set $S\subseteq\A^{m}$, we denote by $\left(S\right)$
the submodule generated by $S$, and if $S=\left\{ s_{1},\ldots,s_{t}\right\} $
we may also write $\left(s_{1},\ldots,s_{t}\right)$.

We denote by $E=\left\{ e_{1},\ldots,e_{m}\right\} $ a basis for
the free $\A$-module $\A^{m}$. The elements of the form $ez$ with
$e\in E$ and $z\in\F$ are called \emph{monomials. }For a subset
$Q$ of the monomials, we write $M\le_{Q}\A^{m}$ to mean that $M$
has a generating set in $\A^{m}$ such that each of its elements is
supported on $Q$. Usually, for ideals inside $\A$, we consider subsets
of $\F$ corresponding to the vertices in some subtree $T$ of the
(right) Cayley graph $\C\defi\mathrm{Cay}\left(\F,B\right)$ of $\F$
with respect to the basis $B$. In this case, instead of $I\le_{\mathrm{vert}\left(T\right)}\A$
(here, of course, $\mathrm{vert}\left(T\right)$ denotes the set of
vertices of $T$), we simply write $I\le_{T}\A$. More generally,
for submodules of $\A^{m}$, we usually consider $m$ disjoint copies
$\C_{1},\ldots,\C_{m}$ of $\mathrm{Cay}\left(\F,B\right)$, with
origins $e_{1},\ldots,e_{m}$, respectively, and consider a collection
of (possibly empty) subtrees $\T=T_{1}\cup\ldots\cup T_{m}$, with
$T_{i}\subset\C_{i}$. We write $M\le_{\T}\A^{m}$ to mean that $M$
is generated by elements supported on the vertices of $\T$. 

For a submodule $M\le\A^{m}$ and a set $S$ of monomials in $\A^{m}$,
we let $M|_{S}$ denote the set of elements of $M$ which are supported
on $S$. This is a vector space over $\k$.

\subsubsection*{Paper organization}

After a very brief survey of related works in Section \ref{subsec:Related-works},
Section \ref{sec:Rational-expressions} proves that $\mathbb{E}_{w}\left[\fix\right]$,
and likewise $\mathbb{E}_{w}\left[\btil\right]$, are given by rational
functions in $q^{N}$ (Theorems \ref{thm:rational expression fix}
and \ref{thm:rational general}, respectively). In Section \ref{sec:The-free-group-algebra}
we study the free group algebra and its ideals, show how the computation
of $\mathbb{E}_{w}\left[\fix\right]$ is related to ``exploration''
processes in the Cayley graph of $\F$, and prove some basic properties
of the $q$-primitivity rank including Proposition \ref{prop:pi_q =00005Cle pi}.
We then study $\lim_{N\to\infty}\mathbb{E}_{w}\left[\fix\right]$
and $\lim_{N\to\infty}\mathbb{E}_{w}\left[\btil\right]$ and prove
Theorems \ref{thm:limit powers fix} and \ref{thm:limit powers general}
in Section \ref{sec:Powers}. Section \ref{sec:A_w} studies the right
$\A$-module $\A/\left(w-1\right)$ and specifies its cyclic generators,
and also gives a criterion to detect when $w-1$ is primitive in a
given rank-2 ideal in $\A$. Section \ref{sec:Critical-ideals-of-rank-2}
deals with the coefficient of $\frac{1}{q^{N}}$ in the Laurent expansion
of $\mathbb{E}_{w}\left[\fix\right]$ and proves Theorem \ref{thm:fixed vectors in pi=00003D2}.
Section \ref{sec:Open-Questions} gathers the many open questions
that are raised by this work. Finally, Appendix \ref{sec:Matan} contains
the proof of Theorem \ref{thm:limit distr}.

\subsection{Related works\label{subsec:Related-works}}

As mentioned above, the two phenomena described in Theorems \ref{thm:rational general}
and \ref{thm:limit powers general} are found in other families of
groups. The fact that the expectation under word measures of ``natural''
class functions over certain families of groups are given by rational
functions was first established for the symmetric group \cite{nica1994number,Linial2010}.
It was later established for the classical groups $\mathrm{U}\left(N\right)$
\cite{MP-Un} and $\mathrm{O}\left(N\right)$ and $\mathrm{Sp}\left(N\right)$
\cite{MP-On} based on Weingarten calculus (see, for instance, \cite{collins2006integration}),
and also in the wreath product $G\wr S_{N}$ for an arbitrary finite
group $G$ \cite{MP-surface-words,shomroni2023wreathI}. A related
phenomenon appears when free groups are replaced by surface groups
(fundamental groups of compact closed surfaces). Indeed, there is
a natural definition of a measure induced by an element of a surface
group on finite groups and certain compact groups, and the expected
value of certain characters of the symmetric group $\mathrm{Sym}\left(N\right)$
under such measures can be approximated to any degree by a rational
function \cite{MP-asymptotic}. A similar result holds for measures
induced by elements of surface groups on $\mathrm{SU}\left(N\right)$
\cite{magee2022random}.

The phenomenon described in Theorems \ref{thm:limit powers fix} and
\ref{thm:limit powers general}, that if $w=u^{d}$ then the limit
expectation of natural class functions in the family under the $w$-measure
depends only on $d$ and not on $u$, is also found in many of the
above mentioned cases. It is true in $\mathrm{Sym}\left(N\right)$
\cite{nica1994number,Linial2010}, in $\mathrm{U}\left(N\right)$
\cite{MSS07,Radulescu06}, as well as in $\mathrm{O}\left(N\right)$
and in $\mathrm{Sp}\left(N\right)$ \cite{MP-On}. It also holds in
the characters analyzed in \cite{MP-asymptotic} for measures on $\mathrm{Sym}\left(N\right)$
induced by elements of surface groups {[}ibid, Theorem 1.2{]}.

Finally, there are analogues to Theorem \ref{thm:fixed vectors in pi=00003D2}
and Conjectures \ref{conj:general pi and fixed vectors} and \ref{conj:irreducible chars},
which give an interpretation to the order of $\mathbb{E}_{w}\left[f\right]-\mathbb{E}_{x}\left[f\right]$,
an interpretation which lies in invariants of $w$ as an element of
the abstract free group $\F$. We mentioned previously that there
are very similar results in the case of $\mathrm{Sym}\left(N\right)$
\cite{PP15,hanany2020word}. There are other invariants of $w$ explaining
the leading order (and sometimes much more than the leading order)
in the expected values of class functions in $\mathrm{U}\left(N\right)$,
$\mathrm{O}\left(N\right)$, $\mathrm{Sp}\left(N\right)$ and $G\wr S_{N}$
\cite{MP-Un,MP-On,MP-surface-words,brodsky2021unitary,shomroni2023wreathI,shomroni2023wreathII}.
A more detailed summary may be found in \cite[Section 1.3]{hanany2020word}.

\section*{Acknowledgments}

We thank Khalid Bou-Rabee for some prehistorical discussions on word
measures on $\gl_{2}\left(\k\right)$. We corresponded with Christian
Berg on subjects around Theorem \ref{thm:limit distr} and its proof,
and we thank him for some very helpful pointers he gave us. We also
thank George Bergman,  Michael Magee and Yotam Shomroni for helpful
discussions, and an anonymous referee for a thorough reading of the
paper and many suggestions for small improvements. This research has
received funding from the Israel Science Foundation: ISF grant 1071/16
and from the European Research Council (ERC) under the European Union’s
Horizon 2020 research and innovation programme (grant agreement no.~850956).

\section{Rational expressions\label{sec:Rational-expressions}}

In this section we prove Theorems \ref{thm:rational expression fix}
and \ref{thm:rational general}, which show that the expectations
under word measures of the class functions we consider on $\gln$
are given by rational functions in $q^{N}$. The proof uses only linear
algebra and can be written in completely elementary terms. While we
start with this approach, we then ``translate'' the proof to the
language of ideals and modules of the free group algebra $\A=\k\left[\F\right]$.
Given our additional results and conjectures, the latter language
is much more fruitful.

\subsection{The function $\protect\fix$ and Theorem \ref{thm:rational expression fix}\label{subsec:The-function-fix}}

\subsubsection*{The elementary approach}

We first illustrate the proof in the somewhat simpler special case
considered in Theorem \ref{thm:rational expression fix}: the function
$\fix$. Let $V_{N}=\k^{N}$ be the vector space of row vectors of
length $N$. Given $w\in\F$, one needs to count all $g_{1},\ldots,g_{r}\in\gln$
and $v\in V_{N}$ such that $v.w\left(g_{1},\ldots,g_{r}\right)=v$.
We consider the entire trajectory of $v$ when the letters of $w$
are applied one by one. Namely, assume that $w$ is written in the
basis $B=\left\{ b_{1},\ldots,b_{r}\right\} $ of $\F$ as $w=b_{i_{1}}^{\varepsilon_{1}}\ldots b_{i_{\ell}}^{\varepsilon_{\ell}}$
(where $i_{j}\in\left\{ 1,\ldots,r\right\} $ and $\varepsilon_{j}\in\left\{ \pm1\right\} $).
We consider the vectors 
\begin{equation}
v^{0}\defi v,~~~v^{1}\defi v^{0}.g_{i_{1}}^{\varepsilon_{1}},~~~v^{2}\defi v^{1}.g_{i_{2}}^{\varepsilon_{2}},~~~\ldots,~~~v^{\ell-1}\defi v^{\ell-2}.g_{i_{\ell-1}}^{\varepsilon_{\ell-1}},~~~v^{\ell}\defi v^{\ell-1}.g_{i_{\ell}}^{\varepsilon_{\ell}}=v^{0}.\label{eq:trajectory}
\end{equation}
We denote this trajectory by $\overline{v}=\left(v^{0},\ldots,v^{\ell}\right)$.
Given that the entire trajectory is determined by $g_{1},\ldots,g_{r}$
and $v=v^{0}$, we do not change our goal by counting $\left(g_{1},\ldots,g_{r};\overline{v}\right)$
satisfying the equations in \eqref{eq:trajectory} instead of $\left(g_{1},\ldots,g_{r};v\right)$
satisfying $v.w\left(g_{1},\ldots,g_{r}\right)=v$.

The basic idea behind our counting is grouping together solutions
$\left(g_{1},\ldots,g_{r};\overline{v}\right)$ according to the linear
relations over $\k$ which $v^{0},\ldots,v^{\ell}$ satisfy. There
are finitely many options here (trivially, at most the number of linear
subspaces of $\k^{\ell+1}$), and, as we show below, for each subspace
of $\k^{\ell+1}$ the number of solutions $\left(g_{1},\ldots,g_{r};\overline{v}\right)$
corresponding to it is either identically zero for every $N$, or
its contribution to $\mathbb{E}_{w}\left[\fix\right]$ is given by
a rational function in $q^{N}$ for every large enough $N$.

Denote by $\left[1,w\right]$ the subtree of $\C=\mathrm{Cay}\left(\F,B\right)$
corresponding to the path from the origin to the vertex $w$. For
every $b\in B$, denote by $D_{b}\left(w\right)$ the vertices of
$\left[1,w\right]$ with an outgoing $b$-edge (within $\left[1,w\right]$),
and denote by $e_{b}\left(w\right)$ the number of $b$-edges in $\left[1,w\right]$,
so $e_{b}\left(w\right)=\left|D_{b}\left(w\right)\right|$. Now consider
a subspace $\Delta\le\k^{\ell+1}$ thought of as a set of equations
on the vectors $v^{0},\ldots,v^{\ell}$, or, equivalently, on the
vertices of $\left[1,w\right]$. Below we denote these vertices by
the corresponding prefix of $w$ in $\F$, and write elements of $\k^{\left[1,w\right]}\defi\k^{\mathrm{vert}\left(\left[1,w\right]\right)}\cong\k^{\ell+1}$
as linear combinations of these vertices. We have
\begin{eqnarray}
\mathbb{E}_{w}\left[\fix\right] & = & \frac{\#\left\{ g_{1},\ldots,g_{r}\in\gln,v\in V_{N}\,\middle|\,v.w\left(g_{1},\ldots,g_{r}\right)=v\right\} }{\left|\gln\right|^{r}}\nonumber \\
 & = & \frac{\#\left\{ g_{1},\ldots,g_{r}\in\gln,\overline{v}\in V_{N}^{~\ell+1}\,\middle|\,\overline{v}~\mathrm{and}~g_{1},\ldots,g_{r}~\mathrm{satisfy}~\eqref{eq:trajectory}\right\} }{\left|\gln\right|^{r}}\nonumber \\
 & = & \sum_{\Delta\le\k^{\left[1,w\right]}}\frac{\#\left\{ g_{1},\ldots,g_{r}\in\gln,\overline{v}\in V_{N}^{~\ell+1}\,\middle|\,\begin{gathered}\overline{v}~\mathrm{satisfies~precisely}~\Delta,\\
\overline{v},g_{1},\ldots,g_{r}~\mathrm{satisfy}~\eqref{eq:trajectory}
\end{gathered}
\right\} }{\left|\gln\right|^{r}}\label{eq:contrib by Delta}
\end{eqnarray}

If there are solutions $\left(g_{1},\ldots,g_{r};\overline{v}\right)$
which satisfy precisely $\Delta$, then the following two conditions
hold:
\begin{description}
\item [{C1:}] $w-1\in\Delta$ (here $w-1$ is the equation $w-1=0$, or,
equivalently, $v^{\ell}-v^{0}=0$).
\item [{C2:}] $\Delta$ is ``closed under multiplication by $b^{\pm1}$''.
Namely, for every $b\in B$ and every equation $\delta=\sum_{z\in D_{b}\left(w\right)}\lambda_{z}z$
($\lambda_{z}\in\k$) supported on $D_{b}\left(w\right)$, denote
by $\delta b\defi\sum_{z\in D_{b}\left(w\right)}\lambda_{z}zb$ the
corresponding equation on the vertices on the termini of the corresponding
$b$-edges. Then 
\[
\delta\in\Delta~~~~\Longleftrightarrow~~~~\delta b\in\Delta.
\]
\end{description}
Conversely, if $\Delta$ satisfies conditions \textbf{C1} and \textbf{C2},
then for every large enough $N$ there exist solutions $\left(g_{1},\ldots,g_{r};\overline{v}\right)$
satisfying precisely $\Delta$, and the contribution of $\Delta$
in \eqref{eq:contrib by Delta} is given by a rational function in
$q^{N}$. Indeed, denote by $\dim\left(\Delta\right)$ the dimension
of the subspace $\Delta$, and by $\dim_{b}\left(\Delta\right)$ the
dimension of the subspace of $\Delta$ consisting of equations supported
on $D_{b}\left(w\right)$. First, we choose a trajectory $\overline{v}\in V_{N}^{~\ell+1}$
satisfying precisely $\Delta$. Note that the number of choices for
such $\overline{v}$ is precisely $\indep_{\ell+1-\dim\left(\Delta\right)}\left(V_{N}\right)$,
where\marginpar{$\protect\indep$}
\[
\indep_{h}\left(V_{N}\right)\defi\left(q^{N}-1\right)\left(q^{N}-q\right)\cdots\left(q^{N}-q^{h-1}\right)
\]
is the number of $h$-tuples of independent vectors in $V_{N}$.\footnote{We could not find a conventional notation for the quantity $\indep_{h}\left(V_{N}\right)$.
However, it is closely related to existing common notation. For example,
$\indep_{h}\left(v\right)=q^{Nh}\cdot\left(q^{-N};q\right)_{h}$,
where $\left(t;q\right)_{h}\defi\left(1-t\right)\left(1-tq\right)\cdots\left(1-tq^{h-1}\right)$
is the $q$-shifted factorial.}

Second, given a trajectory $\overline{v}$ satisfying precisely $\Delta$,
we choose the tuple $g_{1},\ldots,g_{r}\in\gln$ so that $\overline{v},g_{1},\ldots,g_{r}$
satisfy \eqref{eq:trajectory}. We choose $g_{i}$ separately for
every $i=1,\ldots,r$. Let $b=b_{i}$. Note that the vectors of $\overline{v}$
at the starting points of $b$-edges in $\left[1,w\right]$, namely,
in $D_{b}\left(w\right)$, span a subspace of $V$ of dimension $e_{b}\left(w\right)-\dim_{b}\left(\Delta\right)$.
(Such a trajectory may exist only if $e_{b}\left(w\right)-\dim_{b}\left(\Delta\right)\le N$).
In this case, the element $g_{i}$ should map a subspace of dimension
$e_{b}\left(w\right)-\dim_{b}\left(\Delta\right)$ in a prescribed
way, and condition \textbf{C2} guarantees this prescribed way is valid
and can be realized by a linear transformation. The number of elements
in $\gln$ satisfying this constraint is
\[
\left(q^{N}-q^{e_{b}\left(w\right)-\dim_{b}\left(\Delta\right)}\right)\left(q^{N}-q^{e_{b}\left(w\right)-\dim_{b}\left(\Delta\right)+1}\right)\cdots\left(q^{N}-q^{N-1}\right).
\]
If $g_{1},\ldots,g_{r}$ satisfy these constraints and as \textbf{C1}
holds, $\overline{v}$ and $g_{1},\ldots,g_{r}$ satisfy \eqref{eq:trajectory}.
Overall, if $N\ge e_{b}\left(w\right)-\dim_{b}\left(\Delta\right)$
for every $b\in B$, the term corresponding to $\Delta$ in \eqref{eq:contrib by Delta}
is
\begin{eqnarray*}
 &  & \indep_{\ell+1-\dim\left(\Delta\right)}\left(V_{N}\right)\cdot\prod_{b\in B}\frac{\left(q^{N}-q^{e_{b}\left(w\right)-\dim_{b}\left(\Delta\right)}\right)\left(q^{N}-q^{e_{b}\left(w\right)-\dim_{b}\left(\Delta\right)+1}\right)\cdots\left(q^{N}-q^{N-1}\right)}{\left(q^{N}-1\right)\left(q^{N}-q\right)\cdots\left(q^{N}-q^{N-1}\right)}\\
 & = & \frac{\indep_{\ell+1-\dim\left(\Delta\right)}\left(V_{N}\right)}{\prod_{b\in B}\indep_{e_{b}\left(w\right)-\dim_{b}\left(\Delta\right)}\left(V_{N}\right)}
\end{eqnarray*}
which is rational in $q^{N}$. Overall, we obtain
\begin{equation}
\mathbb{E}_{w}\left[\fix\right]=\sum_{\Delta\le K^{\left[1,w\right]}\colon\,\Delta~\mathrm{satisfies~\mathbf{C1},\mathbf{C2}}}\frac{\indep_{\ell+1-\dim\left(\Delta\right)}\left(V_{N}\right)}{\prod_{b\in B}\indep_{e_{b}\left(w\right)-\dim_{b}\left(\Delta\right)}\left(V_{N}\right)},\label{eq:E=00005Bfix=00005D as rational expression elementary approach}
\end{equation}
which completes the proof of Theorem \ref{thm:rational expression fix}.

\subsubsection*{The free-group-algebra approach}

The key observation that leads to the free-group-algebra approach
is that condition \textbf{C2} above is a feature of (as always, right)
ideals of the free group algebra $\A=\k\left[\F\right]$: a right
ideal $I\le\A$ is a $\k$-linear subspace of $\A$ satisfying \textbf{C2
}on the entire Cayley graph $\C$ (rather than on $\left[1,w\right]$
alone). To make this formal, let us recall some notation. For a subtree
$T$ of the Cayley graph $\C=\mathrm{Cay}\left(\F,B\right)$, denote
by $D_{b}\left(T\right)$ the set of vertices in the subtree $T\subset\C$
with an outgoing $b$-edge (inside $T$), and by $e_{b}\left(T\right)=\left|D_{b}\left(T\right)\right|$
the number of such edges. For any ideal $I\le\A$, its restriction
to $T$, denoted
\[
I|_{T}\defi I\cap\k^{\mathrm{vert}\left(T\right)},
\]
is a linear subspace of $\k^{\mathrm{vert}\left(T\right)}$. We say
that a $\k$-linear subspace $\Delta\le\k^{\mathrm{vert}\left(T\right)}$
satisfies \textbf{C2$\left(T\right)$ }if for every $\delta\in\k^{\mathrm{vert}\left(T\right)}$
supported on $D_{b}\left(T\right)$, we have $\delta\in\Delta$ if
and only if $\delta.b\in\Delta$.
\begin{lem}
\label{lem:ideal generated by Delta with C2 does not add elements to tree}Assume
that $\Delta\le\k^{\mathrm{vert}\left(T\right)}$ is a $\k$-linear
subspace satisfying \textbf{C2$\left(T\right)$}. Then $\left(\Delta\right)\le\A$,
the ideal generated by $\Delta$, does not introduce any new elements
supported on $T$, namely
\begin{equation}
\left(\Delta\right)|_{T}=\Delta.\label{eq:(Delta)|_T=00003DDelta}
\end{equation}
\end{lem}

\begin{proof}
It is clear that $\left(\Delta\right)|_{T}\supseteq\Delta$, so it
is enough to show the converse inclusion. We may assume that $T$
is finite: every element of $\A$ has finite support, and every element
of $\left(\Delta\right)$ is generated by finitely many elements of
$\Delta$. So if $T$ is not finite and \eqref{eq:(Delta)|_T=00003DDelta}
fails, replace $T$ with the finite subtree $S\subseteq T$ which
is the convex hull of the support of an element in $\left(\Delta\right)|_{T}\setminus\Delta$
and its finitely many generators in $\Delta$ and replace $\Delta$
with $\Delta|_{S}$.

As in the proof of Theorem \ref{thm:rational expression fix} above,
for large enough $N$, there are $g_{1},\ldots,g_{r}\in\gln$ and
$\overline{v}=\left\{ v_{z}\in V_{N}\right\} _{z\in\mathrm{vert}\left(T\right)}$
such that for every $b$-edge $\xymatrix{z_{1}\ar@{->}[r]_{b} & z_{2}}
$ in $T$, we have $v_{z_{1}}.g_{b}=v_{z_{2}}$, and such that the
equations over $\k$ satisfied by the vectors $\overline{v}$ are
\emph{precisely} the elements of $\Delta$. The tuple $g_{1},\ldots,g_{r}$
defines a group homomorphism $\F\to\gln$ by $b_{i}\mapsto g_{i}$.
This group homomorphism defines, in turn, a homomorphism of $\k$-algebras
$\A\to\mathrm{End}\left(V_{N}\right)$. Equivalently, such a homomorphism
of $\k$-algebras defines a structure of an $\A$-module on $V_{N}$.
Pick an arbitrary $z_{0}\in\mathrm{vert}\left(T\right)\subseteq\F$.
Now $\A$ is itself an $\A$-module, and, moreover, it is a free $\A$-module
with basis $\left\{ z_{0}\right\} $. There is a unique $\A$-module
homomorphism $\phi\colon\A\to V_{N}$ such that $\phi\left(z_{0}\right)=v_{z_{0}}$.
Since $\phi$ is an $\A$-module homomorphism and $T$ is connected,
the choice of the $g_{i}$'s guarantees that $\phi\left(z\right)=v_{z}$
for every $z\in\mathrm{vert}\left(T\right)\subseteq\F$.

Finally, $\ker\phi\le\A$ is a submodule, or an ideal, and the equations
it satisfies on $\mathrm{vert}\left(T\right)$ are precisely those
satisfied by $\overline{v}$, namely, precisely $\Delta$. Thus
\[
\left(\Delta\right)|_{T}\le\left[\ker\phi\right]|_{T}=\Delta.
\]
\end{proof}
Returning to the case $T=\left[1,w\right]$, recall that we write
$I\le_{\left[1,w\right]}\A$ if $I$ is an ideal of $\A$ with generating
set supported on $\left[1,w\right]$. Lemma \ref{lem:ideal generated by Delta with C2 does not add elements to tree}
yields that there is a one-to-one correspondence
\[
\left\{ \begin{gathered}\Delta\le\k^{\left[1,w\right]}\\
\mathrm{satisfying~\mathbf{C1~}\&~\mathbf{C2}}
\end{gathered}
\right\} ~~~\Longleftrightarrow~~~\left\{ I\le_{\left[1,w\right]}\A\,\middle|\,w-1\in I\right\} .
\]
For an ideal $I\le\A$ and every finite subtree $T$ of $\C$, define\marginpar{$d^{T}\left(I\right)$} 

\[
d^{T}\left(I\right)\defi\dim_{\k}\left(I|_{T}\right).
\]
Similarly, for every basis element $b\in B$, denote by $D_{b}\left(T\right)$
the set of vertices in the subtree $T\subset\C$ with an outgoing
$b$-edge (inside $T$), and let\marginpar{$d_{b}^{T}\left(I\right)$}
\[
d_{b}^{T}\left(I\right)\defi\dim_{\k}\left(I|_{D_{b}\left(T\right)}\right).
\]
With this notation, \eqref{eq:E=00005Bfix=00005D as rational expression elementary approach}
is equivalent to 
\begin{equation}
\mathbb{E}_{w}\left[\fix\right]=\sum_{I\le_{\left[1,w\right]}\A\colon I\ni w-1}\frac{\indep_{\left|w\right|+1-d^{\left[1,w\right]}\left(I\right)}\left(V_{N}\right)}{\prod_{b\in B}\indep_{e_{b}\left(w\right)-d_{b}^{\left[1,w\right]}\left(I\right)}\left(V_{N}\right)}.\label{eq:E_w=00005Bfix=00005D rational expression}
\end{equation}
The advantage of translating \eqref{eq:E=00005Bfix=00005D as rational expression elementary approach}
to the language of ideals as in \eqref{eq:E_w=00005Bfix=00005D rational expression}
will soon be apparent. For example, Corollary \ref{cor:contrib is m-rank}
below shows that the summand in \eqref{eq:E_w=00005Bfix=00005D rational expression}
corresponding to $I\le_{\left[1,w\right]}\A$ is of order $\left(q^{N}\right)^{1-\rk I}$.

\subsection{The general case: Theorem \ref{thm:rational general}}

Fix $w\in\F$, $m\in\mathbb{Z}_{\ge1}$ and $\B\in\glm$. Our goal
is to prove that for every large enough $N$, the expectation $\mathbb{E}_{w}\left[\btil\right]$
is a rational function in $q^{N}$. Now we need to count tuples $v_{1},\ldots,v_{m}\in V_{N}$
and $g_{1},\ldots,g_{r}\in\gln$ such that, defining $u_{i}\defi v_{i}.w\left(g_{1},\ldots,g_{r}\right)$
we have 
\begin{equation}
\left(\begin{array}{c}
u_{1}\\
\vdots\\
u_{m}
\end{array}\right)=\B\cdot\left(\begin{array}{c}
v_{1}\\
\vdots\\
v_{m}
\end{array}\right).\label{eq:equations of B}
\end{equation}
As above, we consider the entire trajectories of $v_{1},\ldots,v_{m}$
through the letters of $w$, namely,
\[
\begin{array}{cccc}
v_{1}^{0}=v_{1}~~ & v_{1}^{1}=v_{1}.g_{i_{1}}^{\varepsilon_{1}}~~ & \cdots & v_{1}^{\ell}=v_{1}.w\left(g_{1},\ldots,g_{r}\right)\\
\vdots & \vdots & \ddots & \vdots\\
v_{m}^{0}=v_{m}~~ & v_{m}^{1}=v_{m}.g_{i_{1}}^{\varepsilon_{1}}~~ & \cdots & v_{m}^{\ell}=v_{m}.w\left(g_{1},\ldots,g_{r}\right),
\end{array}
\]
which we denote by $\overline{v}$. Again we group the solutions $\left(g_{1},\ldots,g_{r};\overline{v}\right)$
according to the equations over $\k$ satisfied by $\overline{v}$.
This time, the equations are not given by ideals in $\A$, but rather
by submodules of the right free $\A$-module $\A^{m}$. Formally,
let $E=\left\{ e_{1},\ldots,e_{m}\right\} $ be a basis of the free
module $\A^{m}$. Every element of $\A^{m}$ is a finite linear combination,
with coefficients from $\k$, of monomials $ez$ with $e\in E$ and
$z\in\F$. These monomials are identified with the vertices of $m$
disjoint copies $\C_{1},\ldots,\C_{m}$ of $\mathrm{Cay}\left(\F,B\right)$,
with origins $e_{1},\ldots,e_{m}$, respectively. 

Let $\mathbb{W}$\marginpar{$\mathbb{W}$} denote the union of the
paths $\left[1,w\right]$ in $\C_{1},\ldots,\C_{m}$, so $\mathbb{W}=\bigcup_{e\in E}\left[e,ew\right]$.
Recall that $M\le_{\mathbb{W}}\A^{m}$ means that $M$ is a submodule
of $\A^{m}$ with a generating set supported on $\mathbb{W}$. If
the equations satisfied by the trajectory $\overline{v}$ are precisely
$M|_{\mathbb{W}}$, then $M$ must, in particular, contain the elements
dictated by \eqref{eq:equations of B}, which we denote by \marginpar{$\protect\eqb$}$\eqb\subseteq\A^{m}$.
For example, if $\B=\left(\begin{array}{cc}
2 & 1\\
7 & 3
\end{array}\right)\in\gl_{2}\left(\k\right)$, then $\eqb=\left\{ e_{1}w-2e_{1}-e_{2},e_{2}w-7e_{1}-3e_{2}\right\} $. 

Generalizing the notation from above, if $\T=T_{1}\cup\ldots\cup T_{m}$
is a union of (possibly empty) subtrees $T_{i}\subseteq\C_{i}$, and
$M\le\A^{m}$, define
\begin{eqnarray*}
d^{\T}\left(M\right) & \defi & \dim_{\k}\left(M|_{\mathbb{T}}\right)\\
d_{b}^{\T}\left(M\right) & \defi & \dim_{\k}\left(M|_{D_{b}\left(\T\right)}\right)~~~~~~~~~~~~~~b\in B\\
e_{b}\left(\T\right) & \defi & \left|D_{b}\left(\T\right)\right|.
\end{eqnarray*}
So $\left|D_{b}\left(\mathbb{W}\right)\right|=m\cdot e_{b}\left(w\right)$.
The same argument as above shows that for every $N\ge\max_{b\in B}e_{b}\left(\mathbb{W}\right)$,
\begin{equation}
\mathbb{E}_{w}\left[\btil\right]=\sum_{M\le_{\mathbb{W}}\A^{m}\colon M\supseteq\eqb}\frac{\indep_{m\left(\left|w\right|+1\right)-d^{\mathbb{W}}\left(M\right)}\left(V_{N}\right)}{\prod_{b\in B}\indep_{e_{b}\left(\mathbb{W}\right)-d_{b}^{\mathbb{W}}\left(M\right)}\left(V_{N}\right)}.\label{eq:general rational expression with submodules}
\end{equation}
As there are finitely many submodules $M\le_{\mathbb{W}}\A^{m}$,
the expression \eqref{eq:general rational expression with submodules}
is rational in $q^{N}$. This completes the proof of Theorem \ref{thm:rational general}.

\section{The free group algebra and its ideals\label{sec:The-free-group-algebra}}

This section gathers some known results and some new results about
the free group algebra $\A=\k\left[\F\right]$ and its (as always
in this text, right) ideals, and more generally the free right $\A$-module
$\A^{m}$ and its submodules. Although we assume throughout this paper
that $\k$ is a fixed finite field, most results of the current section
apply to an arbitrary field (not necessarily finite).

The starting point of the story is a 1964 paper of Cohn \cite{cohn1964free}
and a 1969 paper of Lewin (see Footnote \ref{fn:Cohn Vs. Lewin})
which prove that $\A$ is a \emph{free ideal ring}, in the following
sense:
\begin{thm}
\label{thm:Cohn-Lewin}\cite{cohn1964free,lewin1969free} Every ideal
$I\le\A$ is a free $\A$-module. More generally, every submodule
of a free $\A$-module is free. Moreover, every free $\A$-module
$M$ has a unique rank: all bases of $M$ have the same cardinality. 
\end{thm}

See \cite{hog1990short,rosenmann1994ideals,rosenmann1993algorithm}
for additional proofs of this result.

There are two main new results in Section \ref{sec:The-free-group-algebra}.
In Theorem \ref{thm:basis from coincidences} below it is shown that
if an ideal $I\le_{T}\A$ has a generating set supported on some finite
subtree $T$ of $\mathrm{Cay}\left(\F,B\right)$, then it also admits
a basis supported on $T$. Our analysis also leads to Corollary \ref{cor:contrib is m-rank}:
the order of contribution of every ideal $I\le_{\left[1,w\right]}\A$
with $w-1\in I$ to the summation \eqref{eq:E_w=00005Bfix=00005D rational expression}
of $\mathbb{E}_{w}\left[\fix\right]$ is given by its rank. 

Recall that $E=\left\{ e_{1},\ldots,e_{m}\right\} $ is a basis of
the free right module $\A^{m}$, that the elements of $\A^{m}$ are
$\k$-linear combinations of \emph{monomials} $\left\{ ez\right\} _{e\in E,z\in\F}$,
and that we identify these monomials with the vertices of $m$ disjoint
copies $\C_{1},\ldots,\C_{m}$ of $\mathrm{Cay}\left(\F,B\right)$.
Let $\T=T_{1}\cup\ldots\cup T_{m}$\marginpar{$\protect\T$} be a
union of $m$ finite, possibly empty, subtrees $T_{i}\subset\C_{i}$,
and let $M\le_{\T}\A^{m}$ be a submodule generated on $\T$. In order
to study $M$, we expose the vertices of $\mathbb{T}$ one-by-one
and with them the elements of $M$ which are supported on the already-exposed
vertices. Denote by $v_{t}$ the vertex exposed in step $t$, where
$t=1,\ldots,\#\mathrm{vert}\left(\T\right)$, and let $M_{t}$ denote
the submodule generated by $M|_{\left\{ v_{1},\ldots,v_{t}\right\} }$,
so 
\[
\left(0\right)=M_{0}\le M_{1}\le\ldots\le M_{\#\mathrm{vert}\left(\T\right)}=M.
\]

The order by which we expose the vertices of $\T$ should have the
property that as often as possible, we expose neighbours of already-exposed
vertices. Formally, it should be the restriction to $\T$ of a full
order on the vertices of $\C_{1}\cup\ldots\cup\C_{m}$ which abides
to the following assumption.
\begin{defn}[Exploration]
\label{def:exploration} We call a full order $\le$ on the vertices
of $\C_{1}\cup\ldots\cup\C_{m}$ an \emph{exploration }if it is an
enumeration of the vertices (so every vertex has finitely many smaller
vertices), and every vertex is either 
\begin{enumerate}
\item a neighbour of a smaller vertex, or 
\item the smallest vertex in some $\C_{i}$.
\end{enumerate}
An order on a collection $\T=T_{1}\cup\ldots\cup T_{m}$ of (possibly
empty) subtrees $T_{i}\subseteq\C_{i}$ is called an \emph{exploration}
if it is the restriction of an exploration of $\C_{1}\cup\ldots\cup\C_{m}$.
\end{defn}

Note that an order on $\T$ is an exploration if and only if it is
an enumeration of the vertices of $\T$ which satisfies that every
vertex is either a neighbour of a smaller vertex of $\T$ or the first
vertex visited in some $T_{i}$. 

Given a finite $\T$ and $M\le_{\T}\A^{m}$ as above, every step is
either free,\textbf{ }forced or a coincidence, according to the following
conventions.\footnote{This terminology is inspired by \cite{eberhard2021babai}, which,
in turn, was inspired by earlier works dealing with random Schreier
graphs of symmetric groups (see, for example, \cite{broder1987second}).
The analogue in \cite{eberhard2021babai} of our free step is a free
step which is not a coincidence, and the analogue in the same article
of our coincidence is a free step which is also a coincidence.} Assume first that $v_{t}$ is a neighbour of an already-exposed vertex
$u$, and that the edge from $u$ to $v_{t}$ is labeled by $b\in B\cup B^{-1}=\left\{ b_{1}^{\pm1},\ldots,b_{r}^{\pm1}\right\} $:
\begin{equation}
\xymatrix{u\ar@{->}[r]_{b} & v_{t}}
\label{eq:first-type}
\end{equation}
Denote by $D_{b}^{t}$\marginpar{$D_{b}^{t}$} the set of already-exposed
vertices with an outgoing $b$-edge leading to another already-exposed
vertex. This set should include the vertex $u$. If $M|_{D_{b}^{t}}$
contains an element with $u$ in its support, we say the $t$-th step
is \textbf{forced}. If $v_{t}$ is the first vertex we expose in some
$T_{i}$, the $t$-th step is not forced. If a step is not forced,
it is a \textbf{coincidence} if there is an element of $M|_{\left\{ v_{1},\ldots,v_{t}\right\} }$
with $v_{t}$ in its support, and otherwise it is \textbf{free}. 
\begin{lem}
\label{lem:forced-free-coincidence}Let $\T$ and $M\le_{\T}\A^{m}$
be as above and let $v_{1},v_{2},\ldots$ be an exploration of $\mathrm{vert}\left(\T\right)$.
Then step $t$ in the exposure of $M$ along $\T$ is 
\begin{eqnarray*}
\mathrm{forced} & \Longleftrightarrow & M_{t-1}=M_{t}~~~~\mathrm{and~~~~}M|_{\left\{ v_{1},\ldots,v_{t-1}\right\} }\lvertneqq M|_{\left\{ v_{1},\ldots,v_{t}\right\} }\\
\mathrm{free} & \Longleftrightarrow & M_{t-1}=M_{t}~~~~\mathrm{and~~~~}M|_{\left\{ v_{1},\ldots,v_{t-1}\right\} }=M|_{\left\{ v_{1},\ldots,v_{t}\right\} }\\
\mathrm{a~coincidence} & \Longleftrightarrow & M_{t-1}\lvertneqq M_{t}.
\end{eqnarray*}
Moreover, if step $t$ is a coincidence and $f$ is an element of
$M|_{\left\{ v_{1},\ldots,v_{t}\right\} }$ with $v_{t}$ in its support,
then $M_{t}$ is generated by $M_{t-1}$ and $f$.
\end{lem}

Of course, if $M_{t-1}\lvertneqq M_{t}$, then, in particular, $M|_{\left\{ v_{1},\ldots,v_{t-1}\right\} }\lvertneqq M|_{\left\{ v_{1},\ldots,v_{t}\right\} }$.
\begin{proof}
First assume step $t$ is forced. There is some $f\in M|_{D_{b}^{t}}$
with $u$ in its support, and then $f.b\in M|_{\left\{ v_{1},\ldots,v_{t}\right\} }\setminus M|_{\left\{ v_{1},\ldots,v_{t-1}\right\} }$.
Yet $f.b\in M_{t-1}$ and any other element of $M|_{\left\{ v_{1},\ldots,v_{t}\right\} }$,
by subtracting a suitable $\k$-multiple of $f.b$, becomes an element
of $M_{t-1}$. Hence $M_{t-1}=M_{t}$.

If the step is free, then $M|_{\left\{ v_{1},\ldots,v_{t-1}\right\} }=M|_{\left\{ v_{1},\ldots,v_{t}\right\} }$
by definition, and so $M_{t-1}=M_{t}$.

Finally, assume that step $t$ is a coincidence. Fix $N\ge t$, and
consider (row) vectors $u_{1},\ldots,u_{t-1}\in V_{N}=\k^{N}$ with
dependencies corresponding \emph{exactly} to the the elements of $M|_{\left\{ v_{1},\ldots,v_{t-1}\right\} }$,
namely, $\sum_{i=1}^{t-1}\alpha_{i}u_{i}=0$ if and only if $\text{\ensuremath{\sum_{i=1}^{t-1}\alpha_{i}v_{i}\in M}}$.
Let $u_{t}\in V_{N}$ be some vector which is \emph{linearly independent
}of $u_{1},\ldots,u_{t-1}$. For every $b\in B$, there is an element
$g_{b}\in\gl\left(V_{N}\right)$ with $u.g_{b}=u'$ for every $b$-edge
$\left(u,u'\right)$ with $u,u'\in\left\{ u_{1},\ldots,u_{t}\right\} $
(here we rely on that the step is not forced). As in the proof of
Lemma \ref{lem:ideal generated by Delta with C2 does not add elements to tree},
these $g_{b}$'s determine a $\k$-algebra homomorphism $\varphi\colon\A\to\mathrm{End}\left(V_{N}\right)$.
This $\varphi$ gives $V_{N}$ a structure of an $\A$-module. For
every $e\in E$ with $T_{e}$ already visited, pick an arbitrary $v_{e}\in T_{e}\cap\left\{ v_{1},\ldots,v_{t}\right\} $.
Then these monomials $\left\{ v_{e}\right\} $ form a sub-basis of
the free $\A$-module $\A^{m}$, and there is a homomorphism of $\A$-modules
$\psi\colon\A^{m}\to V$ mapping $v_{e}$ to $u_{e}$. By design,
the linear dependencies among $u_{1},\ldots,u_{t}$ correspond precisely
to the elements of $\ker\psi$ supported on $\left\{ v_{1},\ldots,v_{t}\right\} $.
As $u_{t}$ is independent of the rest, we get that 
\[
M_{t-1}\le\ker\psi~~~~~\mathrm{yet}~~~~~M_{t}\nleq\ker\psi,
\]
proving that $M_{t-1}\lvertneqq M_{t}$. 

If step $t$ is a coincidence and $f\in M|_{\left\{ v_{1},\ldots,v_{t}\right\} }$
has $v_{t}$ in its support, then any other element $g\in M_{\left\{ v_{1},\ldots,v_{t}\right\} }$
satisfies that $g-\alpha f\in M|_{\left\{ v_{1},\ldots,v_{t-1}\right\} }$
for some $\alpha=\alpha\left(g\right)\in\k$. Hence the final part
of the statement of the lemma follows.
\end{proof}
\begin{lem}
\label{lem:=000023coincidences independent of order}Let $\T$ and
$M\le_{\T}\A^{m}$ be as above. In every exposure process of $M$
along $\T$ as above, the number of coincidences is the same: it does
not depend on the order of exposure (as long as it is a valid exploration
à la Definition \ref{def:exploration}).
\end{lem}

\begin{proof}
Similarly to the definition of $d^{\T}\left(M\right)$ and $d_{b}^{\T}\left(M\right)$
from Section \ref{sec:Rational-expressions}, let $d^{t}\defi\dim_{\k}\left(M|_{\left\{ v_{1},\ldots,v_{t}\right\} }\right)$
and $d_{b}^{t}\defi\dim_{\k}\left(M|_{D_{b}^{t}}\right)$. Obviously,
$d^{0}=d_{b}^{0}=0$. We now trace how $d^{t}$ and $\sum_{b\in B}d_{b}^{t}$
change with $t$, depending on the three types of steps defined above.
According to the definitions and to Lemma \ref{lem:forced-free-coincidence}:
\begin{itemize}
\item In a forced step, both $d^{t}$ and $\sum_{b}d_{b}^{t}$ increase
by one (compared to $d^{t-1}$ and $\sum_{b}d_{b}^{t-1}$, respectively). 
\item In a free step, both $d^{t}$ and $\sum_{b}d_{b}^{t}$ do not change.
\item In a coincidence, $d^{t}$ increases by one, while $\sum_{b}d_{b}^{t}$
does not change.
\end{itemize}
Therefore, the difference $d^{\T}\left(M\right)-\sum_{b}d_{b}^{\T}\left(M\right)$,
which is, of course, independent of the order of exposure, is equal
to the number of coincidences.
\end{proof}
The proof of Lemma \ref{lem:=000023coincidences independent of order}
actually shows that the number of forced and free steps is also independent
of the order of exposure, but that is not as useful. The proof also
gives the following.
\begin{cor}
\label{cor:contrib as coincidence}Consider the expression \eqref{eq:general rational expression with submodules}
giving $\mathbb{E}_{w}\left[\btil\right]$ as a sum over submodules
$M\le_{\mathbb{W}}\A^{m}$ with $M\supseteq\eqb$. The summand corresponding
to such a submodule $M$ is 
\[
\left(q^{N}\right)^{m-\#\mathrm{coincidences}}\left(1+O\left(\frac{1}{q^{N}}\right)\right),
\]
where we count coincidences in an exposure process of $M$ along $\mathbb{W}$.
\end{cor}

\begin{proof}
The numerator in the summand corresponding to $M$ in \eqref{eq:general rational expression with submodules}
is 
\[
\indep_{m\left(\left|w\right|+1\right)-d^{\mathbb{W}}\left(M\right)}\left(V_{N}\right)=\left(q^{N}\right)^{m\left(\left|w\right|+1\right)-d^{\mathbb{W}}\left(M\right)}\left(1+O\left(\frac{1}{q^{N}}\right)\right).
\]
The denominator is 
\[
\prod_{b}\indep_{e_{b}\left(\mathbb{W}\right)-d_{b}^{\mathbb{W}}\left(M\right)}\left(V_{N}\right)=\left(q^{N}\right)^{\sum_{b}\left[e_{b}\left(\mathbb{W}\right)-d_{b}^{\mathbb{W}}\left(M\right)\right]}\left(1+O\left(\frac{1}{q^{N}}\right)\right).
\]
The result follows as $\sum_{b}e_{b}\left(\mathbb{W}\right)=m\left|w\right|$
and as $d^{\mathbb{W}}\left(M\right)-\sum_{b}d_{b}^{\mathbb{W}}$$\left(M\right)$
is equal to the number of coincidences. 
\end{proof}
Next, we show that the number of coincidences is identical to the
rank of the module $M$. The proof relies on the main theorem of \cite{lewin1969free},
which makes use of the following notion.
\begin{defn}
\label{def:Schreier-transversal}A \textbf{Schreier transversal} of
a submodule $M\le\A^{m}$ is a set $\mathrm{ST}$ of monomials of
$\A^{m}$ which satisfies

$\left(i\right)$ $\mathrm{ST}$ is closed under prefixes: if $ez\in\mathrm{ST}$
with $e\in E$ and $1\ne z\in\F$, and $b\in B\cup B^{-1}$ is the
last letter of $z$, then $ezb^{-1}\in\mathrm{ST}$, and

$\left(ii\right)$ the linear span $\spn_{\k}\left(\mathrm{ST}\right)$
of $\mathrm{ST}$ contains exactly one representative of every coset
of $\A^{m}/M$.
\end{defn}

It is not hard to show that every $M\le\A^{m}$ admits Schreier transversals
--- see \cite[pp.~456-457]{lewin1969free} for an argument as well
as for a concrete construction. Note that a Schreier transversal $\mathrm{ST}$
consists of the vertices in a collection of (possibly infinite) subtrees,
one in $\C_{i}$ for every $i=1,\ldots,m$. The main theorem of \cite{lewin1969free}
is that one may construct a basis for $M$ which is, roughly, in one-to-one
correspondence with the outgoing directed edges from $\mathrm{ST}$
to its complement. Although a version of this theorem holds for any
submodule of any free $\A$-module, we only need the case of finitely
generated $\A$-modules.
\begin{thm}
\label{thm:Lewin}\cite[Thm.~1]{lewin1969free} Let $M\le\A^{m}$
be a submodule, and let $\mathrm{ST}$ be a Schreier transversal of
$M$. For every $f\in\A^{m}$, denote by $\phi\left(f\right)$ the
representative of $f+M$ in $\spn_{\k}\left(\mathrm{ST}\right)$.
Then the set 
\begin{equation}
\left\{ ezb-\phi\left(ezb\right)\,\middle|\,ez\in\mathrm{ST},b\in B,ezb\notin\mathrm{ST}\right\} ~~~\cup~~~\left\{ e-\phi\left(e\right)\,\middle|\,e\in E\setminus\mathrm{ST}\right\} \label{eq:Lewins basis}
\end{equation}
is a basis for $M$ (as a free $\A$-module).
\end{thm}

We stress that in \eqref{eq:Lewins basis}, $b$ is a proper basis
element and not the inverse of one. 
\begin{thm}
\label{thm:basis from coincidences}Let $\T=T_{1}\cup\ldots\cup T_{m}$
be a collection of finite, possibly empty, subtrees $T_{i}\subset\C_{i}$
and assume that $M\le_{\T}\A^{m}$. Then the number of coincidences
in an exposure of $M$ along $\T$ is \emph{equal} to $\rk M$. 

Moreover, $M$ admits a basis supported on $\T$. In fact, every set
of elements $f_{1},\ldots,f_{\rk M}$ supported on $\T$ with the
leading vertex\footnote{Given a full order on the vertices of $\C_{1}\cup\ldots\cup\C_{m}$,
the leading vertex of $0\ne f\in\A$ is the largest vertex in the
support of $f$.} of $f_{i}$ being the monomial exposed in the $i$-th coincidence
is a basis of $M$.
\end{thm}

\begin{proof}
Let $s=\rk M$. Let $\mathrm{ST}$ be a Schreier transversal for $M$.
Then the basis \eqref{eq:Lewins basis} contains $s$ elements. Let
$\mathbb{S}$ be the smallest collection of finite subtrees (one in
each $\C_{i}$) which contain the whole support of these $s$ basis
elements. Note that $\mathbb{S}$ contains exactly $s$ vertices (monomials)
outside $\mathrm{ST}$, and all these vertices are either leaves or
isolated in $\mathbb{S}$ (namely, these are vertices of degree $1$
or $0$ in $\mathbb{S}$). Consider an exposure process of $M$ along
$\mathbb{S}$ according to some exploration such that the vertices
of $\mathbb{S}\cap\mathrm{ST}$ are exposed first and only then the
remaining $s$ vertices. Because there is no non-zero element of $M$
supported on $\mathrm{ST}$, the first $\left|\mathbb{S}\right|-s$
steps are all free. 

We claim that the remaining $s$ steps are all coincidences. Indeed,
\[
\left(0\right)=M_{\left|\mathbb{S}\right|-s}\le M_{\left|\mathbb{S}\right|-s+1}\le\ldots\le M_{\left|\mathbb{S}\right|-1}\le M_{\left|\mathbb{S}\right|}=M.
\]
For $i=1,\ldots,s$, let $f_{i}\in M_{\left|\mathbb{S}\right|-s+i}$
be the basis element from \eqref{eq:Lewins basis} with the vertex
exposed in step $\left|\mathbb{S}\right|-s+i$ in its support. Clearly,
$f_{i}\in M_{\left|\mathbb{S}\right|-s+i}$. By induction, $M_{\left|\mathbb{S}\right|-s+i}=\left(f_{1},\ldots,f_{i}\right)$.
Indeed, $M_{\left|\mathbb{S}\right|-s+1}=\left(f_{1}\right)$, and
if $M_{\left|\mathbb{S}\right|-s+i-1}=\left(f_{1},\ldots,f_{i-1}\right)$
then either step $i$ is a coincidence and then $M_{\left|\mathbb{S}\right|-s+i}=\left(M_{\left|\mathbb{S}\right|-s+i-1},f_{i}\right)$
by Lemma \ref{lem:forced-free-coincidence}, or step is not a coincidence
and then $M_{\left|\mathbb{S}\right|-s+i}=M_{\left|\mathbb{S}\right|-s+i-1}$.
But $\left\{ f_{1},\ldots,f_{s}\right\} $ is a basis by Theorem \ref{thm:Lewin},
so $f_{i}\notin\left(f_{1},\ldots,f_{i-1}\right)=M_{\left|\mathbb{S}\right|-s+i-1}$.
We conclude that $M_{\left|\mathbb{S}\right|-s+i-1}\lneqq M_{\left|\mathbb{S}\right|-s+i}$
so all these $s$ steps are indeed coincidences by Lemma \ref{lem:forced-free-coincidence}.

Now consider the collection of finite trees $\mathbb{U}$, which is
the collection of smallest subtrees (one in each $\C_{i}$) which
contains both $\mathbb{S}$ and the given $\mathbb{T}$. Expose $M$
along $\mathbb{U}$ by two different explorations. In the first order,
expose $\mathbb{S}$ first and then the remaining vertices of $\mathbb{U}$.
There are exactly $s$ coincidences: after we exposed all of $\mathbb{S}$,
we have $M_{\left|\mathbb{\mathbb{S}}\right|}=M$, so no more coincidences
are possible, by Lemma \ref{lem:forced-free-coincidence}. By Lemma
\ref{lem:=000023coincidences independent of order}, $s$ is also
the number of coincidences when we first expose $\mathbb{T}$ and
then the remaining vertices of $\mathbb{U}\setminus\mathbb{T}$. But
again, because $M$ is generated on $\T$, we have $M|_{\mathbb{T}}=M$
and there are no more coincidences after exposing $\mathbb{\T}$.
This shows there are exactly $s=\rk M$ coincidences in an exposure
of $M$ along $\mathbb{T}$.\\

For the second statement, assume that $M\le_{\T}\A^{m}$ and consider
an exposure of $M$ along $\T$. If step $t$ is a coincidence, then
by Lemma \ref{lem:forced-free-coincidence}, $M_{t}=\left(M_{t-1},f_{t}\right)$
where $f_{t}\in M|_{\left\{ v_{1},\ldots,v_{t}\right\} }$ with $v_{t}$
in its support. Hence $M=\left(f_{t_{1}},\ldots,f_{t_{s}}\right)$
where $t_{1},\ldots,t_{s}$ are the $s$ coincidences. But every set
of size $s=\rk M$ which generates $M$ is a basis \cite[Prop. 2.2]{cohn1964free}. 
\end{proof}
From Theorem \ref{thm:basis from coincidences} and Corollary \ref{cor:contrib as coincidence}
we immediately obtain that the order of contribution of a given ideal
to $\mathbb{E}_{w}\left[\fix\right]$ is given by its rank:
\begin{cor}
\label{cor:contrib is m-rank}Consider the expression \eqref{eq:general rational expression with submodules}
giving $\mathbb{E}_{w}\left[\btil\right]$ as a sum over submodules
$M\le_{\mathbb{W}}\A^{m}$ with $M\supseteq\eqb$. The summand corresponding
to such a submodule $M$ is 
\[
\left(q^{N}\right)^{m-\rk M}\left(1+O\left(\frac{1}{q^{N}}\right)\right).
\]
\end{cor}

In Section \ref{subsec:powers general} we show that there are no
submodules of rank $<m$ containing $\eqb$, and so $\lim_{N\to\infty}\mathbb{E}_{w}\left[\btil\right]$,
the limit from Theorem \ref{thm:limit powers general}, is equal to
the number of rank-$m$ submodules supported on $\mathbb{W}$ and
containing $\eqb$. Using Corollary \ref{cor:alg extensions generated on a given tree},
one can show that in this case the restriction to submodules supported
on $\mathbb{W}$ is redundant -- we elaborate in Section \ref{sec:Powers}.

Recall that Definition \ref{def:q-primitivity-rank} introduced $\pi_{q}\left(w\right)$
and $\crit_{q}\left(w\right)$ for every $w\in\F$. Theorem \ref{thm:basis from coincidences}
can also be used to show that the set $\crit_{q}\left(w\right)$ is
always finite. If $N$ is a free $\A$-module and $L\le N$ a submodule
(and therefore free as well), we say that $L$ is a free factor of
$N$ if some basis (and hence every basis) of $L$ can be extended
to a basis of $N$.
\begin{cor}
\label{cor:alg extensions generated on a given tree}Let $M\le N\le\A^{m}$
be two finitely generated submodules of $\A^{m}$, and assume that
there is no intermediate submodule which is a proper free factor of
$N$.\footnote{In analogy with subgroups of the free group $\F$, one may say that
$N$ is an \emph{algebraic} extension of $M$ -- see, e.g., \cite[Def.~2.1]{PP15}.} Namely, if $M\le L\le N$ and $L$ is a free factor of $N$ then
$L=N$. If $M\le_{\T}\A^{m}$ with $\T$ a union of subtrees as above,
then $N\le_{\T}\A^{m}$.
\end{cor}

\begin{proof}
Take a collection $\mathbb{S}$ of subtrees which contains $\T$ and
such that $N\le_{\mathbb{S}}\A^{m}$. Expose $N$ along $\mathbb{S}$
according to some exploration which first exposes $\mathbb{T}$ and
then the remaining vertices. Let $N_{\left|\T\right|}=\left(N|_{\mathbb{T}}\right)$
denote the submodule of $N$ generated by the elements of $N$ supported
on $\T$. Clearly, $M\le N_{\left|\T\right|}$, and using Theorem
\ref{thm:basis from coincidences} to construct a basis for $N$ from
the coincidences of this exposure process, we get that $N_{\left|\T\right|}$
is a free factor of $N$. By assumption we therefore have $N_{\left|\T\right|}=N$,
so $N\le_{\T}\A^{m}$.
\end{proof}
In the following corollary we use the fact that $\k$ is finite. For
example, the element $xyx^{-1}y^{-1}-1\in\A$ has critical ideals
$\left\{ \left(\alpha x-1,\beta y-1\right)\,\middle|\,\alpha,\beta\in\k^{*}\right\} $,
which is an infinite set if $\k$ is infinite. For a general element
$f\in\A$, we say that an ideal $I\le\A$ is \emph{critical} for $f$
if it contains $f$ as an imprimitive element, and it has minimal
rank among all such ideals.
\begin{cor}
\label{cor:finitely many critical extensions}Let $f\in\A$, and suppose
that the subtree $T\subseteq\mathrm{Cay}\left(\F,B\right)$ supports
$f$. Then any critical ideal of $f$ is generated on $T$. In particular,
$\crit_{q}\left(w\right)$ is finite for every word $w\in\F$ and
every prime power $q$.
\end{cor}

\begin{proof}
Assume that $I\le\A$ is critical for $f$, namely, that it is an
ideal of minimal rank which contains $f$ as an imprimitive element.
Assume that $f\in J\le I$ and that $J$ is a free factor of $I$.
In particular, $\rk J\le\rk I$. If $f$ is primitive in $J$, it
is also primitive in $I$, which is impossible. So $f$ is imprimitive
in $J$. But $I$ is critical for $f$, and so $\rk J=\rk I$ and
$J=I$. Therefore the assumption of Corollary \ref{cor:alg extensions generated on a given tree}
applies to $\left(f\right)\le I$, and for every finite subtree $T\subseteq\mathrm{Cay}\left(\F,B\right)$
supporting $f$, we have $I\le_{T}\A$. For every $f\in\A$ we may
take $T$ finite, and if $\k$ is finite, there are only finitely
many ideals supported on $T$.
\end{proof}

\subsection{Properties of the $q$-primitivity rank\label{subsec:properties-of-the-q-primitivity-rank}}

In the current subsection \ref{subsec:properties-of-the-q-primitivity-rank},
we prove some basic properties of the $q$-primitivity rank of words.
Let $H$ be a subgroup of the free group $\F$. We associate to $H$
two (right) ideals of interest. The first is its augmentation ideal
$I_{H}\le\k\left[H\right]$, defined as the kernel of the augmentation
map $\varepsilon_{H}:\k\left[H\right]\rightarrow\k$ where $\varepsilon_{H}\left(\sum_{h\in H}\alpha_{h}h\right)=\sum_{h\in H}\alpha_{h}$.
If $\left\{ h_{\beta}\right\} _{\beta\in B}$ is a basis for $H$
then $\left\{ h_{\beta}-1\right\} _{\beta\in B}$ is a basis for $I_{H}$
\cite[Prop.~4.8]{cohen1972groups}, and in particular $\text{rk}I_{H}=\text{rk}H$.
The second, when considering $H$ as a subgroup of $\F,$ is the (right)
ideal $J_{H}$ of $\mathcal{A}=K\left[\boldsymbol{F}\right]$ generated
by $\left\{ h-1\right\} _{h\in H}$. The following proposition also
follows from \cite[Chap.~4]{cohen1972groups}, but as it is not stated
there explicitly, we add a short proof. 
\begin{prop}
\label{prop:basis for J_H}If $\left\{ h_{\beta}\right\} _{\beta\in B}$
is a basis for $H$ then $\left\{ h_{\beta}-1\right\} _{\beta\in B}$
is a basis for $J_{H}.$ In particular, $\text{rk}J_{H}=\text{rk}H$.
\end{prop}

\begin{proof}
Since $\left\{ h_{\beta}-1\right\} _{\beta\in B}$ already generates
$I_{H}$ in $K\left[H\right]$, it generates $h-1$ for any $h\in H,$
and is thus a generating set for $J_{H}.$ Let $T$ be a right transversal
for $H$ in $\F$ (i.e., a set of representatives of the right cosets
of $H$). Then for every $t\in T$ the set $K\left[H\right]t$ of
elements of $\A$ supported on the coset $Ht$ forms a left $K\left[H\right]$-module,
and the group algebra $\A$ admits a left $K\left[H\right]$-module
decomposition $\A=\bigoplus_{t\in T}K\left[H\right]t$. Let $P_{Ht}$:$\A\to K\left[H\right]t$
be the projections induced by this decomposition. Suppose now that
there is a relation $\sum_{\beta\in B}\left(h_{\beta}-1\right)a_{\beta}=0$
for some coefficients $\left\{ a_{\beta}\right\} _{\beta\in B}$ in
$\A$. For every $t\in T,$ applying the left $K\left[H\right]$-module
map $P_{Ht}$ to both sides yields the relation $\sum_{\beta\in B}\left(h_{\beta}-1\right)P_{Ht}\left(a_{\beta}\right)=0$,
and multiplying by $t^{-1}$ gives $\sum_{\beta\in B}\left(h_{\beta}-1\right)\left(P_{Ht}\left(a_{\beta}\right)t^{-1}\right)=0$.
Since $P_{Ht}\left(a_{\beta}\right)t^{-1}\in K\left[H\right]$ and
$\left\{ h_{\beta}-1\right\} _{\beta\in B}$ is a basis (for $I_{H}),$we
deduce that $P_{Ht}\left(a_{\beta}\right)=0$ for every $\beta\in B$.
Thus, $a_{\beta}=\sum_{t\in T}P_{Ht}\left(a_{\beta}\right)=0$ for
every $\beta\in B$.
\end{proof}
\begin{prop}
\label{prop:Umirbaev}Let $H\leq F$ be finitely generated and let
$w\in F$. If $w-1$ is primitive in $J_{H}$ then $w$ is primitive
in $H$.
\end{prop}

\begin{proof}
Assume that $w-1$ is primitive in $J_{H}$. As $w-1\in J_{H}$, by
\cite[Lem.~4.1]{cohen1972groups}, $w$ lies in $H$. Fix a basis
$h_{1},h_{2},...,h_{k}$ for $H$. Then $\left\{ h_{i}-1\right\} _{i=1}^{k}$
is a basis for $I_{H}$ and $w-1\in I_{H}$, so we can write (uniquely)
$w-1=\sum_{i=1}^{k}\left(h_{i}-1\right)a_{i}$ for some coefficients
$a_{i}\in K\left[H\right].$ By a Theorem of Umirbaev\footnote{Umirbaev's result is actually stated for free group rings over the
integers. However, the proof uses no specific properties of $\mathbb{Z}$
and hence also applies, \textit{mutatis mutandis, }to the field $K$.} \cite[Cor.~on~page~184]{umirbaev1994primitive}, to deduce that $w$
is primitive in $H$ it is enough to show that the coefficients $\left\{ a_{i}\right\} _{i=1}^{k}$
form a left-invertible column in the sense that there exist $u_{1},u_{2},...,u_{k}\in K\left[H\right]$
such that $\sum_{i=1}^{k}u_{i}a_{i}=1.$\\
Since $w-1$ is primitive in $J_{H}$, there exist some elements $f_{2},...,f_{k}\in J_{H}$
completing $w-1$ to a basis of $J_{H}$. By Proposition \ref{prop:basis for J_H},
$\left\{ h_{i}-1\right\} _{i=1}^{k}$ is, too, a basis for $J_{H}$.
Let $C\in M_{kk}\left(\mathcal{A}\right)$ be a change-of-basis matrix
satisfying $\left(h_{1}-1,h_{2}-1,...,h_{k}-1\right)C=\left(w-1,f_{2},...,f_{k}\right),$
where by uniqueness of presenting $w-1$ in the basis $\left\{ h_{i}-1\right\} _{i=1}^{k}$
the first column of $C$ is 
\[
\begin{pmatrix}a_{1}\\
\vdots\\
a_{k}
\end{pmatrix}.
\]
As one can also change basis in the other direction, there exists
some $D\in M_{kk}\left(\mathcal{A}\right)$ such that 
\[
\left(h_{1}-1,h_{2}-1,...,h_{k}-1\right)=\left(w-1,f_{2},...,f_{k}\right)D.
\]
Thus, $\left(w-1,f_{2},...,f_{k}\right)DC=\left(w-1,f_{2},...,f_{k}\right)$,
which by the uniqueness of presentation implies that $DC$ is the
identity matrix. In particular, the first row of $D$ which we denote
by $\left(d_{1},d_{2},...,d_{k}\right)$ is a left inverse to the
first column of $C$ in the sense that 
\begin{equation}
\sum_{i=1}^{k}d_{i}a_{i}=1.\label{eq: left inverse}
\end{equation}
We next show that the elements $\left\{ d_{i}\right\} _{i=1}^{k}$
can be replaced by elements $\left\{ u_{i}\right\} _{i=1}^{k}$ lying
in $K\left[H\right]$. Let $T$ be a left transversal for $H$ in
$\boldsymbol{F}.$ Then as a right $K\left[H\right]$-module, $\A$
decomposes as $\A=\bigoplus_{t\in T}tK\left[H\right]$. Denote by
$P_{tH}$ the projection onto the summand corresponding to $t.$ Then
applying the right $K\left[H\right]$-module map $P_{H}$ to Equation
\ref{eq: left inverse} gives $\sum_{i=1}^{k}P_{H}\left(d_{i}\right)a_{i}=1$.
We finish by letting $u_{i}=P_{H}\left(d_{i}\right).$
\end{proof}
\begin{lem}
\label{lem:prim means prim in every intermediate}Let $J\leq\mathcal{A}$
be an ideal and $f\in J$ a primitive element. Then $f$ is primitive
in every intermediate ideal $f\in I\leq J$.
\end{lem}

\begin{proof}
Since $f$ is primitive in $J$ we can write $J=\left(f\right)\oplus M$
for some ideal $M$. We claim that $I=\left(f\right)\oplus\left(M\cap I\right)$.
The directness is obvious ($M$ already intersects $\left(f\right)$
trivially). It remains to show that $I$ is indeed the sum of the
two summands. Let $a\in I.$ Then $a\in J$ and thus can be decomposed
as $a=a_{1}+m$ where $a_{1}\in\left(f\right)$ and $m\in M.$ But
then $m=a-a_{1}\in I$ and so $m\in M\cap I$.
\end{proof}
In the following corollary we do not assume that $H$ is finitely
generated.
\begin{cor}
\label{cor:w prim in H  iff   w-1 prim in J_H}Let $H\leq F$ and
let $w\in F.$ Then $w$ is primitive in $H$ if and only if $w-1$
is primitive in $J_{H}$.
\end{cor}

\begin{proof}
One implication is immediate from Proposition \ref{prop:basis for J_H}.
For the other implication, suppose that $w-1$ is primitive in $J_{H}.$
Let $S$ be a basis for $H$. When writing $w-1$ in the basis $\left\{ s-1\right\} _{s\in S}$
of $J_{H}$, only finitely many basis elements appear, so there exist
some $h_{1},...,h_{k}\in S$ and $a_{1},..,a_{k}\in\mathcal{A}$ such
that $w-1=$$\sum_{i=1}^{k}\left(h_{i}-1\right)a_{i}.$ Let $H'=\left\langle h_{1},h_{2},...,h_{k}\right\rangle $.
Then $w-1$ lies in $J_{H'}$ and by Lemma \ref{lem:prim means prim in every intermediate}
it is primitive in it. Since $H'$ is finitely generated, Proposition
\ref{prop:Umirbaev} guarantees that $w$ is primitive in $H'$. Since
the relation of being a free factor is transitive and $\left\langle w\right\rangle \overset{*}{\leq}H'\overset{*}{\leq}H$
we are done.
\end{proof}
We can now prove Proposition \ref{prop:pi_q =00005Cle pi} stating
that for every prime power $q$, the $q$-primitivity rank is bounded
from above by the ordinary primitivity rank, namely, $\pi_{q}\left(w\right)\le\pi\left(w\right)$
for every $w\in\F$.
\begin{proof}[Proof of Proposition \ref{prop:pi_q =00005Cle pi}]
 Let $w\in\F$. The ordinary primitivity rank of a word is a non-negative
integer or $\infty$. We first deal with two trivial cases: if $\pi\left(w\right)=\infty$
then there is nothing to prove, and if $\pi\left(w\right)=0$ then
$w=1$ and so $w-1=0$ is contained in the rank-$0$ trivial ideal
of $\mathcal{A}$ as an imprimitive element, so $\pi_{q}\left(w\right)=0$
as well. Suppose now that $\pi\left(w\right)=k\notin\left\{ 0,\infty\right\} $
and let $H$ be a critical subgroup for $w$ in $\F$, i.e., a subgroup
of $\F$ of rank $k$ containing $w$ as an imprimitive element. The
ideal $J_{H}\leq\mathcal{A}$ contains $w-1$ by its definition as
$w\in H$, it contains $w-1$ as an imprimitive element by Corollary
\ref{cor:w prim in H  iff   w-1 prim in J_H}, it has rank $\text{rk}J_{H}=k$
by Proposition \ref{prop:basis for J_H}, and it is a proper ideal
of $\mathcal{A}$ since it is contained in the augmentation ideal
$I_{\F}\lneqq\A.$ We conclude that $\pi_{q}\left(w\right)\leq k=\pi\left(w\right).$
\end{proof}
If $w\in\F$ is primitive, then (analogously to Lemma \ref{lem:prim means prim in every intermediate}),
$w$ is primitive in any subgroup of $\F$ containing it (see, e.g.~\cite[Claim 2.5]{Puder2014}).
In particular, it has primitivity rank $\pi\left(w\right)=\infty$.
Furthermore, any imprimitive word $w\in\F$ must have $\pi\left(w\right)\leq\text{rk\ensuremath{\F}}$
since it is already not primitive in $\F.$ Thus, the primitivity
rank of words in $\F$ takes values in $\left\{ 0,1,2,...,\text{rk}\F\right\} \cup\left\{ \infty\right\} .$
We next show that analogous statements hold when $\pi$ is replaced
with $\pi_{q}$.
\begin{prop}
\label{prop:prim iff pi_q=00003Dinfty}For every $w\in\F$ and prime
power $q$, $\pi_{q}\left(w\right)=\infty$ if and only if $w$ is
primitive in $\F$$.$
\end{prop}

\begin{proof}
If $\pi_{q}\left(w\right)=\infty$ then $w-1$ must be primitive in
$J_{\F}$, which implies by Proposition \ref{prop:Umirbaev} that
$w$ is primitive in $\F$. Conversely, let $w\in\F$ be primitive.
Then there exists some automorphism $\psi\in Aut\left(\F\right)$
such that $\psi\left(w\right)=b_{1}$ (recall that $\left\{ b_{1},\ldots,b_{r}\right\} $
is our fixed basis of $\F$). The automorphism $\psi$ naturally extends
(linearly) to an automorphism of the group ring $\psi:\A\rightarrow\A.$
Since $\psi$ maps ideals to ideals and bases to bases, it is enough
to show that $\psi\left(w-1\right)=b_{1}-1$ is primitive in every
ideal containing it. Suppose it is not, and let $I$ be a critical
ideal for $b_{1}-1$. By Corollary \ref{cor:finitely many critical extensions},
$I$ is generated on $T=\left\{ 1,b_{1}\right\} .$ Let $f\in I\mid_{T}$.
Then $f=\beta b_{1}-\alpha$ for some $\alpha,\beta\in\k$. By the
definition of a critical ideal, $I$ is a proper ideal and so $\alpha,\beta$
must be equal because their difference lies in $I$:
\[
\beta-\alpha=f-\left(b_{1}-1\right)\beta\in I.
\]
Thus, $I\mid_{T}=\spn_{\k}\left\{ b_{1}-1\right\} $ and since $I$
is supported on $T$, $I$ must be the principal right ideal $I=\left(b_{1}-1\right)$
in which $b_{1}-1$ is primitive, a contradiction.
\end{proof}
\begin{cor}
\label{cor:values of pi_q}For every $w\in\F$ and every prime power
$q$, $\pi_{q}\left(w\right)\in\left\{ 0,1,2,...,\text{rk}\F\right\} \cup\left\{ \infty\right\} .$
\end{cor}

\begin{proof}
Let $w\in\F$. If $w$ is primitive in $\F$ then by Proposition \ref{prop:prim iff pi_q=00003Dinfty}
$\pi_{q}\left(w\right)=\infty$. Otherwise, by Corollary \ref{cor:w prim in H  iff   w-1 prim in J_H},
$w-1$ is already imprimitive in $J_{\F}$ and so $\pi_{q}\left(w\right)\leq\text{rk}J_{\F}=\text{rk}\F$.
\end{proof}

\section{Powers and the limit of expected values of stable functions\label{sec:Powers}}

In this section we prove Theorems \ref{thm:limit powers fix} and
\ref{thm:limit powers general}: if $w\ne1$, then $\lim_{N\to\infty}\mathbb{E}_{w}\left[\fix\right]$
and, more generally, $\lim_{N\to\infty}\mathbb{E}_{w}\left[\btil\right]$,
exist. Moreover, if we write $w=u^{d}$ with $u$ non-power and $d\ge1$,
then the limit depends only on $d$, and, in particular, $\lim_{N\to\infty}\mathbb{E}_{w}\left[\fix\right]$
is equal to the number of monic divisors of the polynomial $x^{d}-1\in\k\left[x\right]$. 

As mentioned in Section \ref{sec:Introduction}, a special case of
Theorem \ref{thm:limit powers general}, which includes Theorem \ref{thm:limit powers fix},
first appeared in \cite{West19} and, independently, in \cite{eberhard2021babai}.
Here, we prove the full version of the theorem while following the
strategy from \cite{eberhard2021babai}, which is more elegant than
the one in \cite{West19}. We use here slightly different notions
and give more details than in \cite{eberhard2021babai}. As the proof
is subtle, and for the sake of readability, we first describe the
proof of the special case which is Theorem \ref{thm:limit powers fix}.

\subsection{\label{subsec:powers fix}$\lim_{N\to\infty}\mathbb{E}_{w}\left[\protect\fix\right]$
and the proof of Theorem \ref{thm:limit powers fix}}

From \eqref{eq:E_w=00005Bfix=00005D rational expression} and Corollary
\ref{cor:contrib is m-rank} it follows that 
\[
\mathbb{E}_{w}\left[\fix\right]=\sum_{I\le_{\left[1,w\right]}\A\colon I\ni w-1}\left(q^{N}\right)^{1-\rk I}\left(1+O\left(\frac{1}{q^{N}}\right)\right).
\]
Clearly, as $w\ne1$, an ideal containing $w-1$ has rank at least
$1$. So
\begin{equation}
\mathbb{E}_{w}\left[\fix\right]=\left|\left\{ I\le_{\left[1,w\right]}\A\,\middle|\,I\ni w-1~\mathrm{and}~\rk I=1\right\} \right|+O\left(\frac{1}{q^{N}}\right).\label{eq:lim fix =00003D =000023 rank-1 ideals in =00005B1,w=00005D}
\end{equation}
By Definition \ref{def:q-primitivity-rank}, the only non-critical
rank-1 ideals containing $w-1$ are $\left(w-1\right)$ and $\left(1\right)$,
which are both generated on $\left[1,w\right]$. Any other rank-1
ideal containing $w-1$ is critical, and Corollary \ref{cor:finitely many critical extensions}
guarantees that such ideals are supported on $\left[1,w\right]$.
We obtain that
\begin{equation}
\mathbb{E}_{w}\left[\fix\right]=\left|\left\{ I\le\A\,\middle|\,I\ni w-1~\mathrm{and}~\rk I=1\right\} \right|+O\left(\frac{1}{q^{N}}\right).\label{eq:lim fix =00003D =000023 rank-1 ideals}
\end{equation}
This proves:
\begin{cor}
\label{cor:Conjecture-when-pi=00003D1}Conjecture \ref{conj:general pi and fixed vectors}
holds in the case $\pi_{q}\left(w\right)=1$. Namely, in this case
\[
\mathbb{E}_{w}\left[\fix\right]=2+\left|\crit_{q}\left(w\right)\right|+O\left(\frac{1}{q^{N}}\right).
\]
\end{cor}

In order to prove Theorem \ref{thm:limit powers fix}, it remains
to show that if $w=u^{d}$ with $u$ a non-power and $d\ge1$, then
the ideals $I$ in \eqref{eq:lim fix =00003D =000023 rank-1 ideals}
are precisely $\left\{ \left(p\left(u\right)\right)\,\middle|\,p|x^{d}-1\in\k\left[x\right],p~\mathrm{monic}\right\} $.
First, as any automorphism of $\F$ gives rise to an automorphism
of $\A$, we may replace $w$ by any element in its $\mathrm{Aut}\F$-orbit,
and, in particular, assume that $w$ is cyclically reduced. 

Throughout Section \ref{sec:Powers}, we use the ShortLex order on
monomials in $\A^{m}$ and their finite subsets.
\begin{defn}
\label{def:ShortLex} Fix an arbitrary full order on the basis $E$
of $\A^{m}$, say $e_{1}<e_{2}<\ldots<e_{m}$. Fix an arbitrary full
order on $B\cup B^{-1}$, say $b_{1}<b_{1}^{-1}<b_{2}<\ldots<b_{r}^{-1}$.
The \textbf{ShortLex} order on the monomials $\left\{ ez\right\} _{e\in E,z\in\F}$
is defined by first comparing the length of $z$ (shorter words are
smaller) and using lexicographic order to compare $ez$ with $e'z'$
when $\left|z\right|=\left|z'\right|$. This order induces a full
order on finite sets of monomials by comparing the leading monomial
in each set, breaking ties by looking at the second monomials, and
so on (the empty set is the smallest of all finite sets of monomials).
Finally, we get a pre-order on the elements of $\A^{m}$ by comparing
their supports. An element $f\in\A^{m}$ is called \textbf{monic}
if the $\k$-coefficient of the leading monomial is $1$.
\end{defn}

For example, $\alpha e_{2}b_{1}^{-1}b_{2}<\beta e_{2}b_{1}^{-1}b_{2}+e_{3}b_{1}$
and $e_{1}b_{1}b_{2}<e_{1}b_{3}b_{1}<e_{2}b_{1}b_{2}$ (here $\alpha,\beta\in\k^{*}$).
This ShortLex order is the same one used in \cite{rosenmann1993algorithm}.
(The order used in \cite{lewin1969free} is not quite the same: it
uses length and then \emph{reverse} lexicographic order, and it also
fixes a full order on $\k$ resulting in a full order on $\A^{m}$,
rather than a mere pre-order.) \\

Now, let $I\le\A$ be a rank-1 ideal containing $w-1$. As noted above,
$I$ is generated on $\left[1,w\right]$. In the notation of Section
\ref{sec:The-free-group-algebra}, consider the exposure of $I$ along
the subtree $\left[1,w\right]$, starting with the monomial $1$ and
ending with the monomial $w$. (This happens to be the restriction
of ShortLex to $\left[1,w\right]$.) We shift the indices of the vertices
by one with respect to Section \ref{sec:The-free-group-algebra},
and define $v_{0}=1,\ldots,v_{\left|w\right|}=w$. By Theorem \ref{thm:basis from coincidences},
there is exactly one coincidence is this exposure.\footnote{We remark that to analyze $\lim_{N\to\infty}\mathbb{E}_{w}\left[\fix\right]$,
one does not really need to go through Theorem \ref{thm:basis from coincidences},
nor even consider the rank of ideals. Rather, it is enough to rely
on Corollary \ref{cor:contrib as coincidence}.} In an exposure along a path, a free step is followed by either another
free step or by a coincidence, and in this particular path, the last
step is not free. Thus, the first non-free step must be a coincidence,
and the following steps must all be forced. Namely, if $v_{t}$ is
exposed in a coincidence, $t\in\left\{ 0,1,\ldots,\left|w\right|\right\} $,
then $v_{0},\ldots,v_{t-1}$ are free steps and $v_{t+1},\ldots,v_{\left|w\right|}$
are forced. Denote by $f_{I}\in I$ the monic element supported on
$\left[1,w\right]$ with $v_{t}$ its leading monomial. By Theorem
\ref{thm:basis from coincidences}, $I=\left(f_{I}\right)$. Thus,
the map 
\begin{equation}
I\mapsto f_{I}\label{eq:J mapsto f_J}
\end{equation}
is a one-to-one correspondence.
\begin{lem}
\label{lem:divisors of x^d-1}The ideals $I$ for which $f_{I}$ is
supported on $\left\langle u\right\rangle $ are in one-to-one correspondence
with monic polynomials in $\k\left[x\right]$ dividing $x^{d}-1$.
\end{lem}

\begin{proof}
Consider the subalgebra $\k\left[\left\langle u\right\rangle \right]$
of $\A=\k\left[\F\right]$, the elements of which are linear combinations
of the elements in $\left\langle u\right\rangle =\left\{ u^{i}\,\middle|\,i\in\mathbb{Z}\right\} $.
For every $z\in\F$ and $f\in\k\left[\left\langle u\right\rangle \right]$,
if $z\notin\left\langle u\right\rangle $ then $fz$ is supported
on $\left\langle u\right\rangle z$, which is disjoint from $\left\langle u\right\rangle $.
Thus, if $f\in\k\left[\left\langle u\right\rangle \right]$ and $w-1=u^{d}-1\in\left(f\right)\defi f\A$,
then $w-1$ is also an element of $f\k\left[\left\langle u\right\rangle \right]$,
the ideal generated by $f$ inside $\k\left[\left\langle u\right\rangle \right]$.
Now
\[
\k\left[\left\langle u\right\rangle \right]\cong\k\left[\mathbb{Z}\right]\cong\k\left[x,x^{-1}\right]
\]
is a commutative ring (a principal ideal domain, in fact). If $p\in\k\left[x,x^{-1}\right]$
satisfies $\left(p\right)\ni x^{d}-1$, we may assume, by multiplying
$p$ by a unit element if need be, that $p\in\k\left[x\right]$, $p$
monic, and $p|x^{d}-1$ in $\k\left[x\right]$. Moreover, such $p\in\k\left[x\right]$
is determined uniquely by the ideal $\left(p\right)$. This completes
the proof of the lemma.
\end{proof}
It remains to show that for every $I$ in \eqref{eq:lim fix =00003D =000023 rank-1 ideals},
$f_{I}$ is supported on $\left\langle u\right\rangle $. 
\begin{lem}
The support of $f_{I}$ contains $v_{0}=1$.
\end{lem}

\begin{proof}
Recall that $t$ denotes the coincidence step in the exposure of $I$
along $\left[1,w\right]$. If $t=\left|w\right|$, then $f_{I}=w-1$
and the lemma holds. Assume that $t<\left|w\right|$ and that the
support of $f_{I}$ does not contain $1$. For every $s\ge t$, the
vertex $v_{s}$ is (exposed in) a non-free step, and let $f_{s}\in I$
be the ShortLex-minimal element among all monic elements in $I$ with
$v_{s}$ their leading monomial and which are supported on $\left\{ v_{0},\ldots,v_{s}\right\} $.
(Note that this definition is unambiguous: if $f\ne g$ are two different
monic elements with the exact same support, then there is some linear
combination $\lambda f+\left(1-\lambda\right)g$ which is strictly
smaller). In particular, $f_{t}=f_{I}$ and $f_{\left|w\right|}=w-1$.
Now fix $s\in\left\{ t+1,\ldots,\left|w\right|\right\} $ to be the
smallest index for which $v_{0}=1$ is in the support of $f_{s}$.
There is no element in $I|_{\left\{ v_{0},\ldots,v_{s-1}\right\} }$
with $v_{0}$ in its support, because $I|_{\left\{ v_{0},\ldots,v_{s-1}\right\} }=\spn_{\k}\left\{ f_{t},\ldots,f_{s-1}\right\} $.
Let $b\in B\cup B^{-1}$ denote the label of the edge from $v_{s-1}$
to $v_{s}$. As step $s$ is forced, there is some monic $g\in I|_{\left\{ v_{1},\ldots,v_{s-1}\right\} }$
with leading monomial $v_{s-1}$ such that $g.b$ is supported on
$\left\{ v_{0},\ldots,v_{s}\right\} $. This $g.b$ must have $v_{0}$
in its support, for if not, $f_{s}-g.b$ does, and the latter is supported
on $\left\{ v_{0},\ldots,v_{s-1}\right\} $. We conclude that the
first edge of $w$ must be $b^{-1}$. As $w$ is cyclically reduced,
$s<\left|w\right|$.

Now consider the vertex $v_{s+1}$, and assume the edge from $v_{s}$
to $v_{s+1}$ is $c\in B\cup B^{-1}$, $c\ne b^{-1}$:
\[
\xymatrix{v_{s-1}\ar@{->}[r]^{b} & v_{s}\ar@{->}[r]^{c} & v_{s+1}}
.
\]
As $I|_{\left\{ v_{0},\ldots,v_{s}\right\} }=\spn_{\k}\left\{ f_{t},\ldots,f_{s}\right\} $,
every element $g\in I|_{\left\{ v_{0},\ldots,v_{s}\right\} }$ with
leading monomial $v_{s}$ must have $v_{0}$ in its support. But then,
$g$ cannot possibly be supported on starting points of $c$-edges,
contradicting the fact that step $s+1$ is also forced. Thus $f_{I}$
contains $v_{0}$ in its support. 
\end{proof}
If $t=0$ then $f_{I}=1$ and $I=\left(1\right)$. If $t=\left|w\right|$,
then $f_{I}=w-1$ and $I=\left(w-1\right)$. So assume from now on
that $0<t<\left|w\right|$. Also, denote by $b\in B\cup B^{-1}$ the
first letter of $w$.
\begin{lem}
\label{lem:first letter after coincidence}The letter from $v_{t}$
to $v_{t+1}$ is $b$, and $f_{I}$ must be supported on $D_{b}^{t+1}$. 
\end{lem}

\begin{proof}
Assume the edge from $v_{t}$ to $v_{t+1}$ is $c$. The step exposing
$v_{t+1}$ is forced, but $f_{I}$ is the only monic element in $I|_{\left\{ v_{0},\ldots,v_{t}\right\} }$.
Thus the corresponding element in $D_{c}^{t+1}$ must be $f_{I}$.
As $f_{I}$ has $v_{0}$ in its support, we must have $c=b$.
\end{proof}

\begin{proof}[Completing the proof of Theorem \ref{thm:limit powers fix}]
 Recall that we now assume that $I$ is a rank-1 ideal containing
$w-1$ with $I\ne\left(1\right),\left(w-1\right)$, and that we need
to show that $I$ is supported on $\left\langle u\right\rangle $.
As $I$ contains $w-1$ if and only if it contains $wb-b$, the argument
above (and Corollary \ref{cor:finitely many critical extensions})
show that $I\le_{\left[b,wb\right]}\A$. Expose $I$ along $\left[b,wb\right]$
according to the restriction of ShortLex, and denote the vertices
by $v_{1},\ldots,v_{\left|w\right|+1}$ (so keeping the same labels
as before). As $f_{I}$ is the only monic element in $I|_{\left\{ v_{0},\ldots,v_{t}\right\} }$,
we have that $I|_{\left\{ v_{1},\ldots,v_{t}\right\} }=0$, that $v_{t+1}$
is the first (and only) coincidence in the exposure along $\left[b,wb\right]$,
and that the coincidence is given by $f_{I}.b$. Clearly, $f_{I}.b$
has $v_{1}$ in its support. If $b'$ is the second letter of $w$,
then the same argument as in Lemma \ref{lem:first letter after coincidence}
shows that $f_{I}.b$ is supported on vertices with an outgoing $b'$-edge
in $\left[b,wb\right]$, and that the edge from $v_{t+1}$ to $v_{t+2}$
is $b'$.

By iterating the same argument we get that for every prefix $w'$
of $w$, $f_{I}.w'$ is supported on $\left[1,w^{2}\right]$. Moreover,
the direction in which one can read a prefix of $w$ from some $v_{j}$
in the support of $f_{I}$ along $\left[1,w^{2}\right]$ is necessarily
forward: if it goes backward, then after $\left\lfloor \frac{j}{2}\right\rfloor $
step this path would collide with the path reading $w$ coming from
$v_{0}$ (a letter in $\left[1,w\right]$ cannot be equal to its own
inverse or to the inverse of the following letter). We obtain that
if $f_{I}$ has some $v_{j}=z\in\F$ in its support, then $zw=wz$,
and so $z$ belongs to the centralizer of $w$ in $\F$, which is
$\left\langle u\right\rangle $. This completes the proof.
\end{proof}
\begin{cor}
\label{cor:pa=00003D1 iff power}Let $1\ne w\in\F$. Then $\pi_{q}\left(w\right)=1$
if and only if $w$ is a proper power.
\end{cor}

\begin{proof}
Write $w=u^{d}$ with $u$ a non-power and $d\ge1$. The discussion
above shows that the rank-1 critical ideals of $w-1$ are in one to
one correspondence with the monic divisors of $x^{d}-1\in\k\left[x\right]$,
except for $1$ and $x^{d}-1$. If $d=1$, there are no such divisors,
and so $\pi_{q}\left(w\right)\ge2$. If $d\ge2$, there is at least
one such divisor: the polynomial $x-1$, and so $\pi_{q}\left(w\right)=1$.
\end{proof}
Recall that if $\lambda\in\gl_{1}\left(\k\right)\cong\k^{*}$, then
$\tilde{\lambda}\colon\gln\to\mathbb{Z}_{\ge0}$ counts, for every
element $g\in\gln$, the number of vectors $v\in V=\k^{N}$ satisfying
$v.g$ = $\lambda v$. The same argument given above for $\fix=\tilde{1}$
applies to all $\lambda\in\k^{*}$ and gives the following result.
\begin{cor}
\label{cor:vw=00003Dlambda v}Let $\lambda\in\gl_{1}\left(\k\right)\cong\k^{*}$,
let $1\ne w\in\F$ and write $w=u^{d}$ with $u$ a non-power and
$d\ge1$. Then, 
\[
\lim_{N\to\infty}\mathbb{E}_{w}\left[\tilde{\lambda}\right]=\left|\left\{ p\in\k\left[x\right]\,\middle|\,p|x^{d}-\lambda~\mathrm{and}~p~\mathrm{monic}\right\} \right|.
\]
\end{cor}

\subsection{\label{subsec:powers general}$\lim_{N\to\infty}\mathbb{E}_{w}\left[\protect\btil\right]$
and the proof of Theorem \ref{thm:limit powers general}}

Our next goal is proving Theorem \ref{thm:limit powers general},
which states that for any fixed $\B\in\glm$, the limit $\lim_{N\to\infty}\mathbb{E}_{w}\left[\btil\right]$
exists and depends only on $d$, where $w=u^{d}\ne1$ as before. As
in the proof of Theorem \ref{thm:limit powers fix}, we may assume
that $w$ is cyclically reduced. From \eqref{eq:general rational expression with submodules}
and Corollary \ref{cor:contrib is m-rank} it follows that 
\begin{equation}
\mathbb{E}_{w}\left[\btil\right]=\sum_{M\le_{\mathbb{W}}\A^{m}\colon M\supseteq\eqb}\left(q^{N}\right)^{m-\rk M}\left(1+O\left(\frac{1}{q^{N}}\right)\right),\label{eq:E_w=00005BB=00005D as sum with ranks}
\end{equation}
where $\mathbb{W}$ is the union of the paths $\left[e_{i},e_{i}w\right]\in\C_{i}$
for $i=1,\ldots,m$. Throughout this Subsection \ref{subsec:powers general}
we continue using ShortLex from Definition \ref{def:ShortLex} and
its restriction to collections of subtrees such as $\mathbb{W}$.
\begin{lem}
\label{lem:The-smallest-rank is m}The smallest rank of a submodule
$M\le_{\mathbb{W}}\A^{m}$ containing $\eqb$ is $m$. In particular,
$\lim_{N\to\infty}\mathbb{\mathbb{E}}_{w}\left[\btil\right]$ exists
and 
\begin{equation}
\lim_{N\to\infty}\mathbb{E}_{w}\left[\btil\right]=\left|\left\{ M\le_{\mathbb{W}}\A^{m}\,\middle|\,M\supseteq\eqb~\mathrm{and}~\rk M=m\right\} \right|.\label{eq:limit of B is number of rank-m submodule}
\end{equation}
\end{lem}

\begin{proof}
Let $M\le_{\mathbb{W}}\A^{m}$ contain $\eqb$. We expose $M$ along
$\mathbb{W}$ in the order induced on $\mathbb{W}$ from ShortLex.
So we first expose $e_{1},\ldots,e_{m}$, then, if the first letter
of $w$ is $b\in B\cup B^{-1}$, we expose $e_{1}b,\ldots,e_{m}b$,
and so on. By definition, the last $m$ steps, where $e_{1}w,\ldots,e_{m}w$
are exposed, are not free: $\eqb$ contains an element supported on
$\left\{ e_{i}w,e_{1},\ldots,e_{m}\right\} $ with $e_{i}w$ its leading
monomial. We claim that for every $e\in E$, the first non-free step
in $\left[e,ew\right]$ is a coincidence. In particular, there are
at least $m$ coincidences and so by Theorem \ref{thm:basis from coincidences},
$\rk M\ge m$.

Indeed, assume that the first non-free vertex in $\left[e_{i},e_{i}w\right]$
is $e_{i}z$ for some prefix $z$ of $w$. If $z=1$, then $e_{i}z$
is a coincidence by definition. Now assume that $z\ne1$ and that
$b\in B\cup B^{-1}$ is the last letter of $z$. As $e_{i}z$ is the
first non-free step in $\left[e_{i},e_{i}w\right]$, we have that
$e_{i}zb^{-1}$ was free, so there is no element of $M|_{\mathbb{W}}$
with leading monomial $e_{i}zb^{-1}$. Between the exposure of $e_{i}zb^{-1}$
and that of $e_{i}z$, the vertices exposed do not admit outgoing
$b$-edges (in the already-exposed part of $\mathbb{W}$): these vertices
are either $e_{j}zb^{-1}$ for $j>i$, where the only outgoing edge
is headed backwards and cannot be $b$ as $w$ is reduced; or $e_{j}z$
for $j<i$, where the only outgoing edge is $b^{-1}$. Thus, when
exposing $e_{i}z$ at step $t$, the largest monomial in $D_{b}^{t}$
is $e_{i}zb^{-1}$, but as $e_{i}zb^{-1}$ is free, there are no elements
of $M|_{D_{b}^{t}}\subseteq M|_{\mathbb{W}}$ with leading monomial
$e_{i}zb^{-1}$. So step $t$ cannot be forced and must be a coincidence.
\end{proof}
The proof of Lemma \ref{lem:The-smallest-rank is m} actually shows
that a free vertex in $\left[e,ew\right]$ cannot be followed by a
forced vertex in the same path. As the last vertex in $\left[e,ew\right]$
is non-free, we get the following.
\begin{cor}
\label{cor: free steps then coincidence then forced steps}If $M\le_{\mathbb{W}}\A^{m}$
has rank $m$ and contains $\eqb$, then for every $e\in E$, the
first non-free step in $\left[e,ew\right]$ is a coincidence, and
all later steps in $\left[e,ew\right]$ are forced.
\end{cor}

\begin{rem}
It is possible to extend Corollary \ref{cor:finitely many critical extensions}
from elements and ideals in $\A$ to subsets and submodules in $\A^{m}$,
and conclude that every rank-$m$ submodule of $\A^{m}$ containing
$\eqb$ is supported on $\mathbb{W}$.
\end{rem}

\begin{lem}
\label{lem:enough to show submodules are generated on <u>}Assume
that $1\ne w=u^{d}$ with $d\ge1$ and $u$ a non-power. To prove
Theorem \ref{thm:limit powers general}, it is enough to show that
every submodule $M\le_{\mathbb{W}}\A^{m}$ of rank $m$ with $M\supseteq\eqb$
is generated on $\left\{ eu^{j}\right\} _{e\in E,j\in\left\{ 0,\ldots,d\right\} }$.
\end{lem}

\begin{proof}
Assume that every submodule $M$ from \eqref{eq:limit of B is number of rank-m submodule}
is generated on $\left\{ eu^{j}\right\} _{e\in E,j\in\left\{ 0,\ldots,d\right\} }$.
Then, as in the proof of Lemma \ref{lem:divisors of x^d-1}, these
submodules are in one-to-one correspondence with rank-$m$ submodules
of $\k\left[\left\langle u\right\rangle \right]^{m}$ containing $\eqb$
(and generated on $\left\{ eu^{j}\right\} _{e\in E,j\in\left\{ 0,\ldots,d\right\} }$),
where $\k\left[\left\langle u\right\rangle \right]^{m}$ is the rank-$m$
free module over $\k\left[\left\langle u\right\rangle \right]$. As
before, $\k\left[\left\langle u\right\rangle \right]\cong\k\left[\mathbb{Z}\right]\cong\k\left[x,x^{-1}\right]$,
and the image of $\eqb\subseteq\k\left[\left\langle u\right\rangle \right]^{m}$
in $\k\left[\mathbb{Z}\right]^{m}$ through the corresponding isomorphism
does not depend on $u$ but only on $d$. Hence, the number of submodules
in \eqref{eq:limit of B is number of rank-m submodule} does not depend
on $u$, proving Theorem \ref{thm:limit powers general}.
\end{proof}
\begin{rem}
It is quite straightforward to show that every submodule of $\k\left[\left\langle u\right\rangle \right]^{m}$
containing $\eqb$ must be of rank exactly $m$: after the first coincidence
in each of the $m$ paths, all remaining steps are clearly forced.
\end{rem}

Now fix $M\le_{\mathbb{W}}\A^{m}$ of rank $m$ containing $\eqb$.
For every $f\in M|_{\mathbb{W}}$, denote by $\theta\left(f\right)$
the projection of $f$ to the monomials $e_{1},\ldots,e_{m}$, so
$\theta\left(f\right)$ is a $\k$-linear combination of $e_{1},\ldots,e_{m}$.
For $t=0,\ldots,\left|\mathbb{W}\right|$, let
\[
\Theta_{t}\defi\spn_{\k}\left\{ \theta\left(f\right)\,\middle|\,f\in M|_{D^{t}\left(\mathbb{W}\right)}\right\} \le\spn_{\k}\left\{ e_{1},\ldots,e_{m}\right\} 
\]
(recall that $D^{t}\left(\mathbb{W}\right)$ is the set of first $t$
monomials exposed in $\mathbb{W}$ through ShortLex). So we have
\[
\left\{ 0\right\} =\Theta_{0}\le\Theta_{1}\le\ldots\le\Theta_{\left|\mathbb{W}\right|}=\spn_{\k}\left\{ e_{1},\ldots,e_{m}\right\} ,
\]
where the last equality is due to the fact that $M\supseteq\eqb$,
the equations in $\eqb$ are supported on $\mathbb{W}$, the linear
combinations of $e_{1},\ldots,e_{m}$ given by the $m$ equations
in $\eqb$ are precisely the rows of $\B$, and $\B$ is regular by
definition. Recall (Corollary \ref{cor: free steps then coincidence then forced steps})
that there is a sole coincidence in $\left[e_{i},e_{i}w\right]$ for
every $i=1,\ldots,m$, and let $z_{i}$ denote the prefix of $w$
so that $e_{i}z_{i}$ is the step in which the coincidence of $\left[e_{i},e_{i}w\right]$
takes place. 
\begin{lem}
\label{lem:theta at coincidences are basis}We have $\Theta_{t-1}\lvertneqq\Theta_{t}$
if and only if step $t$ is a coincidence. In particular, if $g_{i}\in M|_{\mathbb{W}}$
is a (monic) element with leading monomial $e_{i}z_{i}$, then the
vectors $\theta\left(g_{1}\right),\ldots,\theta\left(g_{m}\right)$
are linearly independent.
\end{lem}

\begin{proof}
We already explained why $\dim\left(\Theta_{\left|\mathbb{W}\right|}\right)=m$.
Note that $\dim\Theta_{t}-\dim\Theta_{t-1}\in\left\{ 0,1\right\} $,
because every two monic elements $g_{1},g_{2}\in M|_{\mathbb{W}}$
with leading monomial $v_{t}$ satisfy $\theta\left(g_{1}\right)-\theta\left(g_{2}\right)=\theta\left(g_{1}-g_{2}\right)\in\Theta_{t-1}$.
As there are exactly $m$ coincidences, it is enough to prove that
$\Theta_{t-1}=\Theta_{t}$ whenever step $t$ is forced or free. If
step $t$ is free, then $M|_{D^{t-1}\left(\mathbb{W}\right)}=M|_{D^{t}\left(\mathbb{W}\right)}$
and obviously $\Theta_{t-1}=\Theta_{t}$. It thus remains to show
that this is the case also if step $t$ is forced. 

Let $ez$ be the monomial exposed in step $t$ which is forced, and
let $b\in B\cup B^{-1}$ be the edge leading to $ez$. There exists
some $g\in M|_{D_{b}^{t}\left(\mathbb{W}\right)}$ with $ezb^{-1}$
in its support (in fact, its leading monomial), such that the coefficient
of $ezb^{-1}$ in $g$ is $1\in\k$. If $\theta\left(g.b\right)\in\Theta_{t-1}$,
then every other monic $f\in M|_{\mathbb{W}}$ with leading monomial
$ez$ satisfies $\theta\left(f\right)=\theta\left(f-g.b\right)+\theta\left(g.b\right)\in\Theta_{t-1}$
and we are done. So assume that $\theta\left(g.b\right)\notin\Theta_{t-1}$.
In particular, $\theta\left(g.b\right)\ne0$, so $g.b$ has some $e'\in E$
in its support, and so $b^{-1}$ is the first letter of $w$. As $w$
is assumed to be cyclically reduced, $z$ is a proper prefix of $w$. 

Now consider the monomial following $ez$ in $\left[e,ew\right]$.
Say it is $ezc$ for some $b^{-1}\ne c\in B\cup B^{-1}$, and it is
exposed at time $s$ (so $s=t+m$). Because step $t$ is forced, so
is step $s$ (by Corollary \ref{cor: free steps then coincidence then forced steps}).
As in the proof of Lemma \ref{lem:The-smallest-rank is m}, the monomials
exposed between $ez$ and $ezc$ do not belong to $D_{c}^{s}\left(\mathbb{W}\right)$,
so $D_{c}^{s}\left(\mathbb{W}\right)\subseteq D^{t}\left(\mathbb{W}\right)$.
As step $s$ is forced, there exists some monic $f\in M|_{D_{c}^{s}\left(\mathbb{W}\right)}\subseteq M|_{D^{t}\left(\mathbb{W}\right)}$
with $ez$ its leading monomial. As before, as $\theta\left(f-g.b\right)\in\Theta_{t-1}$
but $\theta\left(g.b\right)\notin\Theta_{t-1}$, we get $\theta\left(f\right)=\theta\left(f-g.b\right)+\theta\left(g.b\right)\notin\Theta_{t-1}$.
In particular, $\theta\left(f\right)\ne0$. But $c\ne b^{-1}$ is
not the first letter of $w$, so $f$ cannot have any $e\in E$ in
its support -- a contradiction. This completes the proof of the first
statement of the lemma. This also shows there exist $g_{i}\in M|_{\mathbb{W}}$
with leading monomial $e_{i}z_{i}$, for $i=1,\ldots,m$, such that
$\theta\left(g_{1}\right),\ldots,\theta\left(g_{m}\right)$ are linearly
independent. The second statement of the lemma now follows from the
fact that if $f,g\in M|_{\mathbb{W}}$ are both monic with leading
monomial the $t$-th vertex, then $\theta\left(f\right)-\theta\left(g\right)\in\Theta_{t-1}$.
\end{proof}
Define $\mathbb{W}^{2}\defi\left[e_{1},e_{1}w^{2}\right]\cup\ldots\cup\left[e_{m},e_{m}w^{2}\right]$,
and let $b\in B\cup B^{-1}$ be the first letter of $w$. For every
$i=1,\ldots,m$, let $f_{e_{i}z_{i}}\in M|_{\mathbb{W}}$ be the minimal
monic element with leading monomial $e_{i}z_{i}$.
\begin{lem}
\label{lem:f_ez supported on D_b}For every $i=1,\ldots,m$, $f_{e_{i}z_{i}}$
is supported on $D_{b}\left(\mathbb{W}^{2}\right)$, and the outgoing
$b$-edge at $e_{i}z_{i}$ is headed forward (towards $e_{i}w^{2}$).
\end{lem}

\begin{proof}
We proceed by induction on the order induced by ShortLex on $\left\{ e_{i}z_{i}\right\} _{i=1,\ldots,m}$.
The argument that follows works for both the base case and the induction
step. If $z_{i}=w$, then there is an element in $\eqb$ with leading
monomial $e_{i}w$ which is supported on $E\cup\left\{ e_{i}w\right\} $,
so $f_{e_{i}z_{i}}$ is also supported on $E\cup\left\{ e_{i}w\right\} $
and the claim is clear. So assume that $\left|z_{i}\right|<\left|w\right|$,
and that $e_{i}z_{i}$ is exposed at time $t$ and admits an outgoing
$c$-edge towards $e_{i}w$. Then step $t+m$, in which $e_{i}z_{i}c$
is exposed, is forced, and there exists some $g\in M|_{D_{c}^{t+m}\left(\mathbb{W}\right)}\subseteq M|_{D^{t}\left(\mathbb{W}\right)}$
with leading monomial $e_{i}z_{i}$. By Lemma \ref{lem:theta at coincidences are basis},
$\theta\left(g\right)\notin\Theta_{t-1}$ so $g$ has some $e\in E$
in its support, and therefore $c=b$.

Moreover, we may assume that $g$ is supported on free steps and coincidences
only. Indeed, the submodule $M_{t+m-1}$ is generated on the free
steps and coincidences exposed up to step $t+m-1$ (this is always
the case in every valid exposure process), but by Corollary \ref{cor: free steps then coincidence then forced steps},
in our case these vertices form a valid collection of subtrees $\T$
($\T=T_{1}\cup\ldots\cup T_{m}$, where $T_{j}=\left[e_{j},e_{j}z_{j}\right]\cap\left[e_{j},e_{j}z_{i}\right]$
for $j\ge i$ and $T_{j}=\left[e_{j},e_{j}z_{j}\right]\cap\left[e_{j},e_{j}z_{i}b\right]$
for $j<i$). But $e_{i}z_{i}b$ is forced, so every element with leading
monomial $e_{i}z_{i}b$ belongs to $M_{t+m-1}$, and if we extend
$\T$ to $e_{i}z_{i}b$ it is still a forced step (by Lemma \ref{lem:forced-free-coincidence}).
Thus there is some $g\in M|_{D_{b}\left(\mathbb{T}\cup\left\{ e_{i}z_{i}b\right\} \right)}$
with leading monomial $e_{i}z_{i}$.

If $g$ has some coincidence $e_{j}z_{j}$ in its support other than
$e_{i}z_{i}$, then as $e_{j}z_{j}<e_{i}z_{i}$, our induction hypothesis
applies and $f_{e_{j}z_{j}}\in D_{b}\left(\mathbb{W}^{2}\right)$.
Hence we may subtract $\alpha f_{e_{j}z_{j}}$ from $g$ for some
$\alpha\in\k^{*}$ to decrease $g$, and $g-\alpha f_{e_{j}z_{j}}\in D_{b}\left(\mathbb{W}^{2}\right)$.
If we repeat such subtractions as long as we can, we end up with a
monic element $f$ which is supported entirely on free vertices inside
$D_{b}\left(\mathbb{W}^{2}\right)$ along with its leading monomial
$e_{i}z_{i}$. Because all its non-leading monomials are free, this
$f$ is exactly $f_{e_{i}z_{i}}$ (otherwise $f-f_{e_{i}z_{i}}\ne0$
is supported on free vertices, which is impossible), and we are done. 
\end{proof}

\begin{proof}[Completing the proof of Theorem \ref{thm:limit powers general}]
 Recall that $M\le_{\mathbb{W}}\A^{m}$ is a fixed submodule satisfying
$\rk M=m$ and $M\supseteq\eqb$. By Lemma \ref{lem:enough to show submodules are generated on <u>},
it is enough to show that $M$ is generated by elements supported
on $\left\{ eu^{j}\right\} _{e\in E,j\in\left\{ 0,\ldots,d\right\} }$.
By Theorem \ref{thm:basis from coincidences}, $M=\left(f_{e_{1}z_{1}},\ldots,f_{e_{m}z_{m}}\right)$,
so it is enough to show that $f_{e_{i}z_{i}}$ is supported on $\left\{ eu^{j}\right\} _{e\in E,j\in\mathbb{Z}}$
for all $i$. 

Recall that $b$ is the first letter of $w$. The submodule $M$ contains
$\eqb$ if and only if it contains $\eqb.b\defi\left\{ f.b\,\middle|\,f\in\eqb\right\} $.
Define 
\[
\mathbb{W}^{b}\defi\bigcup_{e\in E}\left[b,wb\right],
\]
and consider the exposure of $M$ along $\mathbb{W}^{b}$ in the order
induced from ShortLex. Clearly, the monomials that were free in the
exposure along $\mathbb{W}$ are free now as well. We claim that the
former coincidences $e_{i}z_{i}$ are now also free: as above, if
$f\in M|_{\mathbb{W}^{b}}$ is monic with leading monomial $e_{i}z_{i}$,
then $f_{e_{i}z_{i}}-f\in M$ is an element with $\theta\left(f_{e_{i}z_{i}}-f\right)=\theta\left(f_{e_{i}z_{i}}\right)$
but with leading monomial smaller than $e_{i}z_{i}$, which contradicts
Lemma \ref{lem:theta at coincidences are basis}. On the other hand,
by Lemma \ref{lem:f_ez supported on D_b}, $f_{e_{i}z_{i}}.b\in M|_{\mathbb{W}^{b}}$
has leading monomial $e_{i}z_{i}b$, and so $e_{i}z_{i}b$ is a coincidence
in the exposure of $M$ along $\mathbb{W}^{b}$. Moreover, the non-leading
monomials of $f_{e_{i}z_{i}}.b$ are all free in the exposure along
$\mathbb{W}^{b}$, so $f_{e_{i}z_{i}}.b$ is the minimal monic element
in $M|_{\mathbb{W}^{b}}$ with leading monomial $e_{i}z_{i}b$. The
same argument as in Lemma \ref{lem:f_ez supported on D_b} shows that
$f_{e_{i}z_{i}}.b$ is supported on $D_{c}\left(\mathbb{W}^{2}\right)$
and the outgoing $c$-edge at $e_{i}z_{i}b$ is headed towards $e_{i}w^{2}$,
where $c\in B\cup B^{-1}$ is the second letter of $w$. 

This argument can now go on to the exposure of $M$ along $\mathbb{W}^{bc}$
and so on, and shows that for every prefix $w'$ of $w$ and every
$i$, $f_{e_{i}z_{i}}.w'$ is supported on $\left[1,w^{2}\right]$.
This completes the proof exactly as in the proof of Theorem \ref{thm:limit powers fix}
in Section \ref{subsec:powers fix}.
\end{proof}

\section{The quotient module $\protect\k\left[\protect\F\right]/\left(w-1\right)$\label{sec:A_w}}

Fix $w\in\F$, and consider the right $\A$-module obtained as a quotient
of the $\A$-module $\A$ by its submodule $\left(w-1\right)$. We
denote this quotient by\marginpar{$\protect\A_{w}$}
\[
\A_{w}\defi\k\left[\F\right]/\left(w-1\right)=\A/\left(w-1\right).
\]
In this section we study this module and prove two main results about
it. First, we show that if $w$ is a non-power, then the only cyclic
generators of $\A_{w}$ are the ``obvious ones'' (Theorem \ref{thm:cyclic generators of A_w}).
Second, we prove that whenever a subtree $T\subseteq\mathrm{Cay}\left(\F,B\right)$
supports both $w-1$ and a rank-2 ideal $I\le_{T}\A$ in which $w-1$
is primitive, there is an element $f\in\A$ supported on $T$ so that
$\left\{ f,w-1\right\} $ is a basis of $I$ (Corollary \ref{cor:complement to a basis of w-1}).
In particular, the latter result yields an algorithm to test whether
$w-1$ is primitive in a given rank-2 ideal (Corollary \ref{cor:algo to test primitivity of w-1 in a rank-2 ideal}).
We need these two results for our proof of Theorem \ref{thm:fixed vectors in pi=00003D2}
in Section \ref{sec:Critical-ideals-of-rank-2}, but we also find
them interesting for their own right. See Section \ref{sec:Open-Questions}
for a discussion on potential generalizations of these results.\\

Consider the Schreier graph\marginpar{$\protect\sw$} 
\[
\sw\defi\mathrm{Sch}\left(\F\curvearrowright\left\langle w\right\rangle \backslash\F,B\right)=\left\langle w\right\rangle \backslash\mathrm{Cay}\left(\F,B\right).
\]
This is a graph whose vertices correspond to the right cosets of the
subgroup $\left\langle w\right\rangle $ in $\F$. For every vertex
$\left\langle w\right\rangle z$ and every $b\in B$, there is a directed
$b$-edge from the vertex $\left\langle w\right\rangle z$ to the
vertex $\left\langle w\right\rangle zb$. In other words, this is
the quotient of $\mathrm{Cay}\left(\F,B\right)$ by the action of
$\left\langle w\right\rangle $ from the left. Note that $\sw$ is
made of a cycle (reading the cyclic reduction of $w$) with infinite
trees hanging from it (unless $\rk\F=1$, in which case $\sw$ is
a mere cycle). This is illustrated in Figure \ref{fig:Schreier graph example}. 

\begin{figure}
\begin{centering}
\includegraphics[viewport=0bp 0bp 200bp 180bp,scale=1.2]{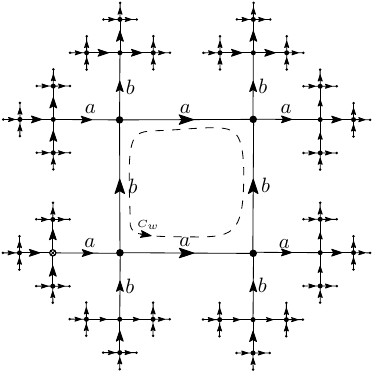}
\par\end{centering}
\caption{\label{fig:Schreier graph example}The Schreier graph $\protect\sw$
for $w=a\left[a,b\right]a^{-1}$. The unique simple cycle is marked
by $c_{w}$.}
\end{figure}

An element $f\in\A$ belongs to the ideal $\left(w-1\right)$ if and
only if for every $z\in\F$, the coefficients in $f$ of the elements
in the right coset $\left\langle w\right\rangle z$ sum up to zero.
Therefore, the elements of $\A_{w}$ are given by $\k$-linear combinations
of right cosets of $\left\langle w\right\rangle $, namely, $\k$-linear
combinations of the vertices of $\sw$. This can also be seen by the
fact that a possible Schreier transversal of the ideal $\left(w-1\right)$
is obtained by considering $\mathrm{Cay}\left(\F,B\right)$, cutting
the axis\footnote{The axis of $w$ is composed of the points in $\mathrm{Cay}\left(\F,B\right)$
moved by left multiplication by $w$ by the least distance.} of $w$ on both sides of one period of the cyclic reduction of $w$,
and taking the connected component of this period.

Now consider the quotient map
\[
\rho\colon\A\twoheadrightarrow\A_{w},
\]
which, by abuse of notation, we also regard as the graph morphism
\[
\rho\colon\mathrm{Cay}\left(\F,B\right)\to\sw.
\]
Note that whenever a subtree $T\subseteq\mathrm{Cay}\left(\F,B\right)$
contains $\left[1,w\right]$, its image $\rho\left(T\right)\subseteq\sw$
contains the cycle in $\sw$. In fact, it suffices that $T$ contains
any interval in the axis of $w$ of length at least the length of
the cyclic reduction of $w$.
\begin{lem}
\label{lem:cyclic submodules in A_w}Let $G\subseteq\sw$ be a connected
subgraph which contains the cycle of $\sw$. Let $f\in\A_{w}$ satisfy
that none of $\left\{ f.z\,\middle|\,z\in\F\right\} $ is supported
on $G.$ Then the submodule $f\A\le\A_{w}$ does not contain any non-zero
element supported on $G$. 
\end{lem}

\begin{proof}
On the vertices of $\sw\setminus G$ define an ``exploration'' as
in Definition \ref{def:exploration}: this is an enumeration of these
vertices such that every vertex is a neighbour of some vertex in $G$
or of a smaller vertex. This exploration induces a pre-order on the
orbit $\left\{ f.z\,\middle|\,z\in\F\right\} $ obtained by comparing
the largest vertex in their support with respect to this exploration
order (by assumption, every element $f.z$ in this orbit has at least
one vertex outside $G$ in its support). Assume without loss of generality
that $f$ is an element of the orbit with the smallest possible maximal
vertex in its support. Denote this vertex $v_{\max}$. Denote by $\overline{G}$
the (connected) subgraph of $\sw$ consisting of $G$ together with
the prefix $\left\{ v\in\mathrm{vert}\left(\sw\backslash G\right)\,\middle|\,v\le v_{\max}\right\} $
of the exploration order on $\sw\setminus G$.

Now consider the element $fg\in\A_{w}$ for an arbitrary $g\in\A$
not supported on the identity $e\in\F$. It suffices to show that
$fg$ is not supported on $\overline{G}$ (let alone on $G$). Write
$f=\alpha_{1}f_{1}+\ldots+\alpha_{m}f_{m}$ with $\alpha_{1},\ldots,\alpha_{m}\in\k^{*}$
and distinct $f_{1},\ldots,f_{m}\in\mathrm{vert}\left(\sw\right)$,
and write $g=\beta_{1}g_{1}+\ldots+\beta_{\ell}g_{\ell}$ with $\beta_{1},\ldots,\beta_{\ell}\in\k^{*}$
and distinct $g_{1},\ldots,g_{\ell}\in\F$ and so that $\left|g_{1}\right|\ge\left|g_{2}\right|\ge\ldots\ge\left|g_{\ell}\right|$.
Denote by $b\in B\cup B^{-1}$ the first letter in $g_{1}$. Then
$f.b$ cannot be supported on $\overline{G}$: otherwise, $f.b$ would
be supported on $G$ together with vertices strictly smaller than
$v_{\max}$ in $\sw\setminus G$ (we use here the fact that $v_{\max}$
is a leaf in $\overline{G}$), contradicting our assumption about
$f$. So there is a monomial $f_{i}$ in the support of $f$ such
that $f_{i}.b$ is a monomial outside $\overline{G}$. But then $f_{i}g_{1}$
is at distance $\left|g_{1}\right|$ from $\overline{G}$, with the
closest vertex of $\overline{G}$ being $f_{i}$. Clearly, $f_{i}g_{1}\ne f_{j}g_{k}$
for every $\left(j,k\right)\ne\left(i,1\right)$, because the only
path of length $\left|g_{1}\right|$ from $\overline{G}$ to $f_{i}g_{1}$
in $\sw$, is the path starting at $f_{i}$ and reading $g_{1}$.
Thus $f_{i}g_{1}$ belongs to the support of $fg$, and $fg$ is not
supported on $\overline{G}$.
\end{proof}
\begin{cor}
\label{cor:complement to a basis of w-1}Let $1\ne w\in\F$ and $T\subseteq\mathrm{Cay}\left(\F,B\right)$
be a subtree which contains $\left[1,w\right]$. Assume that $I\le_{T}\A$
is a rank-2 ideal supported on $T$ which contains $w-1$ as a primitive
element. Then there is an element $f\in\A$ supported on $T$ so that
$\left\{ f,w-1\right\} $ is a basis for $I$.
\end{cor}

\begin{proof}
As $w-1$ is primitive in $I$, there is some $f\in\A$ which completes
it to a basis of $I$. Consider $\overline{T}=\rho\left(T\right)$,
the image of $T$ in $\sw$ and let $\overline{f}=\rho\left(f\right)\in\A_{w}$.
If $\left\{ f,w-1\right\} $ is a basis for $I$, then so is $\left\{ g,w-1\right\} $
for every $g\in\rho^{-1}\left(\overline{f}\right)$, because in this
case $f-g\in\left(w-1\right)$. So if $\overline{f}.z$ is supported
on $\overline{T}$ for some $z\in\F$, we are done: if $\left\{ f,w-1\right\} $
is a basis then so is $\left\{ f.z,w-1\right\} $. Otherwise, we are
in the situation of Lemma \ref{lem:cyclic submodules in A_w}, and
$\overline{f}\A$ does not contain any element supported on $\overline{T}$.
But $\overline{f}\A$ contains $\rho\left(I\right)$ (in fact $\overline{f}\A=\rho\left(I\right)$),
and as $I$ is generated on $T$, $I$ contains an element $h\in I\setminus\left(w-1\right)$
which is supported on $T$. Then $\overline{f}\A\ni\rho\left(h\right)$,
which is a contradiction as $\rho\left(h\right)\ne0$ and is supported
on $\overline{T}$.
\end{proof}
\begin{cor}
\label{cor:algo to test primitivity of w-1 in a rank-2 ideal}If the
field $\k$ is finite,\footnote{We assume throughout the paper that $\k$ is finite, but some of the
results about free group algebras, such as Corollary \ref{cor:complement to a basis of w-1},
hold for infinite fields just as well. In contrast, Corollary \ref{cor:algo to test primitivity of w-1 in a rank-2 ideal}
relies on $\k$ being finite.} there is an algorithm to test, given a (generating set of a) rank-2
ideal $I\le\A^{m}$ and a word $w\in\F$, whether $w-1$ is primitive
in $I$.
\end{cor}

\begin{proof}
By \cite[Prop. 2.2]{cohn1964free}, every pair of generators of $I$
is a basis. So $\left\{ f,w-1\right\} $ is a basis for $I$ if and
only if $f,w-1\in I$ and $\left(f,w-1\right)$ contains the given
generating set of $I$. By Corollary \ref{cor:complement to a basis of w-1},
$w-1$ is primitive in $I$ if and only if there exists an element
$f$ supported on $T$ such that $\left\{ f,w-1\right\} $ is a basis
of $I$. As $\k$ is finite, there are finitely many elements supported
on $T$. Finally, Rosenmann \cite{rosenmann1993algorithm} describes
an algorithm to test whether a given element belongs to a given ideal
in $\A$ (where the ideal is given by a finite generating set).
\end{proof}
Corollaries \ref{cor:complement to a basis of w-1} and \ref{cor:algo to test primitivity of w-1 in a rank-2 ideal}
naturally raise the question to what extent they can be generalized
for ideals of rank larger than two and for elements of $\A$ which
are not of the form $w-1$ -- see Section \ref{sec:Open-Questions}
for a discussion around it.

\subsection{Cyclic generators of $\protect\k\left[\protect\F\right]/\left(w-1\right)$\label{subsec:Cyclic-generators-of-A_w}}

The group algebra $\A=\k\left[\F\right]$ has only trivial units --
a scalar times an element of the group\footnote{This is well known. It can also be seen, for example, by an argument
similar to the one in the proof of Lemma \ref{lem:cyclic submodules in A_w}:
for any $0\ne f\in\A$ with support of size at least $2$, take a
minimal subtree $T$ of $\C=\mathrm{Cay}\left(\F,B\right)$ which
supports an element in the orbit $\left\{ f.z\,\middle|\,z\in\F\right\} $.
Then the argument in the proof of Lemma \ref{lem:cyclic submodules in A_w}
shows that $f\A$ does not contain elements supported on $T$ except
for scalar multiples of $f$.} (this property was conjectured by Kaplansky to hold in all group
algebras of torsion-free groups over fields but a counterexample has
recently been found \cite{gardam2021counterexample}). The goal of
this subsection is to prove a similar result for $\A_{w}=\A/\left(w-1\right)$.
While $\A_{w}$ is not a ring and therefore does not admit units,
it does admit cyclic generators as an $\A$-module: elements $f\in\A_{w}$
such that $f\A=\A_{w}$. Clearly, for every unit of $\A$, its image
in $\A_{w}$ is a cyclic generator. Here we prove that provided that
$w$ is not a power, all cyclic generators of $\A_{w}$ are of this
sort.
\begin{thm}
\label{thm:cyclic generators of A_w}Assume that $1\ne w\in\F$ is
a non-power. Then every cyclic generator of the right $\A$-module
$\A_{w}=\A/\left(w-1\right)$ is an image of a unit of $\A$.
\end{thm}

Namely, every cyclic generator of $\A_{w}$ is a coset of the form
$\alpha z+\left(w-1\right)$ for some $\alpha\in\k^{*}$ and $z\in\F$.
\begin{rem}
Theorem \ref{thm:cyclic generators of A_w} is false for proper powers.
For example, if $\left|\k\right|=3$ and $w=a^{3}$, then $\rho\left(a+1\right)\in\A_{w}$
is not a $\rho$-image of a unit of $\A$: its support in $\sw$ is
of size two. Yet $a^{3}+1\in\left(a+1\right)$ and so $\rho\left(2\right)=\rho\left(a^{3}+1\right)\in\rho\left(a+1\right)\A$.
Thus $\rho\left(a+1\right)$ is a cyclic generator of $\A_{w}$.
\end{rem}

First we show that cyclic generators in $\A_{w}$ may be assumed to
be supported on the cycle of $\sw$.
\begin{lem}
\label{lem:cyclic generators of A_w are on cycle}If $f\in\A_{w}$
is a cyclic generator of $\A_{w}$, then there is some $z\in\F$ such
that $fz$ is supported on the cycle in $\sw$. 
\end{lem}

\begin{proof}
This follows immediately from Lemma \ref{lem:cyclic submodules in A_w}
applied to $G$ being the cycle in $\sw$.
\end{proof}
Let $w\in\F$ be a non-power. If $w'$ is the cyclic reduction of
$w$ then the automorphism of $\F$ mapping $w$ to $w'$ extends
to an automorphism of $\A$ and induces isomorphisms $\A_{w}\stackrel{\cong}{\to}\A_{w'}$
and $\sw\stackrel{\cong}{\to}{\cal S}_{w'}$. Thus we may assume without
loss of generality that $w$ is cyclically reduced. Denote by $\ll w\gg$
the normal closure of $w$ in $\F$, and denote by $\left(\left(w-1\right)\right)$
the two-sided ideal of $\mathcal{A}$ generated by $w-1$. Since we
have $\left\{ 0\right\} \subseteq\left(w-1\right)\subseteq\left(\left(w-1\right)\right)$,
we get canonical epimorphisms of right $\A$-modules 
\[
\mathcal{A}\stackrel{\rho}{\twoheadrightarrow}\mathcal{A}_{w}\stackrel{\tau}{\twoheadrightarrow}\nicefrac{\A}{\left(\left(w-1\right)\right)}.
\]
See Figure \ref{fig:Schreier graph maps}.

\begin{figure}
\begin{centering}
\includegraphics[viewport=350bp 0bp 400bp 250bp,scale=0.6]{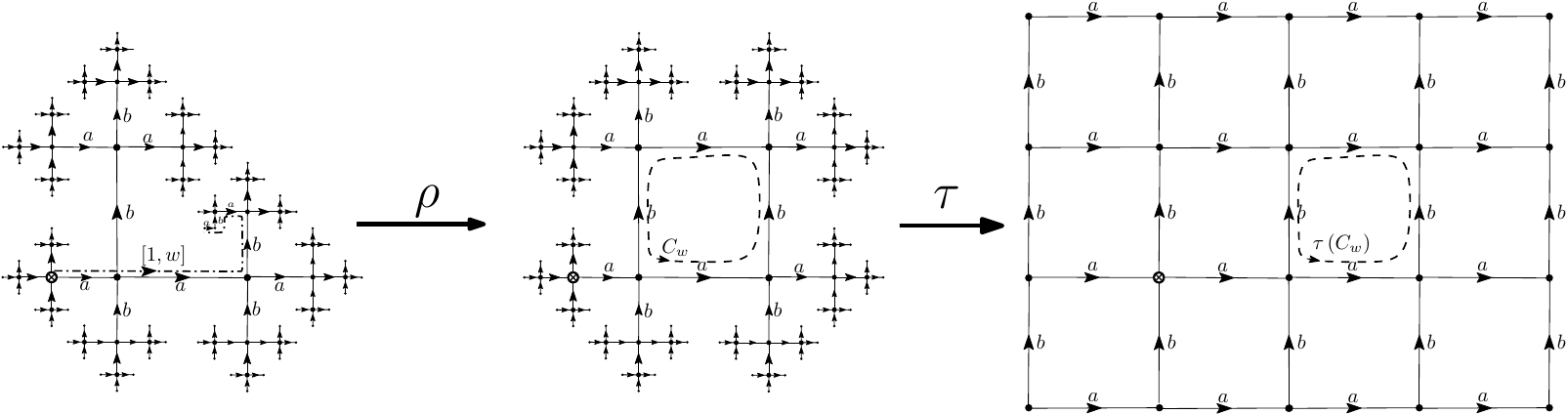}
\par\end{centering}
\caption{\label{fig:Schreier graph maps}Let $w=a\left[a,b\right]a^{-1}.$
The Cayley graph of $\protect\F=\protect\F\left(a,b\right)$ is on
the left with $\left[1,w\right]$ marked. The middle graph is $\protect\sw$,
and the graph at the right side is a piece of the Cayley graph of
$\nicefrac{\protect\F}{\ll w\gg}$. In all graphs, the vertex corresponding
to the identity element or its $\rho$-image is marked with $\otimes$.}
\end{figure}

\begin{lem}
\label{lem:iso of quotient module with qoutient-group algebra}Let
$p:\F\rightarrow\nicefrac{\F}{\ll w\gg}$ be the canonical projection.
The map $\varphi:\nicefrac{\mathcal{\A}}{\left(\left(w-1\right)\right)}\rightarrow K\left[\nicefrac{\F}{\ll w\gg}\right]$
defined by $\sum_{z\in F}\alpha_{z}z+\left(\left(w-1\right)\right)\mapsto\sum_{z\in F}\alpha_{z}p\left(z\right)$
is an isomorphism of $K$-algebras.
\end{lem}

\begin{proof}
The proof is a standard argument in algebra, but we include it for
completeness. By the universal property of group rings, the group
homomorphism $p:\F\rightarrow\nicefrac{\F}{\ll w\gg}\subseteq K\left[\nicefrac{\F}{\ll w\gg}\right]$
extends to a unique $K$-algebra epimorphism $\psi:\mathcal{A}\twoheadrightarrow K\left[\nicefrac{\F}{\ll w\gg}\right]$.
The ideal $\left(\left(w-1\right)\right)$ lies in the kernel of $\psi$:
it is enough to show that $u\left(w-1\right)v\in\ker\psi$ for $u,v\in\F$,
and 
\[
\psi\left(u\left(w-1\right)v\right)=\psi\left(uwv-uv\right)=\psi\left(uwv\right)-\psi\left(uv\right)=p\left(uwv\right)-p\left(uv\right)=0,
\]
where the last equality is because $uwv$ and $uv$ lie in the same
coset of $\ll w\gg$. Thus, the homomorphism $\psi$ induces an epimorphism
$\psi'\colon\nicefrac{\mathcal{\A}}{\left(\left(w-1\right)\right)}\twoheadrightarrow K\left[\nicefrac{\F}{\ll w\gg}\right]$.
For every $z\in\F$, $\psi'$ satisfies $\psi'\left(z+\left(\left(w-1\right)\right)\right)=\psi\left(z\right)=p\left(z\right)$,
and so by linear extension $\psi'$ agrees with $\varphi$ from the
statement of the Lemma (in particular, $\varphi$ is a well-defined
epimorphism of $K$-algebras).

It is left to show that $\varphi$ is injective. Suppose that
\[
\varphi\left(\sum_{z\in F}\alpha_{z}z+\left(\left(w-1\right)\right)\right)=\sum_{z\in F}\alpha_{z}p\left(z\right)=0.
\]
For every coset $C$ of $\ll w\gg$ we have $\sum_{z\in C}\alpha_{z}=0$.
We complete the proof by showing that this implies that $\sum_{z\in C}\alpha_{z}z\in\left(\left(w-1\right)\right)$
-- and then the sum over all cosets would also lie in $\left(\left(w-1\right)\right)$.
Such a finite sum can always be decomposed as a sum over elements
of the form $\alpha\left(z_{2}-z_{1}\right)$ where $\alpha\in K$
and $z_{1},z_{2}\in C$. In every such element, $z_{2}$ can be obtained
from $z_{1}$ by a finite sequence of multiplications from the right
by conjugates of $w$ or $w^{-1}$, and so it is enough to show that
$z_{2}-z_{1}\in\left(\left(w-1\right)\right)$ for $z_{2}=z_{1}\cdot uw^{\varepsilon}u^{-1}$
where $u\in F$ and $\varepsilon\in\left\{ \pm1\right\} .$ And indeed
we have $z_{2}-z_{1}=z_{1}u\left(w^{\varepsilon}-1\right)u^{-1}\in\left(\left(w-1\right)\right)$.
\end{proof}
In our proof of Theorem \ref{thm:cyclic generators of A_w} we use
the following well-known concept.
\begin{defn}
A \textbf{right order} on a group $\Gamma$ is a linear order on $\Gamma$
such that for every $r,s,t\in\Gamma$ with $r<s$ we have $rt<st$.
A group is called \textbf{right-orderable} if it admits a right order.
\end{defn}

It is well-known that Kaplansky's unit conjecture, mentioned above,
is true for right orderable groups -- see, e.g., \cite[Thm.~1.58]{clay2016ordered}.
We add a proof for completeness.
\begin{lem}
\label{lem:right-order implies kaplansky unit}Let $K$ be a field
and $\Gamma$ a right-orderable group. If $ts=1$ for $t,s\in\k\left[\Gamma\right]$
then $t=\lambda g$ for some $\lambda\in\k^{*}$ and $g\in\Gamma$.
\end{lem}

\begin{proof}
Write $t=\sum_{i=1}^{n}\lambda_{i}g_{i}$ for $\lambda_{i}\in\k^{*}$
and $g_{1},...,g_{n}\in\Gamma$ distinct. Since $ts=1$ we know that
$n\neq0$. Now assume towards contradiction that $n\geq2$. Let $<$
be a right order for $\Gamma$. Assume without loss of generality
that $g_{1}<g_{2}<...<g_{n}$ . Write similarly $s=\sum_{j=1}^{m}\mu_{j}h_{j}$
for $h_{1}<h_{2}<...<h_{m}$ and $\mu_{j}\in\k^{*}$. Then we have
$1=rs=\sum_{i,j}\lambda_{i}\mu_{j}g_{i}h_{j}$. \\
We now find two elements of $\Gamma$ such that their coefficients
in $rs$ are nonzero. Let $j_{\min}$ be the index such that $g_{1}h_{j_{\min}}=\min\left\{ g_{1}h_{1},g_{1}h_{2},...,g_{1}h_{m}\right\} $.
In particular, $g_{1}h_{j_{\min}}$ is strictly smaller than any other
$g_{1}h_{j}$ for $j\neq j_{\min}$. In addition, if $i\neq1$ then
$g_{1}h_{j_{\min}}\leq g_{1}h_{j}<g_{i}h_{j}$. Thus, the coefficient
of $g_{1}h_{j_{\min}}$ in $rs$ is $\lambda_{1}\mu_{j_{\min}}\neq0$.
Similarly, let $j_{\max}$ be the index such that $g_{n}h_{j_{\max}}=\max\left\{ g_{n}h_{1},g_{n}h_{2},...,g_{n}h_{m}\right\} $.
A similar argument shows that the coefficient of $g_{n}h_{j_{\max}}$
in $rs$ is $\lambda_{n}\mu_{j_{\max}}\neq0$. Finally, since $n\geq2$
, we have $g_{1}h_{j_{\min}}<g_{n}h_{j_{\min}}\leq g_{n}h_{j_{\max}}$
and so $g_{1}h_{j_{\min}}$ and $g_{n}h_{j_{\max}}$ are distinct
elements of $\Gamma$ with nonzero coefficients in $rs=1$ -- a contradiction.
\end{proof}
The following theorem is a well-known result in the theory of one-relator
groups.
\begin{thm}
\label{thm:one-relator no torsion is orderable}If $1\ne w\in\F$
is a non-power then the one-relator group $\nicefrac{\F}{\ll w\gg}$
is right-orderable.
\end{thm}

\begin{proof}
As $w$ is a non-power, we deduce that $\nicefrac{\F}{\ll w\gg}$
is torsion-free by a theorem of Karass, Magnus and Solitar \cite[Thm.~1]{karrass1960elements}.
By a theorem proven independently by Brodskii \cite[Cor.~2.3]{brodskii1984equations}
and Howie \cite[Cor.~4.3]{howie1982locally}, every torsion-free one-relator
group has the property of being \emph{locally indicable, }which means
that each of its non-trivial finitely generated subgroups admits a
non-trivial homomorphism to $\mathbb{Z}$. Finally, the Burns-Hale
theorem \cite[Thm.~2]{burns1972note} states that a group $H$ is
right-orderable if and only if any non-trivial finitely generated
subgroup of $H$ admits a non-trivial homomorphism to some right-orderable
group. Since $\mathbb{Z}$ is right-orderable (the usual order on
$\mathbb{Z}$ is a right order), the Burns-Hale theorem implies that
every locally indicable group is right-orderable. The combination
of the theorems above gives that $\nicefrac{\F}{\ll w\gg}$ is right-orderable.
\end{proof}

\begin{proof}[Proof of Theorem \ref{thm:cyclic generators of A_w}]
 Let $1\ne w\in\F$ be a non-power and suppose that $\overline{f}\in\mathcal{\A}_{w}$
generates $\mathcal{A}_{w}$. Since $\rho$ is surjective, there exists
some $f\in\mathcal{A}$ such that $\rho\left(f\right)=\overline{f}$.
As $\overline{f}$ generates $\mathcal{A}_{w}$, there exists some
$s\in\mathcal{A}$ such that $\overline{f}s=\rho\left(1\right)$ or,
equivalently, $\rho\left(fs\right)=\rho\left(1\right)$. Applying
$\tau$ to both sides of the equation and using the fact that $\tau\circ\rho$
is a homomorphism of $\k$-algebras we obtain $\tau\rho\left(f\right)\cdot\tau\rho\left(s\right)=\tau\rho\left(1\right)$,
and in particular $\tau\rho\left(f\right)$ has a right inverse in
the quotient $\k$-algebra $\nicefrac{\mathcal{A}}{\left(\left(w-1\right)\right)}$.
Now, since $\tau\rho\left(f\right)$ has a right inverse in $\nicefrac{\mathcal{A}}{\left(\left(w-1\right)\right)}$,
its image under the isomorphism $\varphi$ from Lemma \ref{lem:iso of quotient module with qoutient-group algebra}
has a right inverse in $\k\left[\nicefrac{\F}{\ll w\gg}\right]$.
By Theorem \ref{thm:one-relator no torsion is orderable}, as $w$
is not a power, $\nicefrac{\F}{\ll w\gg}$ is right-orderable. Lemma
\ref{lem:right-order implies kaplansky unit} applied for $\Gamma=\nicefrac{\F}{\ll w\gg}$,
implies that $\varphi\left(\tau\rho\left(f\right)\right)=\lambda g$
for some $\lambda\in\k^{*}$ and $g\in\nicefrac{\F}{\ll w\gg}$.

Without loss of generality, by Lemma \ref{lem:cyclic generators of A_w are on cycle},
we may assume that $\overline{f}=\rho\left(f\right)$ is supported
on cosets of $\left\langle w\right\rangle $ belonging to the unique
simple cycle of $\sw$. The Weinbaum subword theorem \cite[Thm.~2]{weinbaum1972relators}
asserts that none of the non-trivial proper subwords of the cyclic
reduction of $w$ lies in its normal closure $\ll w\gg$. This implies
that two distinct vertices of the cycle of $\sw$ have distinct images
through $\tau$, namely, their images belong to different elements
of $\nicefrac{\F}{\ll w\gg}$. But the $\tau$-image of $\overline{f}$
is $\lambda g$, which is supported on a single element $g\in\k\left[\nicefrac{\F}{\ll w\gg}\right]$.
Thus $\overline{f}$ itself is supported on a single element of the
cycle of $\sw$ and can be lifted to an element $f\in\A$ supported
on a single element of $\F$.
\end{proof}

\section{Critical ideals of rank 2\label{sec:Critical-ideals-of-rank-2}}

Throughout this section fix a non-power $1\ne w\in\F$ and assume
without loss of generality that it is cyclically reduced. Theorems
\ref{thm:rational expression fix} and \ref{thm:limit powers fix}
yield that $\mathbb{E}_{w}\left[\fix\right]=2+\frac{c}{q^{N}}+O\left(\frac{1}{q^{2N}}\right)$
for some constant $c$. Our goal in the current section is to prove
Theorem \ref{thm:fixed vectors in pi=00003D2}: 
\[
c=\left|\crit_{q}^{2}\left(w\right)\right|,
\]
where $\crit_{q}^{2}\left(w\right)$ is the set of rank-2 ideals $I\le\A$
containing $w-1$ as an imprimitive element. 

Recall our formula \eqref{eq:E_w=00005Bfix=00005D rational expression}
for $\mathbb{E}_{w}\left[\fix\right]$ and Corollary \ref{cor:contrib is m-rank}.
It follows that the $\frac{1}{q^{N}}$-coefficient of $\mathbb{E}_{w}\left[\fix\right]$
consists of the contributions of the rank-1 and rank-2 ideals in the
set 
\[
{\cal I}\defi\left\{ I\le_{\left[1,w\right]}\A\,\middle|\,I\ni w-1\right\} .
\]
As $w\ne1$ and is a non-power, by Corollary \ref{cor:pa=00003D1 iff power}
the rank-1 ideals in ${\cal I}$ are precisely $\left(1\right)$ and
$\left(w-1\right)$. The contribution of $\left(1\right)$ to \eqref{eq:E_w=00005Bfix=00005D rational expression}
is precisely $1$, so it does not affect $c$. Denote by $\beta_{w}$
the coefficient of $\frac{1}{q^{N}}$ in the contribution of $\left(w-1\right)$,
namely, this contribution is $1+\frac{\beta_{w}}{q^{N}}+O\left(\frac{1}{q^{2N}}\right)$.
The summand in \eqref{eq:E_w=00005Bfix=00005D rational expression}
corresponding to a rank-2 ideal is $\frac{1}{q^{N}}+O\left(\frac{1}{q^{2N}}\right)$,
so such an ideal contributes exactly $1$ to $c$. Recall that all
the ideals in $\crit_{q}^{2}\left(w\right)$ are in ${\cal I}$, by
Corollary \ref{cor:finitely many critical extensions}. Denote by
$\mathrm{Prim}^{2}\left(w\right)$ the set of rank-2 ideals in ${\cal I}$
in which $w-1$ is primitive. With this notation, the coefficient
$c$ of $\frac{1}{q^{N}}$ in $\mathbb{E}_{w}\left[\fix\right]$ is
$c=\beta_{w}+\left|\mathrm{Prim}^{2}\left(w\right)\right|+\left|\crit_{q}^{2}\left(w\right)\right|$.
Our goal is, thus, to prove that
\[
\beta_{w}+\left|\mathrm{Prim}^{2}\left(w\right)\right|=0.
\]
Recall from Section \ref{sec:A_w} the quotient $\A$-module $\A_{w}\defi\A/\left(w-1\right)$,
the projection $\rho\colon\A\to\A_{w}$ and the Schreier graph $\sw=\left\langle w\right\rangle \backslash\mathrm{Cay}\left(\F,B\right)$.
The elements of $\A_{w}$ are $\k$-linear combinations of the vertices
of $\sw$, and we use $\rho$ to denote also the quotient in the graph
level $\rho\colon\mathrm{Cay}\left(\F,B\right)\to\sw$. Let $C_{w}=\rho\left(\left[1,w\right]\right)$
denote the unique simple cycle in $\sw$ (here we use the fact that
$w$ is assumed to be cyclically reduced).
\begin{lem}
\label{lem:beta_w as expression with 1-dim subspaces}The $\frac{1}{q^{N}}$-coefficient
of the summand corresponding to $I=\left(w-1\right)$ in \eqref{eq:E_w=00005Bfix=00005D rational expression}
is
\begin{equation}
\beta_{w}=-\frac{q^{v\left(C_{w}\right)}-1}{q-1}+\sum_{b\in B}\frac{q^{e_{b}\left(C_{w}\right)}-1}{q-1}.\label{eq:1/qN coefficient of (w-1) in T_w}
\end{equation}
\end{lem}

\begin{proof}
Recall that $\beta_{w}$ is the $\frac{1}{q^{N}}$-coefficient of
the Laurent expansion of 
\begin{equation}
\frac{\indep_{\left|w\right|+1-d^{\left[1,w\right]}\left(I\right)}\left(V_{N}\right)}{\prod_{b\in B}\indep_{e_{b}\left(w\right)-d_{b}^{\left[1,w\right]}\left(I\right)}\left(V_{N}\right)},\label{eq:contrib of (w-1) no. 1}
\end{equation}
for $I=\left(w-1\right)$. Because $f\in\A|_{\left[1,w\right]}$ belongs
to $I=\left(w-1\right)$ if and only if its coefficients in every
fiber over $C_{w}$ sum up to zero, the dimension over $\k$ of $I|_{\left[1,w\right]}$
is precisely $d^{\left[1,w\right]}\left(I\right)=v\left(\left[1,w\right]\right)-v\left(C_{w}\right)=1$
. Similarly, the dimension over $\k$ of $I|_{D_{b}\left(\left[1,w\right]\right)}$
is precisely $d_{b}^{\left[1,w\right]}\left(I\right)=e_{b}\left(\left[1,w\right]\right)-e_{b}\left(C_{w}\right)=0$.
Hence, \eqref{eq:contrib of (w-1) no. 1} is equal to 
\begin{equation}
\frac{\indep_{v\left(C_{w}\right)}\left(V_{N}\right)}{\prod_{b\in B}\indep_{e_{b}\left(C_{w}\right)}\left(V_{N}\right)}=\frac{\left(q^{N}-1\right)\left(q^{N}-q\right)\cdots\left(q^{N}-q^{v\left(C_{w}\right)-1}\right)}{\prod_{b\in B}\left(q^{N}-1\right)\left(q^{N}-q\right)\cdots\left(q^{N}-q^{e_{b}\left(C_{w}\right)-1}\right)}.\label{eq:contrib of (w-1) No. 2}
\end{equation}
Because $C_{w}$ is a cycle, the number of vertices is identical to
the total number of edges. Hence \eqref{eq:contrib of (w-1) No. 2}
is equal to 
\[
\frac{\left(1-\frac{1}{q^{N}}\right)\left(1-\frac{q}{q^{N}}\right)\cdots\left(1-\frac{q^{v\left(C_{w}\right)-1}}{q^{N}}\right)}{\prod_{b\in B}\left(1-\frac{1}{q^{N}}\right)\left(1-\frac{q}{q^{N}}\right)\cdots\left(1-\frac{q^{e_{b}\left(C_{w}\right)-1}}{q^{N}}\right)},
\]
and the $\frac{1}{q^{N}}$-coefficient of the Laurent expansion of
this expression is
\[
\left[-1-q-\ldots-q^{v\left(C_{w}\right)-1}\right]-\sum_{b\in B}\left[-1-q-\ldots-q^{e_{b}\left(C_{w}\right)-1}\right],
\]
which is equal to \eqref{eq:1/qN coefficient of (w-1) in T_w}.
\end{proof}
Denote by $D_{w}$ the set of proper non-trivial cyclic submodules\footnote{Recall that a cyclic submodule is a submodule generated by a single
element.} of $\A_{w}$ generated by some element supported on the cycle $C_{w}$:
\[
D_{w}\defi\left\{ g\A\lvertneqq\A_{w}\,\middle|\,0\ne g\in\A_{w}~\mathrm{and}~g~\mathrm{supported~on}~C_{w}\right\} .
\]

\begin{lem}
\label{lem:couing rank-2 ideals by quotient A/(w-1)}There is a one-to-one
correspondence between $D_{w}$ and $\mathrm{Prim}^{2}\left(w\right)$.
\end{lem}

\begin{proof}
By Corollary \ref{cor:complement to a basis of w-1}, the rank-2 ideals
$I\le_{\left[1,w\right]}\A$ containing $w-1$ as a primitive element
are exactly the rank-2 ideals of the form $\left(w-1,f\right)$ with
$f$ supported on $\left[1,w\right]$ (here we use again the fact
that every pair of generators of a rank-$2$ ideal is a basis --
\cite[Prop.~2.2]{cohn1964free}). Now $g\in\A_{w}$ is supported on
$C_{w}$ if and only if there is some $f\in\A$ supported on $\left[1,w\right]$
with $\rho\left(f\right)=g$. Note that
\[
\left(w-1,f\right)=\rho^{-1}\left(\rho\left(f\right)\A\right),
\]
where $\rho\left(f\right)\A$ is the submodule of $\A_{w}$ generated
by the image of $f$ in $A_{w}$. Because the only rank-1 ideals containing
$w-1$ are $\left(1\right)$ and $\left(w-1\right)$, we have that
$\left(w-1,f\right)$ is of rank 2 if and only if $\rho\left(f\right)\A$
is a non-zero proper submodule of $\A_{w}$.
\end{proof}
Next, we study the different elements $g\in\A_{w}$ supported on $C_{w}$.
In order to understand when two different elements $g,g'$ generate
the same submodule, we construct a graph $\Upsilon$. The vertices
of $\Upsilon$ are the $1$-dimensional linear subspaces of $\k^{\mathrm{vert}\left(C_{w}\right)}$,
so their number is $v\left(\Upsilon\right)=\frac{q^{v\left(C_{w}\right)}-1}{q-1}$.
For every $b\in B$ and every $1$-dimensional subspace $U\le\k^{b\textnormal{-}\mathrm{edges}\left(C_{w}\right)}$
(here $K^{b\textnormal{-}\mathrm{edges}}$ is the space of $K$-linear
combinations of the $b$-edges in $C_{w}$), the subspace $U$ corresponds
to a 1-dimensional subspace $o\left(U\right)$ of the vertices supported
on the origins of the $b$-edges, as well as a 1-dimensional subspace
$t\left(U\right)$ supported on the termini of the $b$-edges. For
every such $U$ we draw a directed $b$-edge from the vertex $o\left(U\right)$
to the vertex $t\left(U\right)$ in $\Upsilon$. Note that $e\left(\Upsilon\right)=\sum_{b\in B}\frac{q^{e_{b}\left(C_{w}\right)}-1}{q-1}$,
so overall
\begin{equation}
\chi\left(\Upsilon\right)\defi v\left(\Upsilon\right)-e\left(\Upsilon\right)\stackrel{Lemma~\ref{lem:beta_w as expression with 1-dim subspaces}}{=}-\beta_{w}.\label{eq:EC of Upsilon is -beta}
\end{equation}
Denote by $\C\left(\Upsilon\right)$ the connected components of $\Upsilon$.
Because $g\A=g.b\A$ for every $g\in\A_{w}$ and $b\in B$, there
is a well-defined surjective map 
\[
\Phi\colon\C\left(\Upsilon\right)\twoheadrightarrow\left\{ g\A\,\middle|\,0\ne g\in\A_{w}~\mathrm{supported~on}~C_{w}\right\} =D_{w}\cup\left\{ \A_{w}\right\} .
\]
One of the connected components of $\Upsilon$ is isomorphic to $C_{w}$:
this is the component consisting of vertices and edges of $\Upsilon$
corresponding to $1$-dimensional subspaces supported on a single
vertex or on a single edge. Denote this component by $C_{0}$. Clearly,
$\Phi\left(C_{0}\right)=\A_{w}$.
\begin{lem}
\label{lem:All-connected-components-but_one-are-trees}All connected
components of $\Upsilon$ except for $C_{0}$ are paths.
\end{lem}

\begin{proof}
The degree of every vertex in $\Upsilon$ is at most $2$, so every
connected component is a path or a cycle. Assume that some component
$C_{0}\ne C\in\C\left(\Upsilon\right)$ is a cycle. Let $U$ be a
vertex in $C$, and assume that this cycle reads the (cyclically reduced)
word $z\in\F$ starting (and ending) at $U$. Recall that $U$ is
a 1-dimensional subspace of $\k^{\mathrm{vert}\left(C_{w}\right)}$
supported on at least two vertices of $C_{w}$, and denote the support
of $U$ by $\supp\left(U\right)$, so $\left|\supp\left(U\right)\right|\ge2$.
In particular, for every $s\in\supp\left(U\right)$, there is a path
in $C_{w}$ reading $z$ leaving $s$ and reaching some $s'\in\supp\left(U\right)$.
Hence, some power $z^{k}$ of $z$ is a path from $s$ to itself for
every $s\in\supp\left(U\right)$. Because $w$ is not conjugate to
$w^{-1}$, every such copy of $z^{k}$ has the same orientation along
$c_{w}$. We get that there is some $y\in\F\setminus\left\langle w\right\rangle $
so that $ywy^{-1}=w$. This is not possible unless $w$ is a proper
power, which is not the case.
\end{proof}
\begin{lem}
\label{lem:every proper ideal corresponds to a single component}The
map $\Phi\colon\C\left(\Upsilon\right)\to D_{w}\cup\left\{ \A_{w}\right\} $
is one-to-one.
\end{lem}

\begin{proof}
Theorem \ref{thm:cyclic generators of A_w} states the only cyclic
generators of $\A_{w}$ are elements supported on a single vertex
of $\sw$, and so $C_{0}$ is the only connected component in $\Upsilon$
mapped to $\A_{w}$. It remains to show that every element of $D_{w}$
has a single preimage in $\C\left(\Upsilon\right)$. Suppose that
$g,g'\in\A_{w}$, both supported on $C_{w}$, so that $g'\A=g\A\in D_{w}$,
namely, $\left\{ 0\right\} \ne g\A=g'\A\lvertneqq\A_{w}$. Let $f,f'\in\A$
be preimages of $g,g'$, respectively, through $\rho^{-1}$, which
are supported on $\left[1,w\right]$. Then $\left(f,w-1\right)=\left(f',w-1\right)$
is a rank-2 ideal by Lemma \ref{lem:couing rank-2 ideals by quotient A/(w-1)}.
Thus there are $p_{1},p_{2},q_{1},q_{2}\in\A$ such that 
\begin{eqnarray*}
f' & = & fp_{1}+\left(w-1\right)p_{2}\\
f & = & f'q_{1}+\left(w-1\right)q_{2},
\end{eqnarray*}
so
\begin{eqnarray*}
f & = & \left(fp_{1}+\left(w-1\right)p_{2}\right)q_{1}+\left(w-1\right)q_{2}=f\cdot p_{1}q_{1}+\left(w-1\right)\left(p_{2}q_{1}+q_{2}\right).
\end{eqnarray*}
But $\left\{ f,w-1\right\} $ is a basis, so by uniqueness we get
$p_{1}q_{1}=1$ (and $p_{2}q_{1}+q_{2}=0$). The only units of $\A$
are scalar product of monomials of the form $\alpha z$ with $\alpha\in\k^{*}$
and $z\in\F$ (this was mentioned and explained in Section \ref{subsec:Cyclic-generators-of-A_w}).
By multiplying $g'$ by a scalar if necessary, we may thus assume
that $p_{1}\in\F$ is a word, and we get that 
\[
g'=\rho\left(f'\right)=\rho\left(fp_{1}+\left(w-1\right)p_{2}\right)=\rho\left(fp_{1}\right)=\rho\left(f\right)p_{1}=gp_{1}.
\]
But $C_{w}$ contains every reduced path between every two of its
vertices, so inside $\Upsilon$ there is a path (reading $p_{1}$)
from the vertex corresponding to $g$ to the one corresponding to
$g'$. In particular, they both belong to the same connected component.
\end{proof}

\begin{proof}[Completing the proof of Theorem \ref{thm:fixed vectors in pi=00003D2}]
Recall that we need to show that $\beta_{w}+\left|\mathrm{Prim}^{2}\left(w\right)\right|=0$.
Consider the above-mentioned map $\Phi\colon\C\left(\Upsilon\right)\to D_{w}\cup\left\{ \A_{w}\right\} $.
As $C_{0}$ is a cycle isomorphic to $C_{w}$, we have $\chi\left(C_{0}\right)=0$.
By Lemma \ref{lem:All-connected-components-but_one-are-trees}, $\chi\left(C\right)=1$
for any $C_{0}\ne C\in\C\left(\Upsilon\right)$, so $\left|\C\left(\Upsilon\right)\setminus\left\{ C_{0}\right\} \right|=\chi\left(\Upsilon\right)$.
Thus
\[
\left|\mathrm{Prim}^{2}\left(w\right)\right|\stackrel{\mathrm{Lemma}~\ref{lem:couing rank-2 ideals by quotient A/(w-1)}}{=}\left|D_{w}\right|\stackrel{\mathrm{Lemma}~\ref{lem:every proper ideal corresponds to a single component}}{=}\left|\C\left(\Upsilon\right)\setminus\left\{ C_{0}\right\} \right|=\chi\left(\Upsilon\right)\stackrel{\eqref{eq:EC of Upsilon is -beta}}{=}-\beta_{w}.
\]
\end{proof}

\section{Open Questions\label{sec:Open-Questions}}

This paper raises quite a few questions and directions for future
research, and we gather the main ones here. As above, $\A=\k\left[\F\right]$
and $\pi_{q}\left(w\right)$ is the $q$-primitivity rank of $w\in\F$
(see Definition \ref{def:q-primitivity-rank}).

\paragraph{Expected number of fixed vectors}

As stated in Conjecture \ref{conj:general pi and fixed vectors},
is it true that for every $w\in\F$, we have $\mathbb{E}_{w}\left[\fix\right]=2+\frac{\left|\crit_{q}\left(w\right)\right|}{q^{N\left(\pi-1\right)}}+O\left(\frac{1}{q^{N\pi}}\right)$
where $\pi=\pi_{q}\left(w\right)$? If true, this would generalize
Corollary \ref{cor:F2} and yield that in free groups of arbitrary
finite rank the words inducing the uniform measure on $\gln$ for
every $N$ are precisely the primitive words -- a result analogous
to \cite[Thm 1.1]{PP15} dealing with $S_{N}$. 

\paragraph{The $q$-primitivity rank}

Recall Conjecture \ref{conj:pi and pi_q}: is it true that $\pi_{q}\left(w\right)=\pi\left(w\right)$
for every $w\in\F$ and every prime power $q$? What is the value
for a generic word (compare with \cite[Cor.~8.3]{Puder2015} and \cite{kapovich2022primitivity})?
Moreover, the Cohn-Lewin theorem applies to the free group algebra
over an arbitrary field, not necessarily finite, and one can analogously
define the $\k$-primitivity rank of $w$ for an arbitrary field $\k$
(and even for certain rings). Is it true that the $\k$-primitivity
rank is equal to $\pi\left(w\right)$ for every field $\k$? 

What about general elements of $\A$? One can define the primitivity
rank $\pi_{\A}\left(f\right)$ of arbitrary $f\in\A$ as the rank
of critical ideals, so $\pi_{q}\left(w\right)=\pi_{\A}\left(w-1\right)$
(and see the paragraph preceding Corollary \ref{cor:finitely many critical extensions}).
What are the possible values of $\pi_{\A}\left(f\right)$ for $f\in\A$?
Does this number have any combinatorial meaning (à la Conjecture \ref{conj:general pi and fixed vectors})?

\paragraph{The expected value of stable irreducible characters}

Recall Conjecture \ref{conj:irreducible chars} which says that for
every stable irreducible character $\chi$ of $\gl_{\bullet}\left(\k\right)$,
$\mathbb{E}_{w}\left[\chi\right]=O\left(\left(\dim\chi\right)^{1-\pi_{q}\left(w\right)}\right)$.
This conjecture should be quite difficult to tackle, as it is not
even known in the somewhat simpler case of the symmetric group. It
is more conceivable that one may be able to prove the weaker result
that $\mathbb{E}_{w}\left[\chi\right]=O\left(q^{-N\cdot\pi_{q}\left(w\right)}\right)$
for every non-power $w$ and every stable irreducible character of
dimension $\Omega\left(q^{2N}\right)$. This kind of result was proved
for stable irreducible characters of $\left\{ S_{N}\right\} _{N}$
\cite[Cor.~1.7]{hanany2020word}, for $\left\{ U\left(N\right)\right\} _{N}$
\cite{brodsky2021unitary} and for $\left\{ G\wr S_{N}\right\} _{N}$
for any finite group $G$ \cite{shomroni2023wreathII}. See also \cite[Appendix A]{stable2023}
for further discussion and a more refined conjecture.

\paragraph{Spectral gap in random Schreier graphs of $\protect\gln$}

Part of the original motivation for studying word measures on $\gln$
lies in questions regarding expansion and spectral gaps in random
Schreier graphs of the groups $\gln$ when $\k$ is fixed and $N\to\infty$.
A recent milestone here is \cite{eberhard2021babai}. Still, the following
question is still open: Consider a random Schreier graph depicting
the linear action of $\gln$ on $\k^{N}\setminus\left\{ 0\right\} $
with respect to two random generators. Do these graphs admit a uniform
spectral gap with probability $\to1$ as $N\to\infty$? If so, is
the spectral gap optimal? It is plausible that the results and conjectures
in this paper may contribute to obtaining such results, in a fashion
similar to analogous proofs for Schreier graphs of $S_{N}$ \cite{Linial2010,Puder2015,friedman2020note,hanany2020word}.

\paragraph{Limit distributions}

Theorem \ref{thm:limit distr} states that for $w$ a non-power, the
distribution of the number of fixed vectors in a $w$-random element
of $\gln$ converges in distribution, as $N\to\infty$, to a limit
distribution which is independent of $w$. Is this true for powers
too? Is this true for an arbitrary stable class function in the ring
$\R$ from page \pageref{page:R ring of stable class functions}?
(This is known for $S_{N}$ -- see \cite[Thm.~1.1]{nica1994number}
and \cite[Thm.~1.14]{puder2022local} for a more general result about
cycles of bounded length.)

\paragraph{Free group algebras}

This paper gives rise to quite a few questions about the free group
algebra $\A$. First, it is natural to guess that Corollary \ref{cor:complement to a basis of w-1}
can be generalized as follows: if $T\subseteq\mathrm{Cay}\left(\F,B\right)$
is a subtree and $f$, supported on $T$, is a primitive element of
$I\le_{T}\A$, can $\left\{ f\right\} $ be extended to a basis of
$I$ which is supported on $T$? 

Recall Theorem \ref{thm:cyclic generators of A_w} that when $w$
is a non-power, the only cyclic generators of the right $\A$-module
$\A/\left(w-1\right)$ are images of unit elements of $\A$. Is this
true for general subgroups of $\F$? Namely, let $H\le\F$ be a finitely
generated subgroup which is not contained in any other subgroup of
equal or smaller rank (in the language of \cite{Puder2014}, this
is $\pi\left(H\right)>\rk H$). Let $J_{H}\defi I_{H}\A=\left(\left\{ h-1\,\middle|\,h\in H\right\} \right)$
(see \cite[Chap.~4]{cohen1972groups}). Is it true that the only cyclic
generators of the quotient $\A$-module $\A/J_{H}$ are images of
unit elements of $\A$? This would be a Kaplansky-type result for
such modules. 

There are many other famous theorems and algorithms about free groups
and their subgroups and we wonder if they have versions that apply
to the free group algebra and its ideals. For example, is there an
analogue of Whitehead's cut vertex criterion which may detect efficiently
whether a given element belongs to a free factor of a given ideal?
See the recent survey \cite{delgado2021list} giving a list of results
about free groups and their subgroups using Stalling core graphs.

\section*{Appendix}

\begin{appendices}

\section{The limit distribution of $\protect\fix$ \label{sec:Matan}}

Fix a non-power $1\ne w\in\F$. Recall that $\fix_{w,N}$ denotes
the number of fixed vectors of a $w$-random element in $\gln$. In
this appendix we explain why the method of moments is applicable for
proving convergence in distribution for $\fix_{w,N}$, thus proving
Theorem \ref{thm:limit distr}. We begin by recalling some basic definitions
for the moment problem.

Given a sequence of real numbers $\left(m_{n}\right)_{n\geq0}$ and
an interval $I\subseteq\mathbb{R}$, a solution to the associated
moment problem is a positive Borel measure $\theta$ supported on
$I$ with moments $\intop_{I}x^{n}d\theta\left(x\right)=m_{n}$. When
$I=\mathbb{R}$ (respectively, $I=[0,\infty)$), the problem is called
a \emph{Hamburger} (respectively, \emph{Stieltjes}) moment problem.
If a solution exists, the moment problem is said to be \emph{solvable}.
A solvable moment problem is further categorized by the number of
solutions: if a unique solution exists, the moment problem is said
to be \emph{determinate} and otherwise it is called \emph{indeterminate},
in which case there are infinitely many solutions since the set of
solutions is convex.

The limiting measure of $\fix_{w,N}$ is a special case of a well-studied
family of measures in the field of orthogonal polynomials. We next
recall this family of measures, and then explain how previous analysis
of the determinacy of its associated moment problems allows us to
deduce the desired convergence in distribution.

Let $p\in\left(0,1\right)$ and $a>0$. The \emph{Al-Salam Carlitz}
polynomials of the second kind $V_{n}^{\left(a\right)}\left(x;p\right)$
(see \cite[pp.~195-198]{chihara1978introduction}, \cite[Sect.~14.24]{koekoek2010hypergeometric},
\cite[pp.~30-33]{christiansen2004indeterminate}) are orthogonal with
respect to the probability measure supported on the sequence $\left\{ p^{-k}\right\} _{k\geq0}$
with masses 
\begin{equation}
w_{AC}\left(p^{-k};a;p\right)=\left(ap;p\right)_{\infty}\frac{a^{k}p^{k^{2}}}{\left(p;p\right)_{k}\left(ap;p\right)_{k}},\label{eq: AC weight}
\end{equation}
where $\left(x;y\right)_{n}=\prod_{j=0}^{n-1}\left(1-xy^{j}\right)$
is the $q$-shifted factorial, or $q$-Pochhammer symbol, and $\left(x;y\right)_{\infty}=\prod_{j=0}^{\infty}\left(1-xy^{j}\right)$.

Let $q$ be a prime power. The limiting measure $\nu$ of $\fix_{w,N}$
is a special case of the family of measures \eqref{eq: AC weight}
with parameters $p=q^{-1}$ and $a=1$. Explicitly, $\nu=\sum_{k=0}^{\infty}w_{AC}\left(q^{k};1;q^{-1}\right)\delta_{q^{k}}$.
The $n$-th moment of $\nu$ is equal to the number of linear subspaces
of an $n$-dimensional vector space over a field with $q$ elements
(see \cite[Prop.~5.7]{fulman2016distribution} or \cite[Eq.~10.10]{chihara1978introduction}).

Let $\nu'$ be the pushforward of $\nu$ under the translation map
$x\mapsto x-1$, i.e.,
\[
\nu'=\sum_{k=0}^{\infty}w_{AC}\left(q^{k};1;q^{-1}\right)\delta_{q^{k}-1}.
\]
The measure $\nu'$ exhibits the interesting phenomenon of having
its Hamburger moment problem be indeterminate while its Stieltjes
moment problem is determinate (see \cite[Sect.~4]{berg1994nevanlinna}).
Since the moments of a random variable $Z$ determine the moments
of $Z-1$ and vice versa, the pushforward map induced by $x\mapsto x-1$
forms a bijection between solutions to the moment problem associated
to $\nu$ on $I=[1,\infty)$ and solutions to the Stieltjes moment
problem associated to $\nu'$. In particular, any measure supported
on $[1,\infty)$ with the same moments as $\nu$ must be equal to
$\nu$.

Let $\nu_{n}$ be a sequence of Borel probability measures on $\mathbb{R}$
supported on $[1,\infty)$, and suppose that for every $k\in\mathbb{N}$
the $k$-th moment of $\nu_{n}$ converges as $k\rightarrow\infty$
to the $k$-th moment of $\nu$, as Theorem \ref{thm:limit powers general}
applied with $\B=I_{k}\in\gl_{k}\left(\k\right)$ yields for $\fix_{w,N}$.
We are now ready to deduce that $\nu_{n}$ converges weakly\footnote{The convergence in distribution of random variables is equivalent
to the weak convergence of their measures.} to $\nu$. The set of Borel probability measures on $\mathbb{R}$
equipped with the topology of weak convergence is metrizable (the
Lévy metric, for example; see \cite[Exer.~3.2.6]{durrett_2019}),
and so it is enough to show that every subsequence of $\nu_{n}$ has
a further subsequence converging weakly to $\nu$. Let $\nu_{n_{k}}$
be such a subsequence. The convergence of the second moments implies
that the sequence $\left(\nu_{n}\right)_{n\in\mathbb{N}}$ is tight
\cite[Thm.~3.2.14]{durrett_2019}, and by Prokhorov's Theorem \cite[Thm.~5.1]{billing},
$\nu_{n_{k}}$ has a further subsequence $\nu_{n_{k_{l}}}$ converging
weakly to some probability measure $\widetilde{\nu}$. The convergence
of moments of $\nu_{n}$ to the moments of $\nu$ implies that $\widetilde{\nu}$
has the same sequence of moments as $\nu$ (\cite[Exer.~3.2.5]{durrett_2019}).
Furthermore, using the Portmanteau Theorem (\cite[Thm.~3.2.11]{durrett_2019})
on the closed set $[1,\infty)\subseteq\mathbb{R}$, we get
\[
\widetilde{\nu}\left([1,\infty)\right)\geq\limsup_{l\rightarrow\infty}\nu_{n_{k_{l}}}\left([1,\infty)\right)=\limsup_{l\rightarrow\infty}1=1,
\]
and so $\widetilde{\nu}$ must also be supported on $[1,\infty)$.
The determinacy of the moment problem associated to $\nu$ on $I=[1,\infty)$
implies that $\widetilde{\nu}=\nu$, finishing the argument.

\end{appendices}

\bibliographystyle{alpha}
\bibliography{word_measures_on_GL}

\noindent Danielle Ernst-West, School of Mathematical Sciences, Tel
Aviv University, Tel Aviv, 6997801, Israel\\
\texttt{daniellewest@mail.tau.ac.il }~\\

\noindent Doron Puder, School of Mathematical Sciences, Tel Aviv University,
Tel Aviv, 6997801, Israel\\
\texttt{doronpuder@gmail.com}~\\

\noindent Matan Seidel, School of Mathematical Sciences, Tel Aviv
University, Tel Aviv, 6997801, Israel\\
\texttt{matanseidel@gmail.com}~\\

\end{document}